%% file: HilderSharma23_Updated.tex
\documentclass[10pt]{article}
\pdfoutput=1
\usepackage{amsmath,amssymb,amsfonts,amsthm} % Typical maths resource packages
\usepackage{mathtools}
\usepackage{stmaryrd} %Widerspruchspfeil
\usepackage{color}
\usepackage{calc}
\usepackage[version=3,arrows=pgf-filled]{mhchem}
\usepackage{csquotes}
\usepackage{colortbl}
\usepackage{extarrows}
\usepackage[ruled]{algorithm2e} %for algorithms
\usepackage{mathrsfs} %fancy caligraphy
\usepackage{import} % different paths for input calls
\usepackage{enumitem}
\usepackage{appendix}
\usepackage{fullpage}
\usepackage{pgfplots}
\usepackage{thmtools}   % for restatable
\usepackage{dsfont} % for \mathds{1}
\allowdisplaybreaks % allow equation break across pages

\usepackage{hyperref}
\hypersetup{
colorlinks   = true, %Colours links instead of boxes
urlcolor     = black, %Colour for external hyperlinks
linkcolor    = blue, %Colour of internal links
citecolor   = blue %Colour of citations, could be ``red''
}
\usepackage{colonequals} %to use the := symbol

\usepackage{soul} % strikethrough \st{}

\usepackage{dsfont}

\mathchardef\mhyphen="2D 

\theoremstyle{definition}
\newtheorem{defi}{Definition}[section]

\theoremstyle{remark}
\newtheorem{rem}[defi]{Remark}
\usepackage{etoolbox} % Required for the  & ext command for adding QED after remarks
\AtEndEnvironment{rem}{\qed}
\AtEndEnvironment{example}{\qed}

\theoremstyle{plain}
\newtheorem{theorem}[defi]{Theorem}
\newtheorem{lem}[defi]{Lemma}
\newtheorem{cor}[defi]{Corollary}
\newtheorem{prop}[defi]{Proposition}

\newcommand{\N}{\ensuremath{\mathbb{N}}}
\newcommand{\R}{\ensuremath{\mathbb{R}}}

\newcommand{\e}{\varepsilon}

\newlength{\leftstackrelawd}
\newlength{\leftstackrelbwd}
\def\leftstackrel#1#2{\settowidth{\leftstackrelawd}%
{${{}^{#1}}$}\settowidth{\leftstackrelbwd}{$#2$}%
\addtolength{\leftstackrelawd}{-\leftstackrelbwd}%
\leavevmode\ifthenelse{\lengthtest{\leftstackrelawd>0pt}}%
{\kern-.5\leftstackrelawd}{}\mathrel{\mathop{#2}\limits^{#1}}}

\def\RelEnt{\mathscr{H}}
\def\RF{\mathscr{R}}

\newcommand{\curlI}{\mathcal{I}}

\def\X{\mathcal{X}}
\def\Y{\mathcal{Y}}
\def\Z{\mathcal{Z}}
\def\P{\mathcal{P}}
\DeclareMathOperator*{\sumsum}{\sum\!\sum} %double sum with less spacing

\DeclareMathOperator\LSI{LSI}
\DeclareMathOperator\TV{TV}
\DeclareMathOperator\av{av}
\def\Wasser{\mathcal{W}}
\newcommand{\Ric}{\operatorname{Ric}}

\def\cg{\hat\mu} %coarse-grained dynamics
\def\eff{\eta} % effective dynamics
\def\aver{\mu^{\av}}
\def\stat{\rho} % invariant measure
\def\eps{\varepsilon} % epsilon
\def\sG{G} %slow-generator to replace C

\newcommand{\comment}[1]{} % remove details
\pgfplotsset{compat=1.18}

\usepackage{xcolor}

\title{Quantitative coarse-graining of Markov chains}
\author{Bastian Hilder\thanks{Centre for Mathematical Sciences, Lund University, PO Box 118, 221 00 Lund, Sweden.\ Email:\href{mailto:bastian.hilder@math.lu.se}{bastian.hilder@math.lu.se}}, Upanshu Sharma\thanks{Fachbereich Mathematik und Informatik, Freie Universit\"at Berlin, Arnimallee 9, 14195 Berlin, Germany.\ Email:\href{mailto:upanshu.sharma@fu-berlin.de}{upanshu.sharma@fu-berlin.de}}}
\date{\today}

\begin{document}
\maketitle

\begin{abstract}
Coarse-graining techniques play a central role in reducing the complexity of stochastic models, and are typically characterised by a mapping which projects the full state of the system onto a smaller set of variables which captures the essential features of the system. 
Starting with a continuous-time Markov chain, in this work we propose and analyse an \emph{effective dynamics}, which approximates the dynamical information in the coarse-grained chain. Without assuming explicit scale-separation, we provide sufficient conditions under which this effective dynamics stays close to the original system and provide quantitative bounds on the approximation error.  We also compare the effective dynamics and corresponding error bounds to the averaging literature on Markov chains which involve explicit scale-separation. We demonstrate our findings on an illustrative test example. 
\end{abstract}

\section{Introduction}

Modelling of complex systems often leads to stochastic models with a wide range of spatial and/or temporal scales. Examples include molecular dynamics~\cite{AllenTildesley17}, chemical kinetics of biophysical systems~\cite{Kitano2001} and climate modelling~\cite{Arnold01}, just to name a few. Coarse-graining is an umbrella term for techniques used to approximate such large and complex systems by simpler and lower dimensional ones. Such an approximation has obvious utility from a computational perspective since a simulation of the full system is often infeasible due to the presence of multiple scales. These techniques are also important from a modelling and analytical viewpoint, since quantities of interest are often described by a small class of variables in the system.  In most situations, the quantities of interest are the slow degrees of freedom that contain information about the long-term dynamics, while the fast-scales are considered irrelevant when analysing long-term behaviour. 

Various relevant phenomenon, such as Markov state models~\cite{HusicPande18} in molecular dynamics, stochastic chemical kinetics~\cite{Gillespie07} for a chemical-reaction system and agent-based models~\cite{CastellanoFortunatoLoreto09} in social dynamics, are modelled by jump processes. Coarse-graining of such models has received considerable attention from practioners in recent years -- for instance to perform a data-driven identification of the transition-rate matrix for a reduced set of stable configurations in a molecular system~\cite{KubeWeber07,PyEmma15} or to deduce  subnetwork dynamics from large-scale biochemical systems~\cite{BraviSollich17}.  From a mathematical perspective, coarse-graining of jump processes is limited to systems with explicit scale-separation, i.e.\ presence of fast and slow scales. In this setup, as the ratio of fast to slow increases (typically characterised by a small parameter), the fast scales equilibrate with respect to the slow ones and consequently some form of averaging allows one to obtain a closed coarse-grained model for the slow variables~\cite{PavliotisStuart08,lahbabiLegoll13,zhang16,HilderPeletierSharmaTse20,MielkeStephan20,PeletierRenger20}. 

As evidenced by the references above, while coarse-graining of jump processes is extremely relevant from a practical viewpoint, its mathematical analysis has been limited to systems with explicit scale-separation. 
Inspired by the related ideas of \emph{conditional-expectation} closures by Legoll and Leli\`evre~\cite{legollLelievre10} for diffusion processes in molecular dynamics and \emph{optimal-prediction} closures by Chorin~\cite{Chorin2003}, in this work we propose a natural approximation for a coarse-grained jump process, called the \emph{effective dynamics} (following the work of~\cite{legollLelievre10}). The main aim of this work is to understand if and when this effective dynamics is indeed a \emph{good} lower-dimensional approximation of the full stochastic jump process. As a first systematic study of effective dynamics for jump processes, in this article we will restrict ourselves to coarse-graining of linear Markov jump processes on a finite state-space. Using entropy techniques and functional inequalities, in the first part of the article we present quantitative error estimates on the coarse-graining error in the absence of explicit scale-separation. In the second part of the article we analyse these estimates in the presence of explicit scale-separation and provide illustrative numerical experiments.

\subsection{Coarse-graining and effective dynamics}

In this work we will focus on Markov jump processes with linear jump rates on a finite state space $\X$ with an irreducible generator $L\in\R^{|\X|\times |\X|}$. The law of this process $t\mapsto \mu_t \in\P(\X)$ solves the forward Kolmogorov equation
\begin{equation}\label{eq:intro-ForKol}
\begin{aligned}
\partial_t\mu = L^T\mu,\\
\mu|_{t=0}=\mu_0,
\end{aligned}
\end{equation}
with initial data $\mu_0\in \P(\X)$. Here $\P(\X)$ is the space of probability measures on $\X$.  Since $L$ is a generator, it satisfies $L(x,x')\geq 0$ for all $x\neq x'$ in $\X$, with $\sum_{x'\in\X} L(x,x')=0$ for any $x\in\X$. Since $L$ is irreducible, the evolution~\eqref{eq:intro-ForKol} admits a positive stationary measure $\stat\in\P_+(\X)$, i.e.\ $L^T\stat=0$. Here $\P_+(\X)$ is the space of positive probability measure (i.e.\ with positive coordinates) on $\X$. Throughout this article we will take the time-marginal viewpoint by studying the forward Kolmogorov equations rather than working with the underlying jump processes.

As stated above, this work is inspired by molecular dynamics and chemical kinetics, where the states in $\X$ correspond to the stable configurations of a molecular system and species concentration respectively. In both these settings it is often crucial to derive lower-dimensional approximations of the reference  dynamics~\eqref{eq:intro-ForKol} so as to either reduce the computational complexity of the system under consideration or to focus on certain key features of the system, for instance the transition between two stable states~\cite{PyEmma15} or a particular subnetwork of reactions~\cite{BraviSollich17}. 

Such a lower-dimensional approximation is typically made possible by means of a so called \emph{coarse-graining map} 
\begin{equation*}
\xi:\X\rightarrow\Y,  \quad x\mapsto \xi(x) = y,
\end{equation*}
which characterises a projection onto a smaller subset of states, i.e. $\Y\subseteq \X$ with $|\Y|\leq |\X|$. Although often there is no explicit scale-separation present, the reduced set of states encoded in $y\in\Y$  typically characterise the `slow' behaviour of the system. The central aim of this work is to propose and study a \emph{Markovian} $\xi$-projection of the dynamics~\eqref{eq:intro-ForKol}. In the rest of the article we will use notation consistent with earlier works on coarse-graining of Markov chains~\cite{HilderPeletierSharmaTse20} and diffusion processes~\cite{legollLelievre10}. One could also make use of matrix notations (see Remark~\ref{rem:MatrixNotation} for details), but we avoid it in this article.

Using the notation $\Lambda_y=\{x\in\X:\xi(x)=y\}$ for the level-sets of $\xi$, in what comes later (see Section~\ref{subsec:CG-eff}), we will show  that the $\xi$-projection of~\eqref{eq:intro-ForKol}, i.e.\ the evolution of $\cg_t \coloneqq \xi_\#\mu_t \in \P(\Y)$ (push-forward of $\mu_t$ under $\xi$) is given by 
\begin{equation}\label{eq:intro-CG}
\partial_t\cg_t=\hat L_t^T\cg_t,
\end{equation}
with a time-dependent generator $t\mapsto \hat L_t \in \R^{|\Y|\times|\Y|}$ 
\begin{equation*}
\hat L_t(y_1,y_2)\coloneqq \sumsum_{x_1\in\Lambda_{y_1},x_2\in\Lambda_{y_2} } L(x_1,x_2)\mu_t(x_1|y_1),
\end{equation*}
where $\mu_t(\cdot|y)\in\P(\Lambda_y)$ is the conditional measure corresponding to $\mu_t$ (see~\eqref{def:marg-cond} for precise definition). In what follows, we will refer to~\eqref{eq:intro-CG} as the \emph{coarse-grained dynamics}, as it exactly characterises the time-marginal dynamics of the projected jump process. Intuitively, the generator $\hat L_t$ is an average of the original generator $L$ over the level-sets of the coarse-graining map $\xi$ with respect to the reference dynamics $\mu_t$. Note that the generator $\hat L_t$ is time-dependent and requires information about $\mu_t$ (in the definition of the generator). Therefore, even though~\eqref{eq:intro-CG} correctly captures the statistics of the projected process, it is as analytically and computationally intractable as  the original system~\eqref{eq:intro-ForKol}. It should also be noted that while~\eqref{eq:intro-CG} is the evolution of the time-marginals for the projected jump process, the projected process itself will not be Markovian.  

Inspired by recent developments~\cite{legollLelievre10} in coarse-graining of diffusion processes in the context of molecular dynamics, we propose the following linear \emph{effective dynamics} 
\begin{equation}\label{eq:intro-eff}
\partial_t\eff=N^T\eff,
\end{equation}
with a time-independent generator $N\in \R^{|\Y|\times|\Y|}$ 
\begin{equation}\label{def:intro-eff-gen}
N(y_1,y_2)\coloneqq \sumsum_{x_1\in\Lambda_{y_1},x_2\in\Lambda_{y_2} } L(x_1,x_2)\stat(x_1|y_1),
\end{equation}
where $\stat(\cdot|y)\in\P(\Lambda_y)$ is the conditional (stationary) measure corresponding to the stationary measure $\stat$ for the full dynamics~\eqref{eq:intro-CG}. 
The intuition behind the construction of the effective dynamics is that in practice the coarse-grained variables encode the slow features of the system and therefore it is natural to expect that the dynamics on the level-sets (characterised by conditional dynamics $\mu_t(\cdot|y)$) will equilibrate considerably faster than the full dynamics. Consequently, the level-set average with respect to $\mu_t(\cdot|y)$ in the generator for the coarse-grained dynamics can be approximated by an average with respect to the the stationary conditional measure $\stat(\cdot|y)$. A crucial advantage of using the effective dynamics~\eqref{eq:intro-eff} over the coarse-grained dynamics~\eqref{eq:intro-CG} is that the generator for the former is time-independent and therefore can be computed offline. 

However, by construction the effective dynamics is only an approximation of the coarse-grained dynamics. This leads to the natural question that will be discussed in this article: 
\begin{center}
\emph{Is the effective dynamics~\eqref{eq:intro-eff} a good approximation of the coarse-grained dynamics~\eqref{eq:intro-CG}, and if so can the error between the two be quantified? }
\end{center}

\subsection{Main results, novelty and outline of the article}
The answer to the question stated above is indeed affirmative, in that we provide sufficient conditions under which the effective dynamics stays close to the coarse-grained dynamics. We now briefly discuss these result. 

Our first main result is Theorem~\ref{thm:RelEntEst}, where for any $t>0$ we prove an error bound of the form (also see~\eqref{eq:Unif-time-est})
\begin{equation}\label{eq:intro-gen-bound}
\RelEnt(\cg_t|\eff_t) \leq \RelEnt(\cg_0|\eff_0) + C,
\end{equation}
where $\RelEnt(\cg_t|\eff_t)$ is the relative entropy (see~\eqref{def:relEnt} for definition) of the coarse-grained dynamics with respect to the effective dynamics at time $t$ and $C>0$ is a constant independent of $t$. Note that by the Csisz\'ar-Kullback-Pinsker (CKP) inequality (see~\eqref{def:CKP}), this also translates into a bound in total-variation distance. The key assumption to prove this error bound is that the conditional stationary measure $\stat(\cdot|y)\in \P(\Lambda_y)$ which appears in the generator $N$~\eqref{def:intro-eff-gen} for the effective dynamics satisfies a logarithmic-Sobolev inequality (shortened to log-Sobolev here onwards). Intuitively, this assumption encodes the underlying idea that the coarse-grained variables are the slow ones, and therefore the dynamics on the level sets of the coarse-graining map will equilibrate quickly to the conditional stationary measure $\stat(\cdot|y)$ (see Remark~\ref{rem:LSI},~\ref{rem:LSI-gen-est} for details). The constant appearing in this log-Sobolev inequality is a key component in the error estimate and it turns out that the size of the constant reflects how well the chosen coarse-graining map captures the scale-separation in the problem. Therefore, our result also gives a tool to measure the quality of a chosen coarse-graining map $\xi$. We refer to Section \ref{sec:choice_CG_map} for an example on how different choices of $\xi$ effect the log-Sobolev constant.

While~\eqref{eq:intro-gen-bound} is a uniform (stability) bound that holds under fairly weak conditions and states that the effective dynamics can at most be a constant distance (in total-variation metric) away from the coarse-grained dynamics, it is completely unclear if it is a sharp bound. To analyse this, in Section~\ref{sec:scale-sep} we look at averaging problems for Markov chains which introduce an explicit scale-separation via a small parameter $\e>0$. The slow variable is the natural choice for the coarse-graining map in this setting. One key observation is that, in general, the resulting effective dynamics depends on $\e$ and is different from the classical \emph{averaged dynamics} which corresponds to the limit $\e\rightarrow 0$. As a consequence, the classical averaging literature and related techniques do not apply to the effective dynamics. In Theorem~\ref{prop:eff-to-aver-eps} we show that the $\e$-dependent effective dynamics $\eff_t^\e$ does converge to the averaged dynamics. The main result here is an $\e$-version of the bound~\eqref{eq:intro-gen-bound}, which for a fixed $\e>0$ and with same initial data $\cg^\e_0=\eff^\e_0$ reads (see Theorem~\ref{thm:eps-est-positive-init})
\begin{equation}\label{eq:intro-eps-bound}
\sup_{t\in [0,T]}\RelEnt(\cg^\e_t|\eff^\e_t) \leq  C(T)\e  ,
\end{equation}
where the finite constant $C$ is independent of $\e$. The bound~\eqref{eq:intro-eps-bound} is indeed \emph{good} in that the error vanishes as $\e\rightarrow 0$, and is expected to be \emph{sharp} as the suggested linear decay is in line with the numerical experiments  (see Section~\ref{sec:numericalExp}). Note that, as compared to the general estimate~\eqref{eq:intro-gen-bound}, in the multiscale setting the estimate has poor long-time scaling due to the presence of $T$ in the right-hand side of~\eqref{eq:intro-eps-bound}. 

\paragraph{Novelty.} A systematic analysis of coarse-graining and Markovian approximations in the absence of scale-separation has been limited to diffusion processes~\cite{legollLelievre10, ZhangHartmannSchutte16, LegollLelievreOlla17, LegollLelievreSharma18, DLPSS18, LelievreZhang18, LegollLelievreSharma18,HartmannNeureitherSharma20}.
This work provides first quantitative results on coarse-graining of jump processes in the absence of scale-separation. To the best of our knowledge, effective dynamics of the type presented in this article have not been considered for Markov chains and our quantitative results clearly indicate their efficacy in approximating the coarse-grained variables. 
Ideas related to coarse-graining of Markov jump processes have been studied under the term \emph{lumpability} (see for instance~\cite{Buchholz94}).
However, these results mainly try to extract information on the full Markov chain from the reduced Markov chain without error and seem unrelated to our results.

While averaging techniques for handling multiscale problems in Markov chains are well understood~\cite{PavliotisStuart08,lahbabiLegoll13,zhang16,HilderPeletierSharmaTse20,MielkeStephan20,PeletierRenger20}, their connection to coarse-graining and effective dynamics is completely new. In this article we explore these connections in detail and provide new insights into and error bounds for averaging problems. Furthermore, our analysis predominantly exploits entropy techniques and functional inequalities, which are traditionally used to study long-time behaviour of jump processes and not for multiscale problems. A key outcome of our results in the multiscale setting is that the effective dynamics converges to the averaged dynamics in a fairly general setting -- such results only hold in a limited setting for diffusions~\cite{HartmannNeureitherSharma20}.  

\paragraph{Outline of the article.}

The remainder of this article is organised as follows. In Section~\ref{sec:setup} we introduce notations, preliminaries and setup the main evolutions. Section~\ref{sec:generalEstimate} deals with general coarse-graining error estimates and Section~\ref{sec:scale-sep} with error estimates in the presence of explicit scale-separation in the system.  
In Section~\ref{sec:numericalExp} we illustrate our theoretical findings with numerical examples.  We conclude with discussions on various related issues in Section~\ref{sec:discussion}. Appendix~\ref{app:forKol-low} collects some crucial results on irreducible generators which will be used throughout this article.

\section{Setup and preliminaries}\label{sec:setup}
In Section~\ref{subsec:setup} below we present some preliminaries and notations. In Section~\ref{subsec:CG-eff} we introduce the two main equations that will be studied throughout the article, the coarse-grained and the effective dynamics, and present some of their useful properties.

\subsection{Preliminaries and notation}\label{subsec:setup}

We consider a continuous-time Markov chain on a finite state space $\X$ with irreducible generator $L\in\R^{|\X|\times |\X|}$ and corresponding forward Kolmogorov equation
\begin{equation}\label{eq:ForKol}
\begin{aligned}
\partial_t\mu = L^T\mu,\\
\mu|_{t=0}=\mu_0,
\end{aligned}
\end{equation}
where $\mu_0\in \P(\X)$. 
We will use $\mu_t$ to denote the time slice of $\mu$ at time $t$. Since $L$ is irreducible and $\X$ is finite (which implies that $L$ is positive recurrent),~\eqref{eq:ForKol} admits a unique and positive stationary measure $\stat\in\P_+(\X)$, i.e.\ $L^T\stat=0$ (see for instance~\cite[Theorem 3.5.2]{Norris98}).

We define the \emph{coarse-graining} map
\begin{equation*}
\xi:\X\rightarrow\Y,
\end{equation*}
where $\Y$ is a finite set with $|\Y|\leq |\X|$. 
For any $y\in\Y$ we use
\begin{equation*}
\Lambda_y\coloneqq\{x\in\X: \xi(x)=y\},
\end{equation*}
for the $y$-level set of $\xi$. Any $\nu\in\P(\X)$, can be decomposed into the corresponding marginal measure (or marginal probability distribution) $\xi_\#\nu\in\P(\Y)$ and for any $y\in\Y$ the family of conditional measures (or conditional probability distributions) $\nu(\cdot|y)\in\P(\Lambda_y)$  defined as
\begin{equation}\label{def:marg-cond}
\forall y\in\Y: \ \xi_\#\nu(y)\coloneqq\sum_{x\in\Lambda_y}\nu(x),\quad
\forall x\in \Lambda_y: \ \nu(x|y)=\frac{\nu(x)}{\xi_\#\nu(y)},
\end{equation}
where the conditional measure $\nu(\cdot|y)$ is defined $\xi_\#\nu$-almost everywhere. In all the applications of this decomposition in this article $\xi_\#\nu$ will be a positive probability measure and so the conditional measure will be defined everywhere. Consequently, for any $x\in\Lambda_y$ we can write $\nu(x)=\nu(x|y)\xi_\#\nu(y)$.

We now introduce some notation for vectors and matrix norms that will be used throughout this article. We write $\|\rho\|_{\R^{|\Z|}}$  for the standard Euclidean norm of a vector $\rho$ in $\R^{|\Z|}$. For a probability measure $\rho\in \P(\Z)$, the total-variation norm is $\|\rho\|_{\TV}\coloneqq\sum_{z\in\Z}|\rho(z)|$. For a matrix $A\in \R^{|\Z|\times|\Z|}$, we will use the Frobenius and supremum norm respectively defined as (with $n=|\Z|$)
\begin{equation*}
\|A\|^2_{\R^{|\Z|\times|\Z|}}\coloneqq\max_{1\leq i,j \leq n} |A_{ij}|^2, \ \ \|A\|_{\infty}\coloneqq\max_{1\leq i \leq n} \sum_{j=1}^{n} |A_{ij}|.
\end{equation*}

Next we introduce the relative entropy and Fisher information. For $\nu,\zeta\in\P(\X)$ we define the \emph{relative entropy} of $\nu$ with respect to $\zeta$ as 
\begin{equation}\label{def:relEnt}
\RelEnt(\nu|\zeta)\coloneqq\begin{dcases} \sum_{x\in\X} \nu(x) \log f(x),  \ \ &f=\frac{d\nu}{d\zeta}, \\
+\infty, & \text{otherwise}.
\end{dcases}
\end{equation}
The Csisz\'ar-Kullback-Pinsker (CKP) inequality connects the total-variation norm on $\P(\X)$ to the relative entropy via
\begin{equation}\label{def:CKP}
\forall \nu,\zeta\in\P(\X): \ \ \|\nu-\zeta\|_{\TV} \leq \sqrt{2\RelEnt(\nu|\zeta)}.
\end{equation}
For $\nu,\zeta\in\P(\X)$ and generator  $M\in \R^{|\X|\times|\X|}$ we define the \emph{$M$-Fisher information} of $\nu$ with respect to $\zeta$ as 
\begin{align}
\RF_{M}(\nu|\zeta)&\coloneqq \sum_{x\in\X} \Bigl[ - \Bigl(M\log\frac{\nu}{\zeta}\Bigr)(x) + \frac{\zeta(x)}{\nu(x)} \Bigl(M\frac{\nu}{\zeta} \Bigr)(x)  \Bigr]\nu(x)  \label{def:FI}\\
&=\sumsum_{x,x'\in\X} \Bigl[ -M(x,x')\log\Bigl( \frac{\nu(x')}{\zeta(x')}\Bigr) + M(x,x')\Bigl( \frac{\nu(x')\zeta(x)}{\nu(x)\zeta(x')}\Bigr)  \Bigr]\nu(x). \notag
\end{align}
The following lemma states that the Fisher information is non-negative. 
\begin{lem}\label{lem:pos-FI}
For any generator $M\in \R^{|\X|\times|\X|}$ and $\nu,\zeta\in\P_+(\X)$, the $M$-Fisher information of $\nu$ with respect to $\zeta$ is non-negative, i.e.\ $\RF_M(\nu|\zeta)\geq 0$.
\end{lem}
\begin{proof}
Note that for any $\nu,\zeta\in\P_+(\X)$ the Fisher information is a well-defined object. Since $M$ is a generator it satisfies $\sum_{x'\in\X} M(x,x')=0$, and therefore we can rewrite the Fisher information as 
\begin{equation}\label{def:FI-alter}
\RF_{M}(\nu|\zeta) = \sumsum_{x,x'\in\X} M(x,x')\zeta(x) \ell(x) \Bigl[\frac{\ell(x')}{\ell(x)} - 1 - \log\Bigl( \frac{\ell(x')}{\ell(x)}\Bigr) \Bigr], \ \ \ell=\frac{\nu}{\zeta}.
\end{equation}
The result then follows since for $\alpha\in(0,\infty]$, $\alpha\mapsto \alpha-1-\log\alpha$ is a non-negative function. 
\end{proof}
A straightforward calculations shows that $\partial_t\RelEnt(\mu_t|\stat) = -\RF_{L}(\mu_t|\stat)$ (see~\cite[Lemma 2.3]{bobkovTetali06}), where $\mu_t$ solves~\eqref{eq:ForKol} and $\stat$ is the corresponding stationary measure, and therefore
for any $t\geq 0$ we have
\begin{equation}\label{eq:ent-decay}
\RelEnt(\mu_t|\stat)  + \int_0^T \RF_L(\mu_t|\stat)dt= \RelEnt(\mu_0|\stat),
\end{equation}
i.e.\ the relative entropy with respect to the stationary measure is a Lyapunov function for~\eqref{eq:ForKol} since the Fisher information is non-negative. 

Next we introduce the log-Sobolev inequality which will play a key role in our main results. A probability measure $\zeta\in\P(\X)$ and a matrix $M\in\R^{|\X|\times|\X|}$ satisfy the \emph{log-Sobolev inequality} with constant $\alpha_{\LSI} > 0$ if 
\begin{equation}\label{def:LSI}
\forall \nu\in \P(\X): \ \ \RelEnt(\nu|\zeta)\leq  \frac{1}{\alpha_{\LSI}} \RF_M(\nu|\zeta).
\end{equation}
Note, that in particular if $|\X| = 1$ then \eqref{def:LSI} is trivially satisfied for all $\alpha_{\LSI} > 0$.

\begin{rem}\label{rem:LSI}
We point out that \eqref{def:LSI} is also referred to as \emph{modified} log-Sobolev inequality in the literature, see e.g.~\cite{bobkovTetali06,erbarMaas12,zhang16}.
However, the log-Sobolev inequality defined in \cite{bobkovTetali06,zhang16} implies \eqref{def:LSI} and therefore, assuming \eqref{def:LSI} is a weaker assumption. 
In particular, both versions have been considered in the reversible and non-reversible setting and both quantify the convergence to equilibrium, see e.g.~\cite[Theorem 2.4]{bobkovTetali06}. It should be noted that the log-Sobolev inequality holds for any irreducible Markov chain (and its invariant measure) on a finite state space without assuming reversibility~\cite[Theorem~2.2.3]{Saloff-Coste97}.
In this article we use \eqref{def:LSI} since it fits naturally in our setting. 

In fact, assuming that the measure $\zeta\in\P(\X)$ and generator $M$, with $M^T\zeta=0$, satisfy the log-Sobolev inequality~\eqref{def:LSI} is equivalent to assuming that any solution to the forward Kolmogorov equation $\partial_t\nu=M^T\nu$ converges exponentially to $\zeta$ in the following sense. Repeating the calculations above, we find $\partial_t\RelEnt(\nu_t|\zeta) = -\RF_{M}(\nu_t|\zeta)$ and therefore~\eqref{def:LSI} implies the exponential convergence
\begin{equation*}
\forall t\geq 0: \  \RelEnt(\nu_t|\zeta) \leq e^{-\alpha_{\LSI}t}\RelEnt(\nu_0|\zeta).
\end{equation*}
Conversely, if this convergence holds for any initial data $\nu_0\in\P(\X)$ we can write
\begin{equation*}
\forall t>0: \ \frac{1}{t} \bigl[\RelEnt(\nu_t|\zeta) - \RelEnt(\nu_0|\zeta) \bigr] \leq \RelEnt(\nu_0|\zeta)\Bigl( \frac{e^{-\alpha_{\LSI}t}-1}{t}\Bigr).
\end{equation*}
Consequently, passing to the limit $t\rightarrow 0$ and using $\partial_t\RelEnt(\nu_t|\zeta)\bigl|_{t=0} = -\RF_{M}(\nu_0|\eta)$, we arrive at the log-Sobolev inequality~\eqref{def:LSI}. 
\end{rem}

The preliminaries presented above (and the definitions for the coarse-grained and effective dynamics in the next section) can also be written in matrix form since we are dealing with finite state space $\X$. In the following remark we briefly summarise these ideas (see~\cite{MielkeStephan20} for details). 
\begin{rem}\label{rem:MatrixNotation}
Given any probability measure $\nu\in \P(\X)=\{\nu\in\R^{|\X|}:\nu_i\geq 0, \sum_i\nu_i=1\}$ and coarse-graining map $\xi:\X\rightarrow\Y$, the marginal measure satisfies
\begin{equation*}
\hat\nu = \xi_\#\nu \ \Longleftrightarrow \ \hat\nu = M \nu 
\end{equation*}
where $\hat\nu\in \P(\Y)=\{\zeta\in\R^{|\Y|}: \zeta_j\geq 0, \sum_j\zeta_j=1\}$ and $M\in\R^{|\Y|\times|\X|}$ is a matrix with $M_{ji}\in\{0,1\}$ and $\sum_{j=1}^{|\Y|} M_{ji}=1$. The level sets $\Lambda_y$ are sometimes referred to as \emph{clusters}. The conditional measure $\nu(\cdot|\cdot)$ can be expressed using the so-called \emph{reconstruction operator} $R^\nu\in \mathbb R^{|\X|\times|\Y|}$, defined as
\begin{equation*}
R^\nu_{ij} = \frac{\nu_{i}}{(M\nu)_{j}},
\end{equation*}
with the relation $\nu(.\vert j) = R_{.j}^\nu$.
\end{rem}

\subsection{Coarse-grained and effective dynamics}\label{subsec:CG-eff}

In this section we define the central dynamics studied in this article and state certain useful properties. 
The following lemma characterises the evolution of $\xi_\#\mu_t$ on $\Y$, which we call the coarse-grained dynamics.  
\begin{lem}\label{lem:CG}
Let $\mu\in C^1([0,T];\P(\X))$ be a solution of~\eqref{eq:ForKol}. Then $\cg\in C^1([0,T];\P(\Y))$ defined via $ \cg_t\coloneqq\xi_\#\mu_t$ evolves according to 
\begin{equation}\label{eq:cg}
\begin{aligned}
\partial_t\cg_t=\hat L_t^T\cg_t,\\
\cg|_{t=0}=\xi_\#\mu_0,
\end{aligned}
\end{equation}
where $\hat L:[0,T]\rightarrow\R^{|\Y|\times|\Y|}$ is given by
\begin{equation}\label{def:cg-gen}
\hat L_t(y_1,y_2)\coloneqq \sumsum_{x_1\in\Lambda_{y_1},x_2\in\Lambda_{y_2} } L(x_1,x_2)\mu_t(x_1|y_1).
\end{equation}
Furthermore for any $t$, $\hat L_t$ is a generator, i.e.\ $\hat L_t(y_1,y_2)\geq 0$ for any $y_1\neq y_2$ and $\sum_{y_2} \hat L_t(y_1,y_2)=0$. 
\end{lem}
\begin{proof}
For any $y_1 \in \Y$ we find
\begin{align*}
\partial_t \cg_t(y_1) &= \sum_{x_1 \in \Lambda_{y_1}} \partial_t \mu_t(x_1) 
= \sumsum_{x_1 \in \Lambda_{y_1},x_2 \in \X} L(x_2,x_1) \mu_t(x_2) 
= \sumsum_{x_1 \in \Lambda_{y_1},x_2 \in \X} L(x_2,x_1) \hat{\mu}_t(\xi(x_2)) \mu_t(x_2 \vert \xi(x_2)) \\
&= \sum_{y_2 \in \Y} \hat{\mu}_t(y_2) \sumsum_{x_1\in\Lambda_{y_1},x_2\in\Lambda_{y_2} } L(x_2,x_1)\mu_t(x_2 \vert y_2) 
= \sum_{y_2 \in \Y} \hat{L}_t(y_2,y_1) \hat{\mu}_t(y_2),
\end{align*}
where we use the explicit expression~\eqref{def:marg-cond} for the marginal and conditional measure. By construction, for any $t\geq 0$ and $y_1\neq y_2$, $\hat L_t(y_1,y_2)\geq 0$. Furthermore for any $y_1\in\Y$,
\begin{equation*}
\sum_{y_2\in\Y} \hat L_t(y_1,y_2) = \sum_{x_2\in\Lambda_{y_1}}  \Bigl(\sumsum_{y_2\in\Y, x_2\in\Lambda_{y_2}} L(x_1,x_2) \Bigr) \mu_t(x_1|y_1) = \sum_{x_2\in\Lambda_{y_1}}  \Bigl(\sum_{x_2\in\X} L(x_1,x_2) \Bigr) \mu_t(x_1|y_1) =0,  
\end{equation*}
where the final equality follows since $L$ is a generator. 
\end{proof}
We will make use of the so-called \emph{effective dynamics}, which describes the evolution of $\eff\in C^1([0,T];\P(\Y))$
\begin{equation}\label{eq:eff}
\begin{aligned}
\partial_t\eff=N^T\eff,\\
\eff|_{t=0}=\eff_0, 
\end{aligned}
\end{equation}
where $N\in \R^{|\Y|\times|\Y|}$ is defined as 
\begin{equation}\label{def:eff-gen}
N(y_1,y_2)\coloneqq \sumsum_{x_1\in\Lambda_{y_1},x_2\in\Lambda_{y_2} } L(x_1,x_2)\stat(x_1|y_1).
\end{equation}
Recall that $\stat\in\P(\X)$ is the stationary measure for the original dynamics~\eqref{eq:ForKol} and for any $y\in\Y$, $\rho(\cdot|y)$ is the corresponding conditional stationary measure. Note that the effective dynamics is linear.

The following result states that the effective generator $N$ is irreducible.
Consequently, using standard results (see for instance~\cite[Appendix C]{HilderPeletierSharmaTse20}) the effective dynamics is a positive probability measure, i.e.\ $\eff_t\in\P_+(\Y)$ for any $t>0$. 

\begin{lem}\label{lem:eff-irred}
The effective generator $N$ defined in~\eqref{def:cg-gen} is irreducible. Furthermore, for any $t>0$, the solutions to~\eqref{eq:cg},~\eqref{eq:eff} are strictly positive, i.e.\ $\cg_t,\eff_t\in\P_+(\Y)$.
\end{lem}
\begin{proof}
The positivity of the solution $\mu_t$ to the coarse-grained dynamics is guaranteed by the positivity of $\mu_t$ (since $L$ is irreducible), the definition of the push-forward measure~\eqref{def:marg-cond} and since $\Lambda_y$ is non-empty for every $y\in\Y$.  

We will now show that the effective generator $N$ is irreducible, which by standard results  (see for instance~\cite[Appendix C]{HilderPeletierSharmaTse20}) ensures that $\eff_t\in\P_+(\Y)$ for any $t>0$.
Fix $y_1,y_2\in\Y$. 
Since $L$ is irreducible, for any $x_1\in\Lambda_{y_1}, x_2\in\Lambda_{y_2}$, there exists a finite sequence $(\tilde x_\alpha)_{\alpha=0,\ldots,n}\in\X$ with $\tilde x_0=x_1, \tilde x_n=x_2$ such that $L(\tilde{x}_\alpha,\tilde{x}_{\alpha+1})>0$ for any $\alpha\in\{0,\ldots, n-1\}$ (see~\cite[Theorem 3.2.1]{Norris98}). 
Consider the sequence $(\tilde y_\alpha)_{\alpha=0,\ldots,m}\in\Y$ defined via $\tilde y_\alpha=\xi(\tilde x_\alpha)$ such that any doubled points are removed and therefore $m\leq n$. Clearly $\tilde y_0=y_1$ and $\tilde y_m=y_2$. Therefore using the definition of $N$ and $\alpha\in\{0,\ldots, m-1\}$ we find
\begin{equation*}
N(\tilde y_\alpha,\tilde y_{\alpha+1}) = \sumsum_{\hat{x}_1\in\Lambda_{\tilde y_\alpha},\hat{x}_2\in\Lambda_{\tilde{y}_{\alpha+1}} } L(\hat{x}_1,\hat{x}_2)\stat(\hat{x}_1|\tilde{y}_{\alpha}) \geq L(\tilde x_\alpha,\tilde x_{\alpha+1})\stat(\tilde x_\alpha|\tilde y_\alpha)>0,
\end{equation*}
where $\tilde x_\alpha\in \Lambda_{\tilde y_\alpha},\tilde x_{\alpha+1}\in \Lambda_{\tilde y_{\alpha+1}}$ by construction of $(\tilde y_{\alpha})$ and $L(\tilde x_\alpha,\tilde x_{\alpha+1})>0$ since $\tilde x_\alpha\neq \tilde x_{\alpha+1}$. Furthermore  we have used $\stat(x|y)>0$ for any $(y,x)\in(\Y,\Lambda_y)$, which follows since $\stat>0$ (since $L$ is irerducible)  and using the decomposition~\eqref{def:marg-cond}. Since $y_1,y_2\in\Y$ are arbitrary, it follows that the effective generator $N$ is irreducible. 
\end{proof}
Note that a similar proof as above can be used to show that the coarse-grained generator $\hat L_t$ is irreducible for any $t>0$. However this need not imply the positivity of $\cg_t$ since the generator depends on time.

Since the effective generator $N$ is irreducible and $\Y$ is finite (which implies that $N$ is positive recurrent), the effective dynamics admits a stationary measure. The 
following result states that the stationary measure for the effective dynamics is $\xi_\#\stat$, where $\stat\in\P(\X)$ is the stationary measure of the original dynamics~\eqref{eq:ForKol}.  Furthermore, the coarse-grained dynamics and the effective dynamics coincide in the long-time limit, i.e.\ as $t\rightarrow\infty$, which is to be expected by the construction of these systems.
\begin{prop}\label{prop:eff-inv}
Let $\cg_t$ $\eff_t$ solve~\eqref{eq:cg},\eqref{eq:eff} respectively. 
The effective dynamics~\eqref{eq:eff} admits $\xi_\#\stat\in\P(\Y)$ as a stationary measure. Furthermore, 
\begin{equation*}
\|\cg_t-\eff_t\|_{\TV}\xrightarrow{t\rightarrow\infty} 0.
\end{equation*}
\end{prop}
\begin{proof}
Using the definition of $N$ and the definition of the conditional, marginal stationary measure~\eqref{def:marg-cond}, for any $y\in\Y$ we find
\begin{align*}
(N^T\xi_\#\stat)(y_2)&=\sum_{y_1\in\Y} N(y_1,y_2)(\xi_\#\stat)(y_1) = \sum_{y_1\in\Y} \Bigl(\sumsum_{x_1\in\Lambda_{y_1},x_2\in\Lambda_{y_2} } L(x_1,x_2)\stat(x_1|y_1) (\xi_\#\stat)(y_1) \Bigr)\\
&=\sum_{x_2\in\Lambda_{y_2}} \Bigl(\sum_{x_1\in\X} L(x_1,x_2)\stat(x_1)  \Bigr) = 0,
\end{align*}
where the final equality follows since $(L^T\stat)(x)=0$ for any $x\in\X$.
Since $\cg_t=\xi_\#\mu_t$ and $\mu_t\rightarrow\stat$ as $t\rightarrow\infty$ it follows that 
\begin{equation*}
\|\cg_t-\xi_\#\stat\|_{\TV} =  \frac12\sum_{y\in\Y} |\cg_t(y)-\xi_\#\stat(y)| \leq \frac12 \sum_{y\in\Y} \sum_{x\in\Lambda_y} |\mu_t(x)-\stat(x)| \xrightarrow{t\rightarrow\infty} 0,
\end{equation*}
where the equality is a standard property of the total-variation metric, the inequality follows from the definition of the marginal measure and the limit follows since $\X,\Y$ are finite. The final result then follows by using the triangle inequality.
\end{proof}

\section{General error estimates}\label{sec:generalEstimate}
We now state our main result, which provides error estimates comparing the coarse-grained dynamics~\eqref{eq:cg} and the effective dynamics~\eqref{eq:eff} in relative entropy. By using the CKP inequality~\eqref{def:CKP}, this result also provides estimates in the total-variation distance.

\begin{theorem}\label{thm:RelEntEst} 
For any $y\in\Y$, let $L^y\in\R^{|\Lambda_y|\times|\Lambda_y|}$ be the restriction of $L$ to $\Lambda_y \times \Lambda_y$. Let $\cg,\eff\in C^1([0,\infty);\P(\Y))$ be the solutions to the coarse-grained dynamics~\eqref{eq:cg} and the effective dynamics~\eqref{eq:eff} respectively. Assume that the conditional stationary measure $\stat(\cdot|y)$ and $L^y$ satisfies the log-Sobolev inequality uniformly in  $y\in\Y$ with constant $\alpha_{\LSI}$ (recall~\eqref{def:LSI}), i.e.\  
\begin{equation}
    \exists\, \alpha_{\LSI}>0, \ \forall y\in\Y, \ \forall \nu\in \P(\Lambda_y): \ \ \RelEnt\bigl(\nu|\rho(\cdot|y)\bigr)\leq  \frac{1}{\alpha_{\LSI}} \RF_{L^y}\bigl(\nu|\rho(\cdot|y)\bigr),
    \label{eq:log_Sobolev_mainThm}
\end{equation}
where $\RF_{L^y}(\cdot|\cdot)$ is the $L^y$-Fisher information (see~\eqref{def:FI}).

Then, for any $t > 0$ the estimate
\begin{equation}\label{eq:RelEntEst}
\RelEnt(\cg_t|\eff_t) \leq \RelEnt(\cg_0|\eff_0) + \frac{C}{\sqrt{\alpha_{\LSI}}} \bigl[ \RelEnt(\mu_0|\rho)-\RelEnt(\mu_t|\rho) \bigr]^\frac12,
\end{equation}
holds, where $C=C(L,N,\stat)$ is independent of $t>0$.

Additionally, for any fixed $T < \infty$ there exists a constant $\tilde{C} = \tilde{C}(L,|\X|,T)$ independent of the effective generator $N$ and the coarse-graining map $\xi$, such that the estimate
\begin{equation}\label{eq:RelEntEstFiniteTime}
\RelEnt(\cg_t|\eff_t) \leq \RelEnt(\cg_0|\eff_0) + \frac{\tilde{C}}{\sqrt{\alpha_{\LSI}}} \bigl[ \RelEnt(\mu_0|\rho)-\RelEnt(\mu_t|\rho) \bigr]^\frac12,
\end{equation}
holds for any $t \in [0,T]$.
\end{theorem}

\begin{rem}
    Under the additional assumption of strictly positive initial datum, i.e.~there exists a constant $c_0 > 0$ such that $\cg_0(y), \eta_0(y) > c_0$ for all $y \in \Y$, the scaling of the error estimate \eqref{eq:RelEntEstFiniteTime} in $\alpha_{\LSI}$ can be improved to arrive at
    \begin{align*}
        \sup_{t \in [0,T]} \RelEnt(\cg_t \vert \eta_t) \leq 2 \RelEnt(\cg_0 \vert \eta_0) + \dfrac{\tilde{C}}{\alpha_{\LSI}} \left[\RelEnt(\mu_0 \vert \rho) - \RelEnt(\mu_T \vert \rho)\right],
    \end{align*}
    where the constant $\tilde{C}$ now additionally depends on $c_0$. For details we refer to Theorem \ref{thm:eps-est-positive-init} and its proof, in particular estimates \eqref{eq:g_t_integral_bound_eps} and \eqref{eq:asym-first-est}, as well as Remark \ref{rem:differences-eps-thm-noneps-thm}. However it is unclear if such an improvement can also be made in the case $T = \infty$.
\end{rem}

We point out that since the constant $\tilde{C}$ in \eqref{eq:RelEntEstFiniteTime} is independent of the effective generator $N$ and the coarse-graining map $\xi$, the only part of the constant in the error estimate \eqref{eq:RelEntEstFiniteTime} that is dependent on the choice of $\xi$ is the log-Sobolev constant $\alpha_{\LSI}$. Theorem~\ref{thm:RelEntEst} thereby lends itself to two important observations. First, the quality of the coarse-graining map is characterised via the log-Sobolev constant. Second, while we have a general error estimate for any coarse-graining map, the effective dynamics will be closer to the coarse-grained dynamics for the choices of maps for which the log-Sobolev constant is large. This is clearly seen in the simple toy-problem setting discussed in Section~\ref{sec:choice_CG_map} where certain choices are clearly more reasonable than others. %}
Note that the constant $\tilde{C}$ depends on $T$ and cannot be extended to $T = \infty$ with the current tools as discussed in Remark~\ref{rem:improvedConstant}.

\begin{rem}\label{rem:diagonal_elems_restriction}
    Since the Fisher information $\RF_{L^y}$ does not depend on the diagonal elements of $L^y$ (see \eqref{def:FI-alter}) we may modify the diagonal elements of $L^y$ such that the rows of $L^y$ sum up to zero thereby making it a generator.
\end{rem}

\begin{rem}\label{rem:LSI-singleton}
   To obtain a uniform $\alpha_{\LSI}$ in \eqref{eq:log_Sobolev_mainThm} it is sufficient to establish a log-Sobolev estimate
    \begin{align*}
        \RelEnt(\nu | \rho(\cdot | y)) \leq \dfrac{1}{\alpha_{\LSI}(y)} \RF_{L^y}(\nu | \rho(\cdot | y)), \quad \nu \in \P(\Lambda_y)
    \end{align*}
    for every fixed $y$ with a $y$-dependent constant $\alpha_{\LSI}(y)$. Then \eqref{eq:log_Sobolev_mainThm} is obtained by taking $\alpha_{\LSI} \coloneqq \inf_{y \in \Y} \alpha_{\LSI}(y)$. Here the infinimum is always positive since $\Y$ is a finite state-space.
    
     As pointed out below \eqref{def:LSI} if $|\Lambda_y| = 1$ for some $y \in \Y$ the inequality holds for all $\alpha_{\LSI}(y)$. Therefore, it is reasonable to take the infimum above only over $y$ such that the corresponding level set $\Lambda_y$ contains more than one element, i.e.\ 
     \begin{align*}
         \alpha_{\LSI} = \inf \{\alpha_{\LSI}(y) \,:\, y \in \Y, |\Lambda_y| > 1\}.
     \end{align*}
     Note that we may always assume that there exists at least one $\Lambda_y$ with more than one element since the only case where all level sets only have one element is the case $\X = \Y$.
     In this case the effective dynamics and coarse-grained dynamics are the same and thus \eqref{eq:RelEntEst} holds trivially since the relative entropy $\RelEnt(\mu_t | \nu_t)$ decays in time if the curves $\mu$ and $\nu$ are produced by the same generator, see \cite[Equation (4)]{HilderPeletierSharmaTse20}.
\end{rem}

Using~\eqref{eq:ent-decay} which implies that the relative entropy with respect to $\stat$ is a Lyapunov function along the solution of~\eqref{eq:ForKol}, and assuming that $\cg_0=\eff_0$, the relative-entropy estimate~\eqref{eq:RelEntEst} can be simplified to 
\begin{equation}\label{eq:Unif-time-est}
\RelEnt(\cg_t|\eff_t) \leq \frac{C}{\sqrt{\alpha_{\LSI}}} \bigl(\RelEnt(\mu_0|\rho)\bigr)^{\frac12},
\end{equation}
which is a \emph{uniform} time estimate since $C$ does not depend on $t$.

\begin{rem}\label{rem:LSI-gen-est}
The main assumption in Theorem \ref{thm:RelEntEst} is that the conditional stationary measure satisfies the log-Sobolev inequality. This assumption has a clear interpretation in the reversible setting where $(L^y)^T\stat(\cdot|y)=0$ (see Section~\ref{sec:eps-rev} with $\e=1$) --  following Remark~\ref{rem:LSI} any dynamics driven by $L^y$ converges exponentially fast to the conditional stationary measure $\stat(\cdot|y)$. In this sense, the log-Sobolev constant corresponds to the speed of equilibration on the level sets $\Lambda_y$, that is the larger $\alpha_{\LSI}$ the faster the conditional measure $\mu_t(\cdot \vert y)$ converges to the conditional stationary measure $\rho(\cdot \vert y)$ as $t \rightarrow \infty$.

However, this interpretation is restricted to the reversible setting, since in general the conditional stationary measure need not be the stationary measure for $L^y$, i.e. $(L^y)^T\stat(\cdot|y)\neq 0$. Therefore, the log-Sobolev inequality is a technical assumption in the non-reversible setting. Similar issues and observations pertaining to assuming that the conditional stationary measure satisfies the log-Sobolev inequality also arise in the context of diffusion processes~\cite[Remark 2.5]{HartmannNeureitherSharma20}. For additional discussion see Section~\ref{sec:discussion}. 
\end{rem}

While we expect that the error between coarse-grained and effective dynamics should vanish as $t\rightarrow\infty$ (recall Proposition~\ref{prop:eff-inv}), Theorem~\ref{thm:RelEntEst} and~\eqref{eq:Unif-time-est} only provide a uniform-in-time upper bound for the error in relative entropy. Using the exponential convergence of the full dynamics~\eqref{eq:ForKol} to the stationary measure we can improve the long-time behaviour of the error estimate. 

\begin{cor}\label{corr:long-time} Under the assumptions in Theorem~\ref{thm:RelEntEst}, for any $t\geq 0$ we have
\begin{equation*}
\|\cg_t-\eff_t\|_{\TV} \leq \min\bigl\{ C_1(t), C_2e^{-Dt}\bigr\},
\end{equation*}
where $D,C_2>0$ are independent of $t$ and for a constant $C>0$ independent of $t$
\begin{equation*}
C_1(t)\coloneqq\Bigl(2\RelEnt(\cg_0|\eff_0) + \frac{C}{\sqrt{\alpha_{\LSI}}} \bigl[ \RelEnt(\mu_0|\rho)-\RelEnt(\mu_t|\rho) \bigr]^\frac12  \Bigr)^{\frac12} .
\end{equation*}
\end{cor}
\begin{proof}
Using Theorem~\ref{thm:RelEntEst} and the CKP inequality it follows that 
\begin{equation}\label{eq:long-time-1}
\|\cg_t-\eff_t\|_{\TV} \leq C_1(t).
\end{equation}
Since the original and effective generators are irreducible (recall Proposition~\ref{lem:eff-irred}), by Proposition~\ref{lem:forKol-exp} there exist constants $C(\mu_0), C(\eff_0),D_\mu,D_\eta>0$ independent of time such that 
\begin{equation*}
\|\cg_t-\xi_\#\stat\|_{\TV} \leq \|\mu_t-\stat\|_{\TV} \leq C(\mu_0)e^{-D_\mu t}, \ \ \|\eff_t-\xi_\#\stat\|_{\TV} \leq C(\eff_0)e^{-D_\eta t}. 
\end{equation*}
Therefore using the triangle inequality we find
\begin{equation}\label{eq:long-time-2}
\|\cg_t-\eff_t\|_{\TV} \leq \max\{C(\mu_0),C(\eff_0)\} e^{-t\min\{D_\mu,D_\eta\} } =: C_2e^{-Dt}.
\end{equation} 
The final result then follows by combining~\eqref{eq:long-time-1} and~\eqref{eq:long-time-2}. 
\end{proof}
Although the exponential bound $C_2e^{-Dt}$ in Corollary~\ref{corr:long-time} provides a decaying estimate in time, it should be smaller than $C_1(t)$ only at extremely long-times. This follows since $C_2e^{-Dt}$ encodes the convergence of the full system (including the `slow' coarse-grained variable) while $C_1$ only encodes the behaviour of the `fast' level-set dynamics characterised by the log-Sobolev constant. 

To prove Theorem~\ref{thm:RelEntEst} we will make use of the following result which provides an explicit characterisation of the error between $\cg_t,\eff_t$ in relative entropy. 

\begin{lem}\label{lem:pre-FIR}
The solutions $\cg,\eff\in C^1([0,\infty);\P(\Y))$ to~\eqref{eq:cg},~\eqref{eq:eff} satisfy 
\begin{equation}\label{eq:ent-bound1}
\RelEnt(\cg_t|\eff_t) - \RelEnt(\cg_0|\eff_0) + \int_0^t \RF_N(\cg_s|\eff_s) \,ds = \int_0^t\sum_{y\in\Y} \log\Bigl( \frac{\cg_s(y)}{\eff_s(y)}\Bigr) \bigl( \partial_t\cg_s  - N^T\cg_s\bigr)(y) \,ds,  
\end{equation}
for any $t > 0$.
Furthermore, for any $t>0$ the term on the right hand side of~\eqref{eq:ent-bound1} satisfies the upper bound 
\begin{equation}\label{eq:Log-est}
\sum_{y\in\Y} \log\Bigl( \frac{\cg_t(y)}{\eff_t(y)}\Bigr) \bigl( \partial_t\cg_t  - N\cg_t\bigr)(y) \leq 2 g_t \sum_{y\in \Y} \|\mu_t(\cdot|y) - \stat(\cdot|y) \|_{\TV}\, \cg_t(y),
\end{equation}
where $\|\cdot\|_{\TV}$ is the total-variation norm and $g:\mathbb R_{\geq 0} \rightarrow\mathbb R$  is given by
\begin{equation}\label{def:C_t}
g_t\coloneqq \sup_{x\in\X} f_t(x), \ \ \ f_t(x_1)\coloneqq\sum_{x_2\in\X} L(x_1,x_2) \Bigl[ \log\Bigl( \frac{\cg_t(\xi(x_1))}{\eff_t(\xi(x_1))}\Bigr) -  \log\Bigl( \frac{\cg_t(\xi(x_2))}{\eff_t(\xi(x_2))}\Bigr)  \Bigr].
\end{equation}
\end{lem}
\begin{proof}
Differentiating $\RelEnt(\cg_t|\eff_t)$ in time we find
\begin{align*}
\partial_t\RelEnt(\cg_t|\eff_t)dt &= \sum_{y\in \Y} \Bigl[ \partial_t \cg_t(y) \log\Bigl(\frac{\cg_t(y)}{\eff_t(y)} \Bigr)  - \frac{\cg_t(y)}{\eff_t(y)} \partial_t\eta_t(y) - \partial_t\hat\mu_t (y)\Bigr] \\
&= \sum_{y\in \Y}  \log\Bigl(\frac{\cg_t(y)}{\eff_t(y)} \Bigr) \bigl( \partial_t \cg_t - N^T\cg_t \bigr)(y) - \sum_{y\in \Y} \Bigl[  \frac{\cg_t(y)}{\eff_t(y)} \partial_t\eta_t(y) - \log\Bigl( \frac{\cg_t(y)}{\eff_t(y)}\Bigr) (N^T\cg_t)(y) \Bigr] \\
& = \sum_{y\in \Y}  \log\Bigl(\frac{\cg_t(y)}{\eff_t(y)} \Bigr) \bigl( \partial_t \cg_t - N^T\cg_t \bigr)(y) - \RF_N(\cg_t|\eff_t),
\end{align*}
where we have used $\sum_y \partial_t\cg_t(y)=0$ since $\cg_t\in\P(\Y)$ to arrive at the second equality, and the definition of the Fisher information~\eqref{def:FI} to arrive at the final equality. Equation~\eqref{eq:ent-bound1} then follows by integrating this equality in time over $[0,t]$. 

Using the evolution of the coarse-grained dynamics~\eqref{eq:cg} we find
\begin{align*}
\sum_{y_1\in \Y}&  \log\Bigl(\frac{\cg_t(y_1)}{\eff_t(y_1)} \Bigr) \bigl( \partial_t \cg_t - N^T\cg_t \bigr)(y_1) = \sumsum_{y_1,y_2\in \Y}   \bigl(\hat L_t(y_1,y_2) - N(y_1,y_2) \bigr)\log\Bigl(\frac{\cg_t(y_2)}{\eff_t(y_2)} \Bigr) \hat\mu_t(y_1) \\
&= \sumsum_{y_1,y_2\in \Y}   \bigl(\hat L_t(y_1,y_2) - N(y_1,y_2) \bigr)\Bigl[ \log\Bigl(\frac{\cg_t(y_2)}{\eff_t(y_2)}\Bigr)-\log\Bigl(\frac{\cg_t(y_1)}{\eff_t(y_1)} \Bigr) \Bigr] \hat\mu_t(y_1)\\
& = \sumsum_{y_1,y_2\in \Y} \sumsum_{x_1\in\Lambda_{y_1},x_2\in\Lambda_{y_2}}   L(x_1,x_2)   \bigl(\mu_t(x_1|y_1) - \stat(x_1|y_1) \bigr)\Bigl[ \log\Bigl(\frac{\cg_t(y_2)}{\eff_t(y_2)}\Bigr)-\log\Bigl(\frac{\cg_t(y_1)}{\eff_t(y_1)} \Bigr) \Bigr] \hat\mu_t(y_1) \\
&\leq g_t \sum_{y_1\in\Y} \hat\mu_t(y_1) \sum_{x_1\in\Lambda_{y_1}} \bigl|\mu_t(x_1|y_1) - \rho(x_1|y_1)  \bigr| = 2g_t \sum_{y_1\in\Y} \big\|\mu_t(\cdot|y_1)- \stat(\cdot|y_1) \big\|_{\TV} \cg_t(y_1),
\end{align*}
where $g_t$ is defined in~\eqref{def:C_t}. Here the second equality follows since $\sum_{y_2} \hat L_t(y_1,y_2) = \sum_{y_2} N(y_1,y_2) =0$, the third equality follows from~\eqref{def:cg-gen},~\eqref{def:eff-gen}. This concludes the proof of~\eqref{eq:Log-est}.
\end{proof}

The following result provides an upper bound on the time integral of the error term $g_t$ derived in the result above. 
\begin{lem}\label{lem:bound-g}
For any $t>0$, $t\mapsto g_t$ defined in~\eqref{def:C_t} satisfies $\|g\|_{L^2([0,t])}  \leq C $, 
where $C$ is independent of $t$.
Additionally, for any $T < \infty$ exists a $\tilde{C}$, which is independent of $N$ and $\xi$, such that $\|g\|_{L^2([0,t])} \leq \tilde{C}$ for all $t \in [0,T]$.
\end{lem}
\begin{proof}
We have the bound 
\begin{equation}\label{eq:C-init-est}
g_t \leq 2\|L\|_\infty \sup_{x\in\X} \bigl| \log \cg_t(\xi(x)) - \log \eff_t(\xi(x))  \bigr|
\end{equation}
In what follows we prove the claimed result by providing upper bounds on $\int_0^\delta g_s^2ds$ and $\int_\delta^t g_s^2ds$ separately, for some fixed $\delta\in(0,1)$. We use this splitting since we estimate the short and long-time behaviour of $g_t$ using different approaches. 

We first estimate $\int_0^\delta g_s^2 ds$. Using~\eqref{eq:C-init-est} we find
\begin{equation}\label{eq:short-time}
\int_0^\delta g^2_s \,ds \leq 8\|L\|^2_{\infty} \sup_{x\in\X} \Bigl[ \int_0^\delta |\log\cg_s(\xi(x))|^2\,ds  + \int_0^\delta |\log\eff_s(\xi(x))|^2\,ds  \Bigr].
\end{equation}
Since $\eta_t\in\mathcal P(\Y)$ solves~\eqref{eq:eff} with irreducible generator $N$, by Proposition~\ref{lem:forKol-low} there exists $c_1>0$ and $n_1\in\mathbb N$ independent of $x\in\X$ such that 
\begin{equation*}
\forall s\in [0,\delta], \ \forall x\in\X: \ \eff_s(\xi(x))\geq c_1s^{n_1}.  
\end{equation*}
We can also obtain modified constants $\tilde{c}_1, \tilde{n}_1$ independent of $N$ by applying Proposition~\ref{prop:lowerBounds_eff_dyn} instead of Proposition \ref{lem:forKol-low}.
Since $\eta_s(\xi(x)) \leq 1$ and $\theta\mapsto|\log\theta|$ is increasing as $\theta\rightarrow 0$, using the lower bound above it follows that
\begin{equation*}
|\log\eta_s(\xi(x))| \leq |\log c_1| + n_1|\log s|, 
\end{equation*}
and therefore we have the estimate
\begin{equation*}
\int_0^\delta |\log\eff_s(\xi(x))|^2 \,ds \leq \delta |\log c_1|^2 + n_1^2\int_0^\delta |\log s|^2ds  = \delta |\log c_1|^2 +n_1^2 \delta\bigl[ (\log\delta)^2-2\log\delta + 2\bigr].
\end{equation*}
Since $\mu$ solves~\eqref{eq:ForKol} with irreducible generator $L$, again by using Proposition~\ref{lem:forKol-low}, there exists $c_2>0$ and $n_2\in\mathbb N$ independent of $x\in\X$ such that 
\begin{equation*}
\forall s\in [0,\delta], \ \forall x\in\X: \ \mu_s(x)\geq c_2s^{n_2},
\end{equation*}
and therefore for any $x\in\X$ and  we have 
\begin{equation*}
\cg_s(\xi(x)) = \sum_{x'\in\Lambda_{\xi(x)}} \mu_s(x') \geq \mu_s(x') \geq c_2 s^{n_2}.
\end{equation*}
Repeating the same arguments as above we arrive at the bound
\begin{equation*}
\int_0^\delta |\log\cg_s(\xi(x))|^2 \,ds \leq  \delta |\log c_2|^2 +N_2^2\delta\bigl[ (\log\delta)^2-2\log\delta + 2\bigr].
\end{equation*}
Substituting these bounds back into~\eqref{eq:short-time} we find
$\int_0^\delta g_s^2 \,ds \leq C(c_1,c_2,n_1,n_2,\delta)$.
Similarly, using the modified constants $\tilde{c}_1$ and $\tilde{n}_1$ we obtain a constant $\tilde{C}$, which is independent of $N$, such that $\int_0^\delta g_s^2 \,ds \leq \tilde{C}$.

Next we estimate $\int_\delta^t g_s^2 ds$. Once more using~\eqref{eq:C-init-est} we find
\begin{equation}\label{eq:long-time}
\int_\delta^t g^2_s \,ds \leq 8\|L\|^2_{\infty} \sup_{x\in\X} \Bigl[ \int_\delta^t |\log\cg_s(\xi(x))-\log\xi_\#\stat(\xi(x)) |^2\,ds  + \int_\delta^t |\log\eff_s(\xi(x))-\log\xi_\#\stat(\xi(x))|^2\,ds  \Bigr].
\end{equation}
We first provide an estimate for the first term in the right hand side of~\eqref{eq:long-time}. Since $\stat\in\P_+(\X)$ and $\X$ is finite, using $c_{\stat}\coloneqq\min_{x\in\X} \stat(x) \in(0,1)$ we find
\begin{equation*}
\forall x\in\X: \ \xi_{\#}\stat(\xi(x)) = \sum_{x\in\Lambda_{\xi(x)}} \stat(x) \geq |\Lambda_{\xi(x)}| c_\stat \geq c_{\stat}.
\end{equation*}
Since $L$ is an irreducible generator, using Proposition~\ref{lem:forKol-exp}, there exists $\tau>\delta$ such that for any $s\geq\tau$ we have $\|\mu_s - \rho\|_{\TV} \leq \frac{1}{2} c_{\stat}$, and consequently for any $x\in\X$, $s\geq \tau$ and $x\in\X$ we have the lower bound
\begin{equation*}
\mu_s(x) \geq \frac12 c_\stat \ \Rightarrow \ \cg_s(\xi(x)) \geq \frac12 c_\stat.
\end{equation*}
Finally, using Proposition~\ref{lem:forKol-low} there exists $d\geq 0$ such that $\cg_s(\xi(x)) \geq d$ for any $s\in [\delta,\tau]$. In summary, for any $x\in\X$ and $s\in[\delta,\infty)$, there exists $\hat c\in (0,1)$ defined as $\hat c\coloneqq\min(c_\rho,c_\stat/2, d)$ such that $\cg_s(\xi(x)),\xi_\#\stat(\xi(x))\in [\hat{c},1]$. Since $\theta\mapsto \log\theta$ is Lipschitz continuous on $[\hat c,1]$ with constant $\hat c^{-1}$, for any $x\in\X$ and $s\in [\delta,\infty)$ we find
\begin{align*}
\bigl|\log\bigl(\cg_s(\xi(x))\bigr)-\log\bigl(\xi_\#\stat(\xi(x))\bigr) \bigr| &\leq \hat c^{-1} \bigl|\cg_s(\xi(x))-\xi_\#\stat(\xi(x)) \bigr|  \leq \hat c^{-1} \sum_{x\in \Lambda_{\xi(x)}} \bigl|\mu_s(x)-\stat(x) \bigr| \\
&\leq 2 \hat c^{-1} \|\mu_s -\stat\|_{\TV} \leq 2C e^{-\lambda s},
\end{align*}
for some $C,\lambda>0$, where we use Proposition~\ref{lem:forKol-exp} to arrive at the final inequality. Therefore, the first term in the right hand side of~\eqref{eq:long-time} admits the upper bound
\begin{equation*}
\sup_{x\in\X} \int_\delta^t |\log\cg_s(\xi(x))-\log\xi_\#\stat(\xi(x)) |^2\,ds \leq 4C^2 \int_{\delta}^\infty e^{-2\lambda s} \,ds \leq C,
\end{equation*}
where the constant on the right hand side of the final inequality is independent of $t>0$. Since the the effective dynamics $\eta_t$ admits $\xi_\#\stat$ as a stationary measure (see Proposition~\ref{prop:eff-inv}) and its generator is irreducible (see Lemma~\ref{lem:eff-irred}), we can repeat the same arguments as above to arrive at a time-independent upper bound for the second integral on the right-hand side of~\eqref{eq:long-time}, i.e.\  
\begin{equation*}
\sup_{x\in\X} \int_\delta^t |\log\eta_s(\xi(x))-\log\xi_\#\stat(\xi(x)) |^2\,ds \leq C.
\end{equation*}
Substituting these upper bounds back into~\eqref{eq:long-time} we arrive at the claimed result. 

To obtain a constant independent of $N$ and $\xi$ we fix $T < \infty$. Then Proposition \ref{prop:lowerBounds_eff_dyn} gives a constant $\tilde{d}(\delta,T)$ such that $\eta_s(\xi(x)) \geq \tilde{d}$. In particular, $\tilde{d}$ is independent of $N$ and $\xi$. Additionally, since $\xi_{\#}\rho(\xi(x)) \geq \min_{x \in \X} \rho(x) > 0$ we can bound
\begin{align*}
    \sup_{x \in \X} \int_\delta^t \left\vert \log(\eta_s(\xi(x)) - \log\xi_{\#}\rho(\xi(x))\right\vert^2 \,ds &\leq \sup_{x \in \X} \int_\delta^t |\log\eta_s(\xi(x))|^2 \,ds + \sup_{x \in \X} \int_\delta^t |\log\xi_{\#}\rho(\xi(x))|^2 \,ds \\
    &\leq \left(|\log(\tilde{d})|^2 + \max_{x \in \X} |\log(\rho(x))|^2\right) (T-\delta)
\end{align*}
for all $t \in [\delta,T]$. The final bound is independent of $N$ and $\xi$, which together with the modified short-time bound proves the second part of the lemma.
\end{proof}

Using the previous two auxiliary results we can now prove Theorem~\ref{thm:RelEntEst}.
\begin{proof}[Proof of Theorem~\ref{thm:RelEntEst}]
Using Lemma~\ref{lem:pre-FIR},~\ref{lem:bound-g}, applying Cauchy-Schwarz inequality in time and Jensen's inequality on $\Y$ we find
\begin{equation}\label{eq:main-auxeq1}
\RelEnt(\cg_t|\eff_t) \leq \RelEnt(\cg_0|\eff_0) + 2 \|g\|_{L^2(0,\infty)} \Bigl( \int_0^t \sum_{y\in\Y} \| \mu_s(\cdot|y)-\stat(\cdot|y)\|_{\TV}^2 \cg_s(y)\,ds  \Bigr)^{\frac12}, 
\end{equation}
where we have dropped the Fisher information term in~\eqref{eq:ent-bound1} as it is non-negative (see Lemma~\ref{lem:pos-FI}). 

Using the CKP inequality~\eqref{def:CKP} and the log-Sobolev inequality~\eqref{def:LSI}, the term inside the summation on the right hand side can be bounded by
\begin{equation*}
\| \mu_t(\cdot|y)-\stat(\cdot|y)\|_{\TV}^2 \leq 2\RelEnt\bigl(\mu_t(\cdot|y)\big|\stat(\cdot|y)\bigr)  \leq  \frac{2}{\alpha_{\LSI}} \RF_{L^y}\bigl(\mu_t(\cdot|y)\big|\stat(\cdot|y)\bigr),
\end{equation*}
and therefore we arrive at the bound 
\begin{align*}
&\sum_{y\in\Y} \| \mu_t(\cdot|y)-\stat(\cdot|y)\|_{\TV}^2 \cg_t(y) \leq \frac{2}{\alpha_{\LSI}} \sum_{y\in\Y}\RF_{L^y}\bigl(\mu_t(\cdot|y)\big|\stat(\cdot|y)\bigr) \cg_t(y)\\
&= \frac{2}{\alpha_{\LSI}} \sumsum_{y\in\Y,x_1,x_2\in \Lambda_{y}} L^y(x_1,x_2) \mu_t(x_1|y) \Bigl[ \frac{\mu_t(x_2|y) \stat(x_1|y)}{\stat(x_2|y) \mu_t(x_1|y)} - 1 - \log \Bigl(\frac{\mu_t(x_2|y) \stat(x_1|y)}{\stat(x_2|y) \mu_t(x_1|y)} \Bigr)  \Bigr]\cg_t(y)\\
&= \frac{2}{\alpha_{\LSI}} \sumsum_{y\in\Y,x_1,x_2\in \Lambda_{y}} L^y(x_1,x_2) \mu_t(x_1) \Bigl[ \frac{\mu_t(x_2) \stat(x_1)}{\stat(x_2) \mu_t(x_1)} - 1 - \log \Bigl(\frac{\mu_t(x_2) \stat(x_1)}{\stat(x_2) \mu_t(x_1)} \Bigr)  \Bigr]\\
&\leq \frac{2}{\alpha_{\LSI}} \RF_L(\mu_t|\stat).
\end{align*}
Here the first equality follows by using the alternate formulation~\eqref{def:FI-alter} of the Fisher information and the second equality follows since for any $y\in\Y$ and $x\in\Lambda_y$ we have $\stat(x)=\stat(x|y)\xi_\#\stat(y)$, $\mu_t(x)=\mu_t(x|y)\cg_t(y)$. The final inequality follows since $L^y$ is the restriction of $L$ to $\Lambda_y$ by definition and $\theta\mapsto \theta-1-\log\theta$ is non-negative for any $\theta> 0$. 

Substituting back into~\eqref{eq:main-auxeq1} and using~\eqref{eq:ent-decay} we find
\begin{align*}
\RelEnt(\cg_t|\eff_t) &\leq \RelEnt(\cg_0|\eff_0) + 2 \|g\|_{L^2(0,\infty)} \sqrt{\frac{2}{\alpha_{\LSI}}} \Bigl( \int_0^t  \RF(\mu_s|\stat)\,ds\Bigr)^{\frac12} \\
&\leq \RelEnt(\cg_0|\eff_0) + 2 \|g\|_{L^2(0,\infty)} \sqrt{\frac{2}{\alpha_{\LSI}}} \bigl[ \RelEnt(\mu_0|\rho)-\RelEnt(\mu_t|\rho) \bigr]^\frac12,
\end{align*}
which is the required result. 

Finally, for a fixed $T < \infty$ we can replace \eqref{eq:main-auxeq1} by
\begin{align*}
    \RelEnt(\cg_t|\eff_t) \leq \RelEnt(\cg_0|\eff_0) + 2 \|g\|_{L^2(0,T)} \Bigl( \int_0^t \sum_{y\in\Y} \| \mu_s(\cdot|y)-\stat(\cdot|y)\|_{\TV}^2 \cg_s(y)\,ds  \Bigr)^{\frac12}
\end{align*}
and then use the alternative upper bound on the $L^2$-norm of $g$ provided by Lemma \ref{lem:bound-g}, which provides an upper bound independent of $N$ and $\xi$. Proceeding then as above yields the alternative error estimate \eqref{eq:RelEntEstFiniteTime} for $t \in [0,T]$.
\end{proof}

\begin{rem}\label{rem:improvedConstant}
We finally outline the challenges in extending the estimate \eqref{eq:RelEntEstFiniteTime} to the case $T = \infty$.
The key part is to further improve the bound on $\|g\|_{L^2(0,T)}$. To do this we require a bound on the the decay rate of $\|\eta_t - \xi_\# \rho\|_{\TV}$, which does not depend on the effective generator $N$ and the coarse-graining map $\xi$. Revisiting the proof of Proposition \ref{lem:forKol-exp} this could for example be achieved by bounding the spectral gap of $N$ from below by a constant only depending on the full generator $L$. However, this remains an open question.
\end{rem}

\section{Estimates with explicit scale-separation}\label{sec:scale-sep}

In the last section we considered error estimates comparing the coarse-grained and effective dynamics for a given coarse-graining map. In this section we discuss the asymptotic behaviour of these estimates in the presence of explicit scale-separation. Specifically, we will focus on averaging problems for Markov chains (see for instance~\cite{PavliotisStuart08,lahbabiLegoll13,HilderPeletierSharmaTse20}). In Section~\ref{sec:eps-setup} we introduce these averaging problems and discuss the behaviour of the effective dynamics in the limit of infinite scale-separation. In Section~\ref{sec:eps-fixed} we present quantitative error estimates between the coarse-grained and effective dynamics and in Section~\ref{sec:eps-rev} we focus on the specific setting of reversible Markov chains. Finally, in Section \ref{sec:choice_CG_map} we discuss the scaling of the log-Sobolev constant $\alpha_{\LSI}$ with respect to $\varepsilon$ for different choices of $\xi$ in a simple example.

\subsection{Averaging problems for Markov chains and effective dynamics}\label{sec:eps-setup}
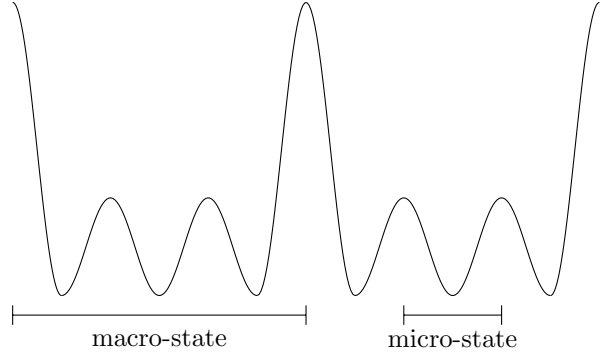
\begin{figure}[h!]
\begin{center}
    \begin{tikzpicture}[scale=1.3]
    \draw [domain=-2.5:-0.5, samples=100] plot (\x, {cos(pi*\x r)*cos(pi*\x r)});
    \draw [domain=0.5:2.5, samples=100] plot (\x, {cos(pi*\x r)*cos(pi*\x r)});
    \draw [domain=-3:-2.5, samples=50] plot(\x, {3*cos(pi*\x r)*cos(pi*\x r)});
    \draw [domain=2.5:3, samples=50] plot(\x, {3*cos(pi*\x r)*cos(pi*\x r)});
    \draw [domain=-0.5:0.5, samples=50] plot(\x, {3*cos(pi*\x r)*cos(pi*\x r)});
    \draw (-3,-0.1) -- (-3,-0.3) (-3,-0.2)--(-1.5,-0.2) node[align=center,below=2pt]{macro-state}--(0,-0.2) (0,-0.1) -- (0,-0.3);
    \draw (1,-0.1) -- (1,-0.3) (1,-0.2)--(1.5,-0.2) node[align=center,below=2pt]{micro-state}--(2,-0.2) (2,-0.1) -- (2,-0.3);
    \end{tikzpicture}
    \caption{Energy landscape with two macro-states.}
    \label{fig:energyLandscape}
\end{center}
\end{figure}

Although multiscale Markov chains are fairly classical and arise in various contexts, we briefly motivate them from the perspective of kinetic Monte-Carlo methods in molecular dynamics. Consider a particle moving in a potential-energy landscape, which consists of small and large barriers as described in Figure~\ref{fig:energyLandscape}. The large energy barriers  introduce a natural scale-separation since it is harder for the particle to jump across them  compared to the smaller barriers. More precisely we can model the behaviour of such a particle as a Markov jump process on $\X = \Y\times \Z$ where $\Y$ corresponds to the states separated by the large energy barriers while $\Z$ is the part of the state space separated by small energy barriers. For simplicity, we assume that there is only one large barrier, i.e.\ $\Y = \{0,1\}$ and finitely many small barriers corresponding to each of these large barriers, i.e.\ $\Z = \{0, \dots, n-1\}$. This intuitively means that the state space is divided up into two \emph{macro-states}, each of 
which contain $n\in\N$ easily accessible \emph{micro-states}.  

To make these ideas concrete, consider a family of forward Kolmogorov equations parameterised by $\e>0$
\begin{equation}\label{eq:eps-KolEq}
\begin{dcases}
\partial_t \mu^\e = (L^\e)^T \mu^\e, \\
\mu^\e_{t = 0} = \mu_0,
\end{dcases}
\end{equation}
on $\X = \Y \times \Z$ with $\Y = \{0,1\}$ and $\Z = \{0, \dots, n-1\}$, generated by the family of generators ($Q_i,D_i, G_{i,1-i}\in\R^{|\Z|\times|Z|}$ for $i=0,1$)
\begin{equation}\label{eq:CG-res-gen}
L^\e= \frac{1}{\e}Q + \sG \coloneqq \frac{1}{\e}\begin{pmatrix}Q_0 & 0 \\ 0 & Q_1\end{pmatrix} +  \begin{pmatrix}D_0 & \sG_{0,1} \\
\sG_{1,0} & D_1 \end{pmatrix},
\end{equation}
i.e.\ with
\[
Q((y,z),(y',z')) = \begin{cases}
Q_y(z,z') &\text{if\, $y'=y$}\\
0 &\text{otherwise}
\end{cases},\qquad \sG((y,z),(y',z')) = \begin{cases}
\sG_{y,y'}(z,z') &\text{if\, $y'\ne y$}\\
D_y(z) & \text{if $y' = y$ and $z'=z$} \\ 
0 &\text{otherwise}
\end{cases},
\]
for $x=(y,z)$, $x'=(y',z')\in\X$ satisfying
\begin{equation}\label{CG-L-ass}
\forall x\in\X: \ \sum_{x'\in\X}Q(x,x')=0=\sum_{x'\in\X} \sG(x,x'),
\end{equation}
and diagonal matrix $D_y$, $y\in\Y$, which satisfies 
\begin{equation}\label{eq-def:D}
\forall z\in\Z: \ D_y(z)\coloneqq-\sum_{z'\in\Z} \sG_{y,1-y}(z,z').
\end{equation}
We assume that $L^\e$ is irreducible, and therefore~\eqref{eq:eps-KolEq} admits a stationary solution $\stat^\e\in\P_+(\X)$. Additionally we assume that $Q_0$ and $Q_1$ are irreducible generators as well. Consequently, the dynamics driven by $Q_i$ for $i=0,1$ admits a stationary measure $\stat_i \in\P_+(\Z)$.

Now let us take a closer look at each of these components. The small parameter $\varepsilon > 0$ models the scale-separation arising due to the difference in the heights of the barriers. The matrix $Q_y\in\R^{n\times n}$ encodes the jumps  between micro-states within the $y$-th macro-state. The matrix $G_{y,1-y}\in\R^{n\times n}$ encodes the transition from the $y$-th macro-state to $(1-y)$-th macro-state. The summability condition~\eqref{CG-L-ass} ensures that $L^\e$ is a generator.

This setting is of coarse-graining type, in that, for $0<\e\ll 1$ the dynamics in the macro-state equilibrates and the limit is a jump process on $\Y$. Consequently, the natural coarse-graining map in this setting is the projection onto the slow-variable, i.e.\  
\begin{equation*}
\xi:\X\rightarrow\Y \ \ \text{with} \ \ \xi (x) = y, \ \text{for any } x=(y,z)\in\X. 
\end{equation*}

\begin{rem}
Since  $\X = \Y \times \Z$ and $\xi$ is the projection onto the slow-variable, we simplify our notation by identifying the conditional measures on the level sets $\Lambda_y$ with probability measures on $\Z$.
In a slight abuse of notation, we will often denote both measures by the same symbol whenever the meaning is clear from the context.
For instance, we write $\rho^\e(z \vert y)$ instead of $\rho^\e(x \vert y)$ where $x=(y,z)$.	
\end{rem}

With the generator $L^\e$ and the coarse-graining map $\xi$, we can construct the corresponding coarse-grained and the effective dynamics. For $\e>0$, the coarse-grained dynamics $t\mapsto \cg_t \in\P(\Y)$ solves 
\begin{equation}\label{eq:cg-eps}
\partial_t \cg_t^\e = \bigl(\hat L_t^\e \bigr)^T\cg_t, \ \ \text{ with } \ \ \hat L^\e_t(y_1,y_2)\coloneqq \sumsum_{x_1\in\Lambda_{y_1},x_2\in\Lambda_{y_2} } L^\e(x_1,x_2)\mu_t^\e(x_1|y_1),
\end{equation}
and the effective dynamics $t\mapsto \eff_t \in\P(\Y)$ solves
\begin{equation}\label{eq:eff-eps}
\partial_t \eff_t^\e = \bigl( N^\e \bigr)^T\eff_t, \ \ \text{ with } \ \  N^\e(y_1,y_2)\coloneqq \sumsum_{x_1\in\Lambda_{y_1},x_2\in\Lambda_{y_2} } L^\e(x_1,x_2)\stat^\e(x_1|y_1),
\end{equation}
where $\mu_t^\e(\cdot|y), \stat^\e(\cdot|y)\in\P(\Lambda_y)$ are the conditional measures corresponding to $\mu^\e_t, \stat^\e$ respectively (see~\eqref{def:marg-cond} for definition) and $\Lambda_y = \{y\} \times \Z$ is the $y$-level set of $\xi$.

In what follows we will make use of the the limiting dynamics of~\eqref{eq:eps-KolEq} as $\e\rightarrow 0$. Specifically, it can be shown (see~\cite[Section 3]{HilderPeletierSharmaTse20} for a proof) that the coarse-grained dynamics converges to the so-called \emph{averaged dynamics}. 
\begin{theorem}[Classical averaging]\label{thm:classical_averaging}
Define the limiting generator $L^{\av}\in \mathbb R^{|\Y|\times |\Y|}$ as 
\begin{equation}\label{eq:aver-gen}
L^{\av}\coloneqq \begin{pmatrix}
-\lambda_0 & \lambda_0  \\ \lambda_1 & -\lambda_1
\end{pmatrix}, \ \ 
\lambda_y \coloneqq \sumsum_{z_1,z_2\in\Z} \stat_y(z_1) \sG_{y,1-y}(z_1,z_2),
\end{equation}
where $\stat_y$ is the stationary solution corresponding to $Q_y$. The solution to the coarse-grained dynamics $\cg^\e \in C^1([0,T];\P(\Y))$ (see~\eqref{eq:cg-eps}) converges with respect to the uniform topology in time and strong topology on $\P(\Y)$ to $\aver\in C^1([0,T];\P(\Y))$ which solves the averaged dynamics 
\begin{equation}	\label{eq:averaged_dynamics}
\partial_t \aver_t = \bigl( L^{\av}\bigr)^T\aver_t.
\end{equation}
\end{theorem}
While here we have phrased the averaging result in terms of the forward Kolmogorov equations to be consistent with rest of the article, these convergence results can be considerably generalised (see for instance~\cite[Chapter 16]{PavliotisStuart08} for pathwise convergence). 

Note that in general, the effective dynamics~\eqref{eq:eff-eps} is different from the averaged dynamics, which is easily seen for instance since the effective dynamics explicitly depends on $\e$, while the averaged dynamics does not. This is due to a fundamental difference between the two approaches -- classical averaging requires that the fast dynamics characterised by $Q_y$~\eqref{eq:eps-KolEq} has an stationary measure (for a fixed value of the slow variables), whereas the effective dynamics relies on the existence of an stationary measure for the full system characterised by $L^\e$ (which depends on $\e$). However, in the restrictive setting of reversible Markov chains the two dynamics coincide (see Section~\ref{sec:eps-rev}).

A natural question is to understand the behaviour of the effective dynamics as $\e\rightarrow 0$ and how it compares to the averaged dynamics. In Theorem~\ref{prop:eff-to-aver-eps} below we provide a quantitative error estimate between these two. To prove this error estimate we will need the following result which characterises the behaviour of the stationary measure $\stat^\e$ in the limit $\e\rightarrow 0$ (see~\cite[Lemma 3.3]{HilderPeletierSharmaTse20} for proof). 

\begin{lem}\label{lem:stat-limit-prop} Let $\stat^\e\in\P_+(\X)$ be a sequence of stationary measures corresponding to the generator $L^\e$~\eqref{eq:CG-res-gen}, i.e.\ $(L^\e)^T\stat^\e=0$. Then there exists a positive probability measure $\stat\in\P_+(\X)$ such that $\stat^\e\rightarrow \stat$ as $\e\rightarrow 0$. Furthermore, the conditional stationary measures satisfies $\stat^\e(\cdot|y)\rightarrow \stat_y(\cdot)$ as $\e\rightarrow 0$, where $\rho_y\in \P(\Z)$ satisfies $Q_y^T\rho_y=0$. Consequently, the marginal stationary measures converge to the limiting marginal stationary measure, i.e. $\xi_\#\stat^\e\rightarrow \xi_\#\stat$ as $\e\rightarrow 0$.   
\end{lem}

\begin{theorem}\label{prop:eff-to-aver-eps} 
For any $T < \infty$ and $\e\in(0,\e_0)$ there exists a constant $C < \infty$ independent of $\e>0$ such that
\begin{equation*}
\sup_{t \in [0,T]} \|\eta^\e_t - \aver_t\|_{\TV} \leq C\left(\|\eta_0^\e - \aver_0\|_{\TV} + \varepsilon\right)
\end{equation*}
In particular, if $\eta_0^\e \rightarrow \aver_0$ in $\P(\Y)$ as $\e \rightarrow 0$, then $\eta^\e \rightarrow \aver$ in $C([0,T];\P(\Y))$ as $\e \rightarrow 0$.  
\end{theorem}
\begin{proof}
Using the equivalence of the total-variation norm and the Euclidean norm on a finite state space, we have the estimate 
\begin{equation}\label{eq:eff-ave-aux1}
\begin{aligned}
\|\eta^\e - \mu^{\av}\|_{C([0,T];\P(\Y))} &= \sup_{t \in [0,T]} \|\eta_t^\e - \mu^{\av}_t\|_{\TV}= \sup_{t \in [0,T]} \Bigl\|e^{t(N^\e)^T}\eta_{t=0}^\e - e^{t(L^{\av})^T}\aver_{t=0}\Bigr\|_{\TV} \\
&\leq C\Bigl[\sup_{t \in [0,T]} \Bigl\|\bigl(e^{t(N^\e)^T} - e^{t(L^{\av})^T}\bigr) \eta^\e_{t=0}\Bigr\|_{\R^{|\Y|}} + \sup_{t \in [0,T]} \Bigl\|e^{t(L^{\av})^T} (\eta^\e_{t=0} - \aver_{t=0})\Bigr\|_{\R^{|\Y|}}\Bigr].
\end{aligned}
\end{equation}
We now estimate both terms on the right-hand side separately.
For the second term we find
\begin{equation}\label{eq:eff-ave-aux2}
\sup_{t \in [0,T]} \Bigl\|e^{t(L^{\av})^T} (\eta^\e_{t=0} - \aver_{t=0})\Bigr\|_{\R^{|\Y|}} \leq \Bigl(\sup_{t \in [0,T]} \|e^{t(L^{\av})^T}\|_{\R^{|\Y| \times |\Y|}}\Bigr) \|\eta^\e_{t=0} - \aver_{t=0}\|_{\R^{|\Y|}} \leq C \|\eta^\e_{t=0} - \aver_{t=0}\|_{\TV}
\end{equation}
with a constant $C = C(L^{\av},T)$ independent of $\e$. For the first term we make use of the following matrix inequality for $M_1, M_2 \in \R^{|\Y| \times |\Y|}$
\begin{align*}
\|e^{M_1 + M_2} - e^{M_1} \|_{\R^{|\Y| \times |\Y|}} \leq \|M_1\|_{\R^{|\Y| \times |\Y|}} e^{\|M_1\|_{\R^{|\Y| \times |\Y|}}} e^{\|M_2\|_{\R^{|\Y| \times |\Y|}}}.
\end{align*}
Choosing $M_1 = t (L^{\av})^T$ and $M_2 = t((N^\e)^T - (L^{\av})^T)$ yields
\begin{equation}\label{eq:eff-ave-aux3}
\Bigl\|e^{t(N^\e)^T} - e^{t(L^{\av})^T}\Bigr\|_{\R^{|\Y| \times |\Y|}} \leq |t| \left\|(N^\e)^T - (L^{\av})^T\right\|_{\R^{|\Y| \times |\Y|}} e^{|t|\|(L^{\av})^T\|_{\R^{|\Y| \times |\Y|}}} e^{|t|\|(N^\e)^T - (L^{\av})^T\|_{\R^{|\Y| \times |\Y|}}}.
\end{equation}
In what follows we will show that there exists a constant $C$ independent of $\e$ such that 
\begin{equation}	\label{eq:genConvergenceEffToAver}
\|(N^\e)^T - (L^{\av})^T\|_{\R^{|\Y| \times |\Y|}} \leq C \varepsilon.
\end{equation}
Applying this bound to~\eqref{eq:eff-ave-aux3} and substituting along with~\eqref{eq:eff-ave-aux2} back into~\eqref{eq:eff-ave-aux1} we arrive at the required quantitative result. 
Furthermore, if the initial data converges, the upper bound vanishes as $\e \rightarrow 0$, which proves the remaining part of the proposition.

In the remainder of this proof we show~\eqref{eq:genConvergenceEffToAver}. 
Using $\Lambda_y = \{y\} \times \Z$, for any $y_1,y_2 \in \Y$ we can rewrite the effective generator $N^\e$ as
\begin{equation*}
N^\e(y_1,y_2) = \dfrac{1}{\e}\sumsum_{z_1,z_2 \in \Z} Q((y_1,z_1),(y_2,z_2)) \rho^\e(z_1 \vert y_1) + \sumsum_{z_1,z_2 \in \Z} \sG((y_1,z_1),(y_2,z_2)) \rho^\e(z_1 \vert y_1)
\end{equation*}
Note that the first sum vanishes if $y_1 \neq y_2$ since $Q((y_1,z_1),(y_2,z_2)) = 0$, and if $y_1 = y_2 = y$ we have
\begin{equation}\label{eq:eps-Q-effect}
\sumsum_{z_1,z_2 \in \Z} Q((y,z_1),(y,z_2)) \rho^\e(z_1\vert y) = \sum_{z_1 \in \Z} \rho^\e(z_1 \vert y) \sum_{z_2 \in \Z} Q((y,z_1),(y,z_2)) = 0.
\end{equation}
\sloppy{Here we have used $\sum_{z_2 \in \Z} Q((y,z_1),(y,z_2)) = \sum_{x_2 \in \X} Q((y,z_1),x_2) = 0$, which follows again  since $Q((y_1,z_1),(y_2,z_2)) = 0$ for $y_1 \neq y_2$.} 
Therefore for $y \in \Y$ we find 
\begin{equation*}
N^\varepsilon(y,1-y) = \sumsum_{z_1,z_2 \in \Z} \rho^\e(z_1 \vert y) \sG_{y,1-y}(z_1,z_2),
\end{equation*}
i.e.\ the explicit scale-separation parameter drops out from the effective generator due to the definition of $\xi$. Note that the diagonal terms $N^\e(y,y)$ have a similar form as above due to the definition~\eqref{eq-def:D} of $D_y$.
Using the definition of the limiting generator $L^{\av}$ we obtain the estimate
\begin{equation}\label{eq:eff-ave-aux4}
\|(N^\e)^T - (L^{\av})^T\|_{\R^{|\Y| \times |\Y|}} \leq \sqrt{|\Z|}\|\sG\|_{\R^{|\Z| \times |\Z|}} \sup_{y \in \Y} \|\rho^\e(\cdot \vert y) - \rho_y(\cdot)\|_{\R^{|\Z|}}
\end{equation}
and therefore we need to estimate  $\|\rho^\e(\cdot \vert y) - \rho_y(\cdot)\|_{\R^{|\Z|}}$ as $\e \rightarrow 0$ uniformly in $y \in \Y$.

Fix $y \in \Y$. Since $\rho^\varepsilon$ is the stationary measure of $(L^\varepsilon)^T$, for any $z \in \Z$ we find
\begin{align}
0 &= ((L^\varepsilon)^T \rho^\varepsilon)((y,z)) = \sum_{x_2 \in \X} L^\varepsilon(x_2,(y,z)) \rho^\varepsilon(x_2) \nonumber \\
&= \dfrac{1}{\varepsilon} \sum_{z_2 \in \Z} Q_y(z_2, z) \rho^\varepsilon((y,z_2)) + \sum_{z_2 \in \Z} \sG_{1-y,y}(z_2,z) \rho^\varepsilon((1-y,z_2)) - \sum_{z_2\in\Z} \sG_{y,1-y}(z,z_2)\stat^\e((y,z)), \label{eq:effToAveraged_aux1}
\end{align}
where the final negative term arises due to $D_y$~\eqref{eq-def:D}
By Proposition \ref{prop:eff-inv}, the stationary measure of the effective dynamics \eqref{eq:eff-eps} is given by $\xi_{\#}\rho^\varepsilon$. 
Lemma~\ref{lem:stat-limit-prop} states that $\rho^\e(\cdot|y)\rightarrow \stat_y(\cdot)$ as $\e\rightarrow 0 $ where $\stat_y$ is a positive measure due to the irreducibility of $Q_y$. Therefore, we can divide \eqref{eq:effToAveraged_aux1} by $\xi_{\#}\rho^\varepsilon(y)$  and obtain
\begin{align*}
0 &= \sum_{z_2 \in \Z} \dfrac{1}{\varepsilon} Q_y(z_2, z) \rho^\varepsilon(z_2 \vert y) + \sum_{z_2 \in \Z} \sG_{1-y,y}(z_2,z) \dfrac{\rho^\varepsilon((1-y,z_2))}{\xi_\# \rho^\varepsilon(y)} - \sum_{z_2 \in \Z} \sG_{y,1-y}(z,z_2) \dfrac{\rho^\varepsilon((y,z))}{\xi_\# \rho^\varepsilon(y)}\\  
&=: \dfrac{1}{\varepsilon} (Q^T_y \rho^\varepsilon(\cdot \vert y))(z) + \curlI^\varepsilon(z;y),
\end{align*} 
where we use the definition of the conditional measure \eqref{def:marg-cond}.
We point out that since $\xi_\# \rho^\varepsilon$ converges to a positive measure, $\curlI^\varepsilon(z;y)$ can be bounded independently of $\varepsilon, y, z$.
Multiplying the equation above  by $\varepsilon$ yields
\begin{equation}	\label{eq:effToAveraged_aux2}
Q^T_y\rho^\varepsilon(\cdot \vert y) = - \varepsilon \curlI^\varepsilon(\cdot; y).
\end{equation}
Note that $0$ is an eigenvalue for $Q_y$ (as it admits $\stat_y$ as a stationary measure) and the real parts of all the eigenvalues of $Q_y$ lie in $[-2r_y,0]$ due to the Gerschgorin's circle theorem. Next, we define the non-negative irreducible matrix $P_y \coloneqq Q_y + (\max_{z \in \Z} \vert Q_y(z,z)\vert) I$, where $I$ is the identity matrix on $\R^{|\Z| \times |\Z|}$. Clearly $r_y$ is an eigenvalue of $P_y$ and the real parts of all the eigenvalues of $Q_y$ lie in $[-r_y,r_y]$. Thus $r_y$ is the spectral radius of $P_y$ and by the Perron-Frobenius theorem $r_y$ is a simple eigenvalue of $P_y$ (with algebraic and geometric multiplicity one), and therefore $0$ is a simple eigenvalue for $Q_y$, i.e.\ the eigenspace corresponding to $0$ is one-dimensional and spanned by $\stat_y$. 

We define the spectral projection $\mathbb{P}_{\neq 0} : \R^{|\Z|} \rightarrow \R^{|\Z|}$ onto the non-zero eigenspaces of $Q_y^T$.
Applying the projection to \eqref{eq:effToAveraged_aux2} and using $\mathbb{P}_{\neq 0}Q_y^T=Q_y^T\mathbb{P}_{\neq 0}$ we find
\begin{align*}
Q_y^T\mathbb{P}_{\neq 0}\rho^\varepsilon(\cdot \vert y) = - \varepsilon \mathbb{P}_{\neq 0} \curlI^\varepsilon(\cdot;y).
\end{align*}
Since $Q_y^T$ is invertible on $\operatorname{Ran}(\mathbb{P}_{\neq 0})$, the range of $\mathbb{P}_{\neq 0}$, we find 
\begin{equation*}
\|\mathbb{P}_{\neq 0} \rho^\varepsilon(\cdot \vert y)\|_{\R^{|\Z|}} \leq \varepsilon \|(Q_y^T)^{-1} \mathbb{P}_{\neq 0}\|_{\R^{|\Z|\times |\Z|}} \|\curlI^\varepsilon(\cdot;y)\|_{\R^{|\Z|}} \leq C \varepsilon,
\end{equation*}
for some constant $C$ independent of $\varepsilon,y,z$, where we have used a uniform lower bound on $\xi_\#\stat^\e$ due to Lemma~\ref{lem:stat-limit-prop}. 
Next, exploiting that $\operatorname{Ran}(I-\mathbb{P}_{\neq 0}) = \operatorname{Span}(\rho_y)$ we write
\begin{align*}
\rho^\varepsilon(\cdot \vert y) - \rho_y(\cdot) = (I-\mathbb{P}_{\neq 0}) \rho^\varepsilon(\cdot \vert y) - \rho_y + \mathbb{P}_{\neq 0}\rho^\varepsilon(\cdot \vert y) =: (\nu(\varepsilon,y) - 1) \rho_y + \mathbb{P}_{\neq 0}\rho^\varepsilon(\cdot \vert y),
\end{align*}
with $\nu(\varepsilon,y) \in \R$. Summing over all elements in $\Z$, using $\sum_{z \in \Z} [\rho^\varepsilon(\cdot \vert y) - \rho_y(\cdot)] = 0$ and $\rho_y \in \P(\Z)$ then yields
\begin{equation*}
|\nu(\varepsilon,y) - 1| = \Bigl\vert\sum_{z \in \Z} \mathbb{P}_{\neq 0}\rho^\varepsilon(\cdot \vert y)\Bigr\vert \leq C \|\mathbb{P}_{\neq 0} \rho^\varepsilon(\cdot \vert y)\|_{\R^{|\Z|}} \leq C \varepsilon.
\end{equation*}
This proves (using norm equivalence on $\R^{|\Z|}$) that $\|\rho^\varepsilon(\cdot\vert y) - \rho_y\|_{\TV} \leq C \varepsilon$ with a constant $C>0$ independent of $\varepsilon$ and substituting into~\eqref{eq:eff-ave-aux4} we arrive at the required bound~\eqref{eq:genConvergenceEffToAver}. 
\end{proof}

\subsection[Error estimates for fixed scale-separation parameter]{Error estimates for fixed $\e>0$}\label{sec:eps-fixed}

In this section we provide quantitative error estimates comparing the coarse-grained~\eqref{eq:cg-eps} and effective dynamics~\eqref{eq:eff-eps} in the $\e$-dependent setting which mirrors the general result in Theorem~\ref{thm:RelEntEst}. As in the general setting, we assume a log-Sobolev inequality on the level sets of $\xi$.
More precisely, we will assume that there exists a $\varepsilon_0 > 0$ and a constant $\alpha_{\LSI} > 0$ such that for any $\varepsilon \in (0,\e_0)$, $y \in \Y$ and $\nu \in \P(\Z)$ we have the estimate 
\begin{equation}\label{eq:LSIeps}\tag{$\LSI_\e$}
\RelEnt(\nu \vert \rho^\varepsilon(\cdot\vert y)) \leq \dfrac{1}{\alpha_{\LSI}} \RF_{Q_y}(\nu \vert \rho^\varepsilon(\cdot \vert y)).
\end{equation}
We note that the uniformity of $\alpha_{\LSI}$ with respect to $\e$ is essentially an assumption that the constant does not blow up as $\e\rightarrow 0$. 
However, since $\rho^\e(\cdot \vert y)$ converges to $\rho_y$  (stationary measure of $Q_y$) as $\e\rightarrow 0$ (see Lemma \ref{lem:stat-limit-prop}) and both relative entropy and relative Fisher information are continuous in the finite setting, the estimate \eqref{eq:LSIeps} reduces to the logarithmic Sobolev inequality as discussed for example in \cite{bobkovTetali06,zhang16} and the references therein. See Section~\ref{sec:choice_CG_map} for a detailed discussion on the log-Sobolev assumption in the case of a reversible toy-problem.

We now state the main error estimate in the presence of explicit scale-separation. 

\begin{theorem}\label{thm:eps-est-positive-init}
Let $\cg^\e, \eta^\e \in C^1([0,T];\P(\Y))$ be the solutions to coarse-grained dynamics \eqref{eq:cg-eps} and the effective dynamics \eqref{eq:eff-eps} with initial data $\cg^\e_0$ and $\eta_0^\e$ respectively. Assume that
\begin{enumerate}[label=({A}\arabic*)]
\item\label{ass-quantres1} The initial data for coarse-grained and effective dynamics is positive, i.e.\ there exists $c_0>0$ independent of $\e$ such that $\cg_0^\e(y),\eff_0^\e(y) > c_0$ for any $y\in\Y$. Furthermore, the initial data $\mu^\e_0$ converges in $\P(\X)$ and $\cg_0^\e,\eff_0^\e$ converge  in $\P(\Y)$ as $\e\rightarrow 0$.  
\item\label{ass-quantres2} There exists $\e_0 > 0$ and $\alpha_{\LSI} > 0$ such that \eqref{eq:LSIeps} holds.    
\end{enumerate}
Then there exists $C>0$ such that for any $\e \in (0,\e_0)$
\begin{equation}		\label{eq:quantEstEps}
    \sup_{t \in [0,T]} \RelEnt(\cg_t^\e \vert \eta_t^\e) \leq 2 \RelEnt(\cg_0^\e \vert \eta_0^\e) + C \dfrac{\e T}{\alpha_{\LSI}} \left[\RelEnt(\mu_0^\e \vert \rho^\e) - \RelEnt(\mu_T^\e \vert \rho^\e)\right].
\end{equation}
If additionally there exists $C_0 < \infty$ independent of $\e$ such that for any $\e \in (0,\e_0)$ we have 
\begin{equation}\label{ref:ass-improved-rate}
    \RelEnt(\cg_0^\e \vert \eta_0^\e) \leq C_0\e, 
\end{equation} 
then there exists a constant $D=D(T)>0$ independent of $\e$ and $\alpha_{\LSI}$ such that 
\begin{equation}\label{eq:quantEstEpsImproved}
    \sup_{t \in [0,T]} \RelEnt(\cg^\e_t \vert \eta^\e_t) \leq \biggl( 2C_0 + \frac{D}{\alpha_{\LSI}} \biggr) \e. 
\end{equation}
In particular, the estimate~\eqref{eq:quantEstEpsImproved} holds when $\cg^\e_0=\eff^\e_0$.
\end{theorem}

\begin{proof}
Repeating the arguments as in the proof of Lemma~\ref{lem:pre-FIR} and using the positivity of the Fisher information (see Lemma~\ref{lem:pos-FI}) we arrive at
\begin{align}
\RelEnt(\cg_t^\varepsilon|\eff_t^\varepsilon) - \RelEnt(\cg_0^\varepsilon|\eff_0^\varepsilon) &\leq 2\int_0^t g_s \Bigl[\sum_{y\in \Y} \|\mu_s^\varepsilon(\cdot|y) - \stat^\varepsilon(\cdot|y) \|_{\TV}\, \cg_s^\varepsilon(y) \Bigr] ds \nonumber\\ 
& \leq 2\|g\|_{L^2((0,T))} \Bigl(  \int_0^t \sum_{y\in\Y} \| \mu^\e_s(\cdot|y)-\stat^\e(\cdot|y)\|_{\TV}^2 \cg^\e_s(y)\,ds  \Bigr)^{\frac12}, 
\label{eq:asymp-pre-est}
\end{align}
where $g_s$ is defined in \eqref{def:C_t} (now with the explicit dependence on $\e$) and the second inequality follows by applying the Cauchy-Schwarz inequality in time. 

We first prove that there exists an $\e$-independent constant $C >0$ such that 
\begin{equation}\label{eq:g_t_integral_bound_eps}
    \|g\|_{L^2(0,T)} \leq C \sqrt{T} \sup_{t \in [0,T]} \|\cg_t^\e - \eta_t^\e\|_{\TV}.  
\end{equation}
Using $\xi((y,z))=y$ and the decomposition~\eqref{eq:CG-res-gen} of the generator, $f_t$ defined in \eqref{def:C_t} can be rewritten as
\begin{align*}
    f_t((z_1,y_1)) &= \sum_{(z_2,y_2) \in \Z \times \Y} L^\e((z_1,y_1),(z_2,y_2)) \Bigl[\log\Bigl(\dfrac{\cg_t^\e(y_1)}{\eta_t^\e(y_1)}\Bigr) - \log\Bigl(\dfrac{\cg_t^\e(y_2)}{\eta_t^\e(y_2)}\Bigr)\Bigr] \\
    &= \sum_{z_2 \in \Z} \sG_{y_1,1-y_1}(z_1,z_2) \Bigl[\log\Bigl(\dfrac{\cg_t^\e(y_1)}{\eta_t^\e(y_1)}\Bigr) - \log\Bigl(\dfrac{\cg_t^\e(1-y_1)}{\eta_t^\e(1-y_1)}\Bigr)\Bigr].
\end{align*}
Therefore, $g_t = \sup_{(y,z)\in\Y\times\Z} f_t((y,z))$ satisfies
\begin{equation}\label{eq:g_t_bound_eps}
    g_t \leq 2 \Bigl(\sup_{y\in\Y}\|\sG_{y,1-y}\|_\infty \Bigr) \sup_{y \in \Y}\left\vert \log\cg_t^\e(y) - \log\eta_t^\e(y)\right\vert.
\end{equation}
By Theorem~\ref{thm:classical_averaging},~\ref{prop:eff-to-aver-eps} and since $\cg_0^\e,\eta_0^\e$ have a limit in $\P(\Y)$ as $\e \rightarrow 0$, both $\cg^\e$ and $\eta^\e$ converge uniformly in time to solutions of the averaged dynamics \eqref{eq:averaged_dynamics} with positive initial data.
Repeating the proof of Lemma~\ref{lem:eff-irred} along with the positivity of $\rho_y$ for every $y\in\Y$ (since $Q_y$ is irreducible), it follows that the generator $L^{\av}$ of the averaged dynamics is irreducible. Therefore we can apply Proposition~\ref{lem:forKol-low} to obtain a uniform and positive lower bound on any solution of \eqref{eq:averaged_dynamics} with positive initial data. Since the coarse-grained and effective dynamics converge to the averaged dynamics, there exists a constant $c_0 > 0$, which is independent of $\e \in (0,\e_0)$ such that
\begin{align*}
    \inf_{t \in [0,T], y \in \Y} \cg_t^\e(y) \geq c_0, \text{ and } \inf_{t \in [0,T], y \in \Y} \eta_t^\e(y) \geq c_0.
\end{align*}
Since the logarithm is uniformly Lipschitz continuous on $[c_0,1]$ with constant $c_0^{-1}$ we arrive at
\begin{align*}
    \sup_{y \in \Y} \left\vert \log\cg_t^\e(y) - \log\eta_t^\e(y)\right\vert \leq c_0^{-1} \|\cg_t^\e - \eta_t^\e\|_{\TV}.
\end{align*}
Inserting this into \eqref{eq:g_t_bound_eps} and integrating over $[0,t]$ (with $t\in[0,T]$) leads to 
\begin{equation*}
    \int_0^t g_s^2 \,dt \leq 4\Bigl(\sup_{y\in\Y}\|\sG_{y,1-y}\|_\infty \Bigr)^2 c_0^{-2} T \sup_{s \in [0,T]} \|\cg_s^\e - \eta_s^\e\|_{\TV}^2.
\end{equation*}
This proves \eqref{eq:g_t_integral_bound_eps}. 

Now we provide an estimate for the integral term in the right hand side of~\eqref{eq:asymp-pre-est}. Using~\eqref{eq:LSIeps} and the linearity of $\RF_M$ in $M$ we find 
\begin{equation*}
    \|\mu_t^\e(\cdot \vert y) - \rho^\e(\cdot \vert y)\|_{\TV}^2 \leq 2 \RelEnt(\mu_t^\e(\cdot \vert y) \vert \rho^\e(\cdot \vert y)) \leq \dfrac{2}{\alpha_{\LSI}} \RF_{Q_{y}}(\mu_t^\e(\cdot \vert y) \vert \rho^\e(\cdot \vert y)) = \dfrac{2 \e}{\alpha_{\LSI}} \RF_{\e^{-1}Q_{y}}(\mu_t^\e(\cdot \vert y) \vert \rho^\e(\cdot \vert y)).
\end{equation*}
Since $(L^\e)^y = \e^{-1} Q_y$ is the restriction of $L^\e$ to the level set $\Lambda_y = \{y\} \times \Z$, repeating the arguments as in the proof of Theorem~\ref{thm:RelEntEst} we arrive at 
\begin{equation*}
\int_0^t \sum_{y\in\Y} \| \mu^\e_s(\cdot|y)-\stat^\e(\cdot|y)\|_{\TV}^2 \cg^\e_s(y)\,ds \leq \dfrac{2 \e}{\alpha_{\LSI}} \int_0^t \RF_{L^\e}(\mu_s^\e|\stat^\e) = \dfrac{2 \e}{\alpha_{\LSI}} \bigl[ \RelEnt(\mu^\e_0|\stat^\e)-\RelEnt(\mu^\e_t|\stat^\e) \bigr].
\end{equation*}
Substituting this bound along with~\eqref{eq:g_t_integral_bound_eps} back into~\eqref{eq:asymp-pre-est} we arrive at 
\begin{equation}\label{eq:asym-first-est}
\RelEnt(\cg^\e_t|\eff^\e_t) - \RelEnt(\cg^\e_0|\eff^\e_0)\leq C\sqrt{\dfrac{\e T}{\alpha_{\LSI}}} \sup_{t\in [0,T]}\|\cg_t^\e -\eff_t^\e\|_{\TV} \left[\RelEnt(\mu_0^\e \vert \rho^\e) - \RelEnt(\mu_t^\e \vert \rho^\e)\right]^\frac{1}{2}.
\end{equation}
Using $h^\e(t)\coloneqq \sup_{t\in [0,T]}\RelEnt(\hat\mu^\e_t|\eta^\e_t)$ along with the CKP inequality~\eqref{def:CKP} in~\eqref{eq:asym-first-est} we then find
\begin{align*}
    h^\e(t) &\leq \RelEnt(\cg^\e_0|\eff^\e_0) + C \sqrt{\frac{2\e T}{\alpha_{\LSI}}} \sqrt{h^\e(t)} \sqrt{\RelEnt(\mu_0^\e \vert \rho^\e) - \RelEnt(\mu_T^\e \vert \rho^\e)} \\
    &\leq \RelEnt(\cg_0^\e | \eff^\e_0)  + \e \frac{C^2 T}{2\alpha_{\LSI}} \left[\RelEnt(\mu_0^\e|\rho^\e) - \RelEnt(\mu_T^\e | \rho^\e)\right]  + \frac12 h^\e(t).
\end{align*}
The first inequality follows since $\RelEnt(\mu_t^\e \vert \rho^\e)$ is monotonically decreasing (recall \eqref{eq:ent-decay}) and thus the supremum of the entropy difference is attained at $t = T$. The second inequality then follows by applying the Young's inequality. This proves \eqref{eq:quantEstEps}.

Finally, to obtain the improved error estimate \eqref{eq:quantEstEpsImproved} we plug the assumption~\eqref{ref:ass-improved-rate} on the initial datum into \eqref{eq:quantEstEps}. The required estimate~\eqref{eq:quantEstEpsImproved} then follows as $\RelEnt(\mu_0^\eps|\rho^\eps)$ is bounded. This bound holds as $\RelEnt(\mu_0^\e \vert \stat^\e)\rightarrow \RelEnt(\mu_0 \vert \stat)$ when $\e\rightarrow 0$ since $\stat^\e\rightarrow \stat$ (see Lemma~\ref{lem:stat-limit-prop}) and the initial data $\mu_0^\e$ for the full system is assumed to converge in $\P(\X)$.   
\end{proof}

\begin{rem}\label{rem:differences-eps-thm-noneps-thm}
    When comparing the main error estimate in the general case \eqref{eq:RelEntEst} provided in Theorem \ref{thm:RelEntEst} to the error estimate obtained \eqref{eq:quantEstEps} in Theorem \ref{thm:eps-est-positive-init} in the case of explicit scale separation, one notices that the scaling in $\alpha_{\LSI}$ is different. While the general estimate \eqref{eq:RelEntEst} scales with $\sqrt{\alpha_{\LSI}}^{-1}$, the estimate $\eqref{eq:quantEstEps}$ scales with $\alpha_{\LSI}^{-1}$. In particular, we also obtain a linear scaling in $\varepsilon$ instead of a square root scaling which might be expected from the general estimate.

    The reason for this improvement lies in the different treatment of $g$ in the proofs. In fact, this is the main difference between the proofs of Theorem \ref{thm:RelEntEst} and \ref{thm:eps-est-positive-init}. In the general case we provide a non-specific upper bound on the $L^2$-norm of $g$ in Lemma \ref{lem:bound-g}. In contrast, in the proof of Theorem \ref{thm:eps-est-positive-init}, we show that the $L^2$-norm of $g$ can be bounded from above by the TV-norm of $\cg^\e - \eta^\e$, see \eqref{eq:g_t_integral_bound_eps}. This improved control then allows for an error estimate with better scaling in $\alpha_{\LSI}$ (and $\e$). However, to obtain \eqref{eq:g_t_integral_bound_eps} we use the Lipschitz continuity of the logarithm away from zero and thus, the estimate is tied to the Assumption \ref{ass-quantres1} of uniformly positive initial data.

    We point out that it is still possible to obtain estimates of the form in Theorem \ref{thm:RelEntEst} from \eqref{eq:asym-first-est} by using that $\|\cg_t^\e -\eff_t^\e\|_{\TV}$ is uniformly bounded in $t$ since $\cg^\e_s$, $\eff^\e_s$ converge as $\e\rightarrow 0$ (recall Theorem~\ref{thm:classical_averaging} and Theorem~\ref{prop:eff-to-aver-eps}).  
\end{rem}

Theorem~\ref{thm:eps-est-positive-init} requires that the initial data is positive. This restriction can be removed if the time interval $[0,T]$ is replaced by $[\delta, T]$ for some $\delta > 0$ (see~\eqref{app-eq:delta-tau} in Appendix~\ref{app:forKol-low} for a similar discussion). In this setting, the irreducibility of $N^\e$ and $L^\e$ implies that the solutions to the coarse-grained equation~\eqref{eq:cg-eps} and the effective equation~\eqref{eq:eff-eps} are instantly positive independent of the initial conditions (see Lemma \ref{lem:eff-irred}). Together with $\RelEnt(\cg_t^\e \vert \rho^\e) \leq \RelEnt(\cg_0^\e \vert \rho^\e)$ for $t \geq 0$ (see~\eqref{eq:ent-decay}) this leads to the following quantitative result.
\begin{prop}\label{prop:eps-est-delta}
Let $\cg^\e, \eta^\e \in C^1([0,T];\P(\Y))$ be the solutions to coarse-grained dynamics \eqref{eq:cg-eps} and the effective dynamics \eqref{eq:eff-eps} with initial data $\cg^\e_0$ and $\eta_0^\e$ respectively. Assume that~\ref{ass-quantres2} in Theorem~\ref{thm:eps-est-positive-init} holds and that the initial data $\cg_0^\e,\eff_0^\e$ converge  in $\P(\Y)$ as $\e\rightarrow 0$.  Then for any $\delta>0$ there exists a constant $C>0$ independent of $\e$ such that  for any $t\in [\delta, T]$ we have
\begin{equation*}
    \RelEnt(\cg_t^\e \vert \eta_t^\e) \leq 2 \RelEnt(\cg_\delta^\e \vert \eta_\delta^\e) + C \dfrac{2\e T}{\alpha_{\LSI}}\left[\RelEnt(\mu_0^\e \vert \rho^\e) - \RelEnt(\mu_T^\e \vert \rho^\e)\right].
\end{equation*}
\end{prop}
Note that the improved error estimate~\eqref{eq:quantEstEpsImproved} in Theorem~\ref{thm:eps-est-positive-init} cannot be applied straightforwardly on the interval $[\delta,T]$ since this would require $\RelEnt(\cg_\delta^\e \vert \eta_\delta^\e) \leq C_0 \e$.
In general this cannot be guaranteed by the control on the initial data alone.

\begin{rem}\label{rem:difficulty-eps}
We point out that the general error estimate derived in Theorem \ref{thm:RelEntEst} does not require positivity of the initial data, which is one of the main assumptions in Theorem \ref{thm:eps-est-positive-init}.
The assumption of positive initial data is required to obtain an $\e$-independent bound~\eqref{eq:g_t_integral_bound_eps} on $g_t$ in the proof of Theorem \ref{thm:eps-est-positive-init}.
Specifically, we need to control the decay of $\cg^\e_t$ to zero as $t \rightarrow 0$ independently of $\e$ to obtain an $\e$-independent bound on $\int_0^\delta g_t^2 \,dt$, i.e.\  we require an estimate of the form $\cg_t^\e(y) \geq c t^N$ for $c > 0$ and $N \in \N$ independent of $\e$.
In the proof of Theorem \ref{thm:RelEntEst} this estimate is derived by applying Proposition~\ref{lem:forKol-low} to the full solution $\mu_t$ and then using $\cg_t = \xi_\#\mu_t$.
However, proceeding like this in the $\e$-dependent case leads to a constant $c \sim e^{-1/\e}$ since $\max_{x \in \X} -L^\e(x,x) \sim \e^{-1}$ as $\e \rightarrow 0$.

Alternatively, we might use the fact that the solution to~\eqref{eq:cg-eps} can be written as
\begin{equation*}
    \cg_t^\e = \exp\left(\int_0^t \hat{L}^\e_s \,ds\right) \cg_0^\e.
\end{equation*}
Using the definition of $\hat{L}^\e_s$ and the definition of $L^\e$~\eqref{eq:CG-res-gen} we obtain
\begin{equation*}
    \left(\int_0^t \hat{L}^\e_s \,ds\right)(y,1-y) = \sum_{z_1,z_2 \in \Z} \sG_{y,1-y}(z_1,z_2) \int_0^t \mu_s^\e(z_1 \vert y)\,ds
\end{equation*}
and thus it is sufficient to control the decay of $\int_0^t \mu_s^\e(z \vert y)\,ds$ as $t \rightarrow 0$ for small $\e > 0$.
We point out, that by using \cite[Lemma 3.2 \& Lemma 3.4]{HilderPeletierSharmaTse20} it follows that $\mu_t^\e(\cdot\vert y)$ converges narrowly to $\rho_y$ in the space of measures on $[0,T]\times \Z$, that is $\int_0^T f_t \mu_t^\e(z \vert y) \,dt \rightarrow \int_0^T f_t \rho_y(z) \,dt$ for all $f \in C([0,T];\R)$ as $\e \rightarrow 0$.
However, this type of convergence seems to be insufficient to control $\int_0^t \mu_s^\e(z \vert y)\,ds$ well enough for our purposes.
Nevertheless, this hints that the restriction to positive initial data in Theorem \ref{thm:eps-est-positive-init} is purely technical.
The technical nature of this restriction is further fostered by the numerical experiments in Section \ref{sec:numericalExp}, where choosing non-negative initial data does not affect the convergence rate.
\end{rem}

\subsection{Error estimate for reversible processes}\label{sec:eps-rev}

So far we have worked with general continuous-time Markov chains without making any assumption regarding the reversibility of the underlying stochastic process. In this section, we work with reversible Markov chains, i.e.\ chains that satisfy the detailed balance condition
\begin{equation}\label{eq:det-bal}
\rho(x_1) L(x_1,x_2) = \rho(x_2) L(x_2,x_1),
\end{equation}
for any $x_1,x_2 \in \X$, where $L$ is an irreducible generator and $\rho$ is the corresponding stationary measure. The log-Sobolev inequality has been a central ingredient in our analysis so far, and it turns out that there are natural sufficient conditions for this inequality to hold in the reversible setting (see for instance~\cite{FathiMaas16,erbarFathi18}). In particular we can guarantee that this inequality holds for the reversible generator used in the numerical experiments discussed in Section \ref{sec:numericalExp}.

The reversible setting has two distinct features in the context of effective dynamics. First, the conditional stationary measure is independent of $\e$ (see Lemma~\ref{lem:reversibility-eps-indep} below) which makes the quantitative result considerably simpler as the log-Sobolev inequality does not depend on $\e$. Second, the $\e$-independent conditional stationary measure $\stat(\cdot|y)$ (recall Theorem~\ref{thm:classical_averaging} for definition) is the same as $\stat_y(\cdot)$ and therefore the averaged and effective dynamics are the same (see Corollary~\ref{cor:eff=aver}). Consequently our techniques provide new error estimates and insights into averaging problems. 

\begin{lem}\label{lem:reversibility-eps-indep}
Let $\rho_y \in \P(\Z)$ be the stationary measure of $Q_y$ for $y \in \Y$.
If $L^\e$ is reversible for any $\e > 0$, the conditional stationary measure satisfies $\rho^\e(\cdot \vert y) = \rho_y(\cdot)$ for any $\e>0$. 
Furthermore, the marginal stationary measure $\xi_\#\rho^\e = \pi$ for any $\e>0$, where $\pi\in\P(\Y)$ solves 
\begin{equation}\label{def:N-rev}
    N^T \pi = 0 \ \  \text{ with }  \ \ N\coloneqq\begin{pmatrix}
-\lambda_0 & \lambda_0  \\ \lambda_1 & -\lambda_1
\end{pmatrix}, \ \ 
\lambda_y \coloneqq \sumsum_{z_1,z_2\in\Z} \stat_y(z_1) \sG_{y,1-y}(z_1,z_2),
\end{equation}	  
i.e.\ $\pi$ is the stationary measure corresponding to $N$. Consequently, the stationary measure corresponding to $L^\e$ is $\e$-independent, i.e.\ $(L^\e)^T\stat=0$ for any $\e>0$ where $\stat((y,z))=\pi_y \stat_y(z)$.
\end{lem}
\begin{proof}
Fix $\e>0$ and let $\rho^\e$ be the stationary measure of $L^\e$. Since $L^\e$ satisfies the detailed balance condition, it is a self-adjoint operator in $L^2(\X,\rho^\e)$, i.e.\ for any $f_1,f_2 \in L^2(\X,\rho^\e)$ we have 
\begin{equation*}
    \sumsum_{x_1,x_2 \in \X} f_1(x_1) L^\e(x_1,x_2) f_2(x_2) \rho^\e(x_1) = \sumsum_{x_1,x_2 \in \X} f_1(x_1) L^\e(x_2,x_1) f_2(x_2) \rho^\e(x_2).
\end{equation*}
Choosing $f_2 = \delta_{(y_1,z_1)}$ and $f_1 = \chi_{\{y_1\} \times \Z}$ for any $(y_1,z_1) \in \X$, where $\delta_x$ is the Dirac delta located at $x \in \X$ and $\chi_{A}$ is the characteristic function of $A \subseteq \X$, we arrive at 
\begin{align*}
    \sum_{z_2 \in \Z} L^\e((y_1,z_1),(y_1,z_2)) \rho^\e((y_1,z_1)) &= \sum_{z_2 \in \Z} L^\e((y_1,z_2),(y_1,z_1)) \rho^\e((y_1,z_2)) \\
    \Longleftrightarrow \sum_{z_2 \in \Z} Q_{y_1}(z_1,z_2) \xi_\# \rho^\e(y_1) \rho^\e(z_1 \vert y_1) &= \sum_{z_2 \in \Z} Q_{y_1}(z_2, z_1) \xi_\# \rho^\e(y_1) \rho^\e(z_2 \vert y_1),
\end{align*} 
where the $D_y$ terms cancel from both sides. Since $\sum_{z_2}Q_{y_1}(z_1,z_2)=0$, the left-hand side of the second equality above is zero. As $\stat^\e$ is a positive probability measure (it is the stationary measure of an irreducible generator), $\xi_\#\stat^\e(y_1)>0$ for every $y_1\in\Y$, and therefore $Q^T_{y_1} \stat^\e(\cdot \vert y_1) = 0$. Consequently $\stat^\e(\cdot \vert y)=\stat_y(\cdot )$ for any $y\in\Y$ since the irreducibility of $Q_y$ implies that it has a unique stationary measure. 

Repeating the calculations as in~\eqref{eq:effToAveraged_aux1} and $\stat^\e(\cdot|y)=\stat_y(\cdot)$  it follows that for any $x=(y,z)\in\X$ we find 
\begin{equation*}
0= \bigl((L^\e)^T\stat^\e\bigr)(x) \Longleftrightarrow  -\lambda_y \xi_\#\stat^\e(y) + \lambda_{1-y} \xi_\#\stat^\e(1-y)= 0  \Longleftrightarrow
\begin{pmatrix}-\lambda_0 & \lambda_1 \\ \lambda_0 & -\lambda_1 \end{pmatrix}\begin{pmatrix} \xi_\# \stat^\e(0) \\ \xi_\# \stat^\e(1)\end{pmatrix} = 0,
\end{equation*}
where $\lambda_y$ for $y\in\Y=\{0,1\}$ are defined in~\eqref{def:N-rev}, i.e.\ $N^T \xi_\#\stat^\e=0$ for every $\e>0$. Since $N$ is $\e$-independent and admits a unique stationary measure, it follows that $\xi_\#\stat^\e=\pi\in\P(\Y)$ where $N^T\pi=0$. 
\end{proof}

A straightforward implication of Lemma~\ref{lem:reversibility-eps-indep} is that the effective and averaged dynamics are the same in this setting. 
\begin{cor}\label{cor:eff=aver}
Assume that the original generator $L^\e$ is reversible for every $\e>0$. Then the effective dynamics $t\mapsto\eta_t\in\P(\Y)$ evolves according to the $\e$-independent  generator $N\in\R^{|\Y|\times|\Y|}$ defined in~\eqref{def:N-rev}, i.e.\ 
\begin{equation}\label{eq:rev-eff}
\partial_t\eff_t = N^T\eff_t. 
\end{equation}
In particular, the effective and averaged dynamics~\eqref{eq:aver-gen} have the same evolution.
\end{cor}
The proof follows by explicitly rewriting the effective generator~\eqref{eq:eff-eps} as in the proof of Lemma~\ref{lem:reversibility-eps-indep}.  

Since the conditional stationary measure is $\e$-independent by Lemma~\ref{lem:reversibility-eps-indep}, the log-Sobolev inequality \eqref{eq:LSIeps} in the reversible setting is $\e$-independent and reads
\begin{equation}	\label{eq:LSI-esp-reversible}
\RelEnt(\nu \vert \rho_y) \leq \dfrac{1}{\alpha_{\LSI}} \RF_{Q_y}(\nu \vert \rho_y),
\end{equation}
for $\nu\in\P(\Z)$. We now state the analogue of Theorem~\ref{thm:eps-est-positive-init} in the reversible setting. 

\begin{prop}\label{prop:convergenceRateReversible}
Let $L^\e$ be a reversible, irreducible generator of the form~\eqref{eq:CG-res-gen}, and $\cg^\e, \eta \in C^1([0,T];\P(\Y))$ be the solutions to corresponding coarse-grained dynamics \eqref{eq:cg-eps} and the effective dynamics \eqref{eq:rev-eff} with initial data $\cg^\e_0$ and $\eta_0$ respectively. Assume that
\begin{enumerate}[label=({B}\arabic*)]
\item The initial data for coarse-grained and effective dynamics is positive, i.e.\ there exists $c_0>0$ independent of $\e$ such that $\cg_0^\e(y),\eff_0(y) > c_0$ for any $y\in\Y$. Furthermore the initial data $\mu^\e_0, \, \cg_0^\e$ converge in $\P(\X),\, \P(\Y)$ respectively as $\e\rightarrow 0$.
\item For any $y\in\Y$ and $\nu\in\P(\Z)$, the family of conditional stationary measures $\stat_y\in\P(\Y)$ satisfies~\eqref{eq:LSI-esp-reversible}.
\end{enumerate}
Then there exists $C>0$ such that for any $t\in [0,T]$ and $\e\in(0,\e_0)$
\begin{equation*}
    \sup_{t \in [0,T]} \RelEnt(\cg_t^\e \vert \eta_t) \leq 2 \RelEnt(\cg_0^\e \vert \eta_0) + C \dfrac{\e T}{\alpha_{\LSI}} \left[\RelEnt(\mu_0^\e \vert \rho^\e) - \RelEnt(\mu_T^\e \vert \rho^\e)\right].
\end{equation*}
If additionally there exists $C_0 < \infty$ independent of $\e$ such that for any $\e \in (0,\e_0)$ we have $\RelEnt(\cg_0^\e \vert \eta_0) \leq C_0\e$, then there exists a constant $D=D(T)$ independent of $\e$ and $\alpha_{\LSI}$ such that 
\begin{equation*}
    \sup_{t \in [0,T]} \RelEnt(\cg_t^\e \vert \eta_t) \leq \biggl( 2C_0+\frac{D}{\alpha_{\LSI}}\biggr) \eps.
\end{equation*}
\end{prop}
Note that in the estimates above both the effective dynamics and stationary measure are $\e$-independent. 

In the remainder of this section we discuss sufficient conditions on the fast-generator $Q_y\in\R^{|\Z|\times|\Z|}$ such that the conditional stationary measure $\stat_y$ satisfies the log-Sobolev inequality~\eqref{eq:LSI-esp-reversible}. First of all, note that  the reversibility of $L^\e$ implies the reversibility of $Q_y$ for all $y \in \Y$ since choosing $x_1 = (y,z_1)$ and $x_2 = (y,z_2)$ in the detailed-balance condition~\eqref{eq:det-bal} for $L^\e$, using the form~\eqref{eq:CG-res-gen} of $L^\e$ and writing the stationary measure as $\rho((y,z)) = \xi_\#\rho(y) \rho^\e(z \vert y)$ (recall Lemma~\ref{lem:reversibility-eps-indep}) yields
\begin{align*}
\xi_\#\rho(y) \rho_y(z_1 \vert y) Q_y(z_1,z_2) = \xi_\#\rho^\e(y) \rho^\e(z_2 \vert y) Q_y(z_2,z_1).
\end{align*}
Note that $\xi_\#\rho(y) > 0$ for all $y \in \Y$ as $\stat$ is a positive probability measure and since $z_1, z_2 \in \Z$ are arbitrary it follows that $Q_y$ is in detailed balance.
Therefore we can apply the framework of \cite{erbarMaas12,FathiMaas16,erbarFathi18}, where the validity of the estimate \eqref{eq:LSI-esp-reversible} is related to the entropic Ricci curvature of the triplet $(\Z,Q_y,\rho_y)$ denoted by $\Ric(\Z,Q_y,\rho_y)$  (see \cite[Definition 1]{erbarMaas12} and \cite[Definition 2.1]{erbarFathi18} for precise definition). In particular, the following result holds. 

\begin{lem}\label{lem:LSIguaranteeRev}
Let $Q_y$ be reversible and assume that $\Ric(\Z,Q_y,\rho_y) \geq 0$ for all $y \in \Y$, then there exists an $\alpha_{\LSI} > 0$ such that the log-Sobolev inequality \eqref{eq:LSI-esp-reversible} holds.
\end{lem}
\begin{proof}
We aim to apply \cite[Theorem 6.1]{erbarFathi18}, which guarantees the existence of an $\alpha_{\LSI} > 0$ if $\Ric(\Z,Q_y,\rho_y) \geq 0$ and the diameter of $(\X,d_\Wasser)$ is bounded. 
Here, $d_\Wasser$ is a distance on $\X$ defined in \cite[Section 2.4]{erbarFathi18} and is bounded from above by $d_{Q_y}$, up to a constant, with
\begin{align*}
    d_{Q_y}(z_1,z_2) = \inf \Biggl\{\sum_{i = 1}^{n-1} \dfrac{1}{\sqrt{\min(Q_y(\tilde{z}_i,\tilde{z}_{i+1}), Q_y({\tilde{z}_{i+1},\tilde{z}_i}))}}\Biggr\},
\end{align*}
where the infimum is taken over all sequences $\tilde{z}_1 = z_1, \tilde{z}_2,\dots,\tilde{z}_n = z_2$ such that $Q_y(\tilde{z}_i,\tilde{z}_{i+1}) > 0$  (see \cite[Lemma 2.3]{erbarFathi18}).
Note that this also implies $Q_y(\tilde{z}_{i+1},\tilde{z}_i) > 0$ by reversibility and since $Q_y$ is irreducible such a sequence exists for all $z_1,z_2 \in \Z$.
Thus, the distance is well-defined.
In particular, since $\Z$ is finite the distance is also bounded, which proves the lemma.
\end{proof}

We now apply the above theory to the examples studied in the numerical experiments Section \ref{sec:numericalExp}.
With $\Z = \{0,\dots,n-1\}$, we set
\begin{equation}\label{eq:genCircle}
Q_y = \begin{pmatrix}
    d & r_+ & 0 & \cdots & 0 & r_- \\
    r_- & d & r_+ & & & 0\\
    0 & r_- & d & r_+ & & \vdots \\
    \vdots & & \ddots & \ddots & \ddots &  0 \\
    0 & & & r_- & d & r_+ \\
    r_+ & 0 & \cdots & 0 & r_- & d
\end{pmatrix} \in \R^{n \times n}
\end{equation}
for $y \in \Y = \{0,1\}$ with $r_+, r_- > 0$ and $d = -(r_+ + r_-)$.
Following \cite[Example 5.6]{erbarMaas12}, we obtain a mapping representation $(G,\tau)$ of $Q_y$ (see \cite[Definition 5.2]{erbarMaas12} for a definition) by defining $G = \{+,-\}$ with
\begin{align*}
+(z) &= z+1 \operatorname{mod} n, \\
-(z) &= z-1 \operatorname{mod} n
\end{align*}
and $\tau(z,+) = r_+$ and $\tau(z,-) = r_-$.
This satisfies \cite[Proposition 5.4]{erbarMaas12}, which yields that $\Ric(\Z,Q_y,\rho_y) \geq 0$ and thus, Lemma \ref{lem:LSIguaranteeRev} and in particular Theorem \ref{thm:eps-est-positive-init} apply.
For an explicit expression of the log-Sobolev constant see~\cite[Theorem 6.1]{erbarFathi18}.

\begin{rem}\label{rem:LSI_reversible}
The example \eqref{eq:genCircle} is rather specific as it models a birth-death process with periodic boundary conditions (i.e.~the first and last states are connected) and constant birth- and death-rates.
However, Lemma \ref{lem:LSIguaranteeRev} also holds for a far more general class of classical birth-death processes without periodic boundary conditions which have state-dependent birth- and death-rates under mild assumptions on the rates.
For the corresponding result we refer to \cite[Theorem 5.1]{mielke13}.
Note that the geodesic $\lambda_Q$-convexity provided in \cite{mielke13} implies a non-negative Ricci curvature bound if $\lambda_Q \geq 0$.
Moreover, it turns out that the assumptions of Lemma \ref{lem:LSIguaranteeRev} are also satisfied for a large class of other examples~\cite{erbarMaas12,FathiMaas16}, which includes a random walk on the discrete hypercube and the full graph $\{0,1\}^n$ as well as the birth-death processes on a countable state space.
\end{rem}

\subsection{Different choices for coarse-graining maps}\label{sec:choice_CG_map}
We conclude this section, by discussing how different choices of coarse-graining maps $\xi$ can lead to vastly different scaling behavior of the corresponding log-Sobolev constant. In view of the error estimates provided by our main Theorems \ref{thm:RelEntEst} and \ref{thm:eps-est-positive-init} this shows that the choice of coarse-graining map has to reflect the slow-fast structure of the problem in order to obtain a good error bound.

For $\varepsilon > 0$ we consider the process generated by
\begin{align*}
    L^\varepsilon = \begin{pmatrix}
        -\varepsilon^{-1} & \varepsilon^{-1} & 0 & 0 \\
        \varepsilon^{-1} & -(1 + \varepsilon^{-1}) & 1 & 0 \\
        0 & 1 & -(1+\varepsilon^{-1}) & \varepsilon^{-1} \\
        0 & 0 & \varepsilon^{-1} & -\varepsilon^{-1}
    \end{pmatrix}
\end{align*}
which models a process on a four-state state-space $\X = \{0,1,2,3\}$, where only neighbouring states are communicating and additionally, the jump rates alternate between $1$ and $\varepsilon^{-1}$. The graph of the process and the corresponding energy landscape are depicted in Figure \ref{fig:example_graph}. This generator is irreducible and reversible with the corresponding stationary measure $\rho(x) = 1/4$ for $x \in \X$ (uniform measure) for every $\varepsilon>0$.

\begin{figure}
    \centering
    \begin{tikzpicture}[line width = 0.3mm]
        \draw[domain=-3:3,samples=100,variable=\x,line width = 0.4pt] plot ({\x},{2*exp(-(0.5*\x)^2/(0.2)^2)+0.5*exp(-(0.5*\x+1)^2/(0.2)^2)+0.5*exp(-(0.5*\x-1)^2/(0.2)^2)});
  
        \node[circle, draw, fill=white] (A) at (-3,-0.75) {0};
        \node[circle, draw, fill=white] (B) at (-1,-0.75) {1};
        \node[circle, draw, fill=white] (C) at (1,-0.75) {2};
        \node[circle, draw, fill=white] (D) at (3,-0.75) {3};
  
        \draw (A) -- node[above] {$\varepsilon^{-1}$} (B);
        \draw (B) -- node[above] {$1$} (C);
        \draw (C) -- node[above] {$\varepsilon^{-1}$} (D);
    \end{tikzpicture}

    \caption{Graph of a birth-death process on a state space $\X = \{0,1,2,3\}$ where the jump rates alternate between $\varepsilon^{-1}$ and $1$. Above the graph is the plot of an energy landscape, which could generate such a process (also see Figure \ref{fig:energyLandscape}). Since the birth-death process considered here is reversible, this relation can be made  precise~\cite{Berglund11,Arnrich-et-al-12}. In particular, the energy barrier between states $1$ and $2$ is much higher than the barrier between $0$ and $1$, and $2$ and $3$ respectively. Note that a lower energy barrier corresponds to a faster jump rate $\varepsilon^{-1}$.}
    \label{fig:example_graph}
\end{figure}
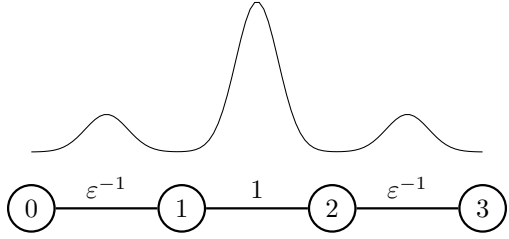

Note that $L^\varepsilon$ can be written as
\begin{align*}
    L^\varepsilon= \dfrac{1}{\varepsilon} \begin{pmatrix}
        -1 & 1 & 0 & 0 \\
        1 & -1 & 0 & 0 \\
        0 & 0 & -1 & 1 \\
        0 & 0 & 1 & -1
    \end{pmatrix} + \begin{pmatrix}
        0 & 0 & 0 & 0 \\
        0 & -1 & 1 & 0 \\
        0 & 1 & -1 & 0 \\
        0 & 0 & 0 & 0
    \end{pmatrix}
\end{align*}
and therefore is of the form \eqref{eq:CG-res-gen} discussed in this multiscale section. Hence the first (and natural) choice for $\xi$ is to group the states $\{0,1\}$ and $\{2,3\}$ together, which can for example be achieved by
\begin{align*}
    \xi(x) = \left\lfloor \dfrac{x}{2} \right\rfloor,
\end{align*}
where $\lfloor \cdot \rfloor$ denotes rounding down to the next integer. This choice corresponds to grouping together states that can easily access each other. For this choice $\Y = \{0,1\}$ with $\Lambda_0 = \{0,1\}$ and $\Lambda_1 = \{2,3\}$.
Using $L^y$ for the restriction of $L^\varepsilon$ to the level set $\Lambda_y \times \Lambda_y$, we find 
\begin{align*}
    L^0 = L^1 = \dfrac{1}{\varepsilon} \begin{pmatrix}
        -1 & 1 \\
        1 & -1
    \end{pmatrix} = \dfrac{1}{\varepsilon} Q.
\end{align*}
Utilizing Remark \ref{rem:diagonal_elems_restriction} we have chosen the diagonal elements of $L^y$ such that the rows sum up to zero.
Applying Remark \ref{rem:LSI_reversible} to the generator $Q$ and using the linearity of the Fisher information $\RF_M$ with respect to $M$ implies the existence of $\alpha_{\LSI} > 0$ which satisfies
\begin{align}
    \RelEnt(\nu \vert \rho(\cdot | y)) \leq \dfrac{\varepsilon}{\alpha_{\LSI}} \RF_{\varepsilon^{-1}Q} (\nu \vert \rho(\cdot | y)).
    \label{eq:LSI_example1}
\end{align}
Therefore the log-Sobolev constant scales like $\varepsilon^{-1}$ and consequently the error $\RelEnt(\hat{\mu}_t^\varepsilon | \eta_t)$ decays at least like $\varepsilon$ in accordance to Theorem \ref{thm:eps-est-positive-init}.

Alternatively, we can group the states together like $\{0\}$, $\{1,2\}$ and $\{3\}$. This can for example be achieved by the coarse-graining map
\begin{align*}
    \xi(x) = \left\lceil \dfrac{x}{2} \right\rceil,
\end{align*}
where $\lceil \cdot \rceil$ denotes rounding up to the next integer. This choice yields $\Y = \{0,1,2\}$ and $\Lambda_0 = \{0\}$, $\Lambda_1 = \{1,2\}$ and $\Lambda_2 = \{3\}$. Since $\Lambda_1$ is the only level set with more than one element, it is sufficient to obtain a log-Sobolev inequality (recall Remark~\ref{rem:LSI-singleton})
\begin{align*}
    \RelEnt(\nu \vert \rho(\cdot | 1)) \leq \dfrac{1}{\alpha_{\LSI}} \RF_{L^1}(\nu \vert \rho(\cdot | 1)), \quad L^1 = \begin{pmatrix}
        -1 & 1 \\
        1 & -1
    \end{pmatrix}.
\end{align*}
Again according to Remark \ref{rem:LSI_reversible} such an $\alpha_{\LSI}$ exists. In fact, since $L^1 = Q$ we obtain the same $\alpha_{\LSI}$ as in \eqref{eq:LSI_example1}. Applying Theorem \ref{thm:RelEntEst} then gives an error estimate with a constant which is independent of $\varepsilon$ and in particular does not decay to zero as $\varepsilon \rightarrow 0$.
Hence, although our theory gives an error estimate for this choice of $\xi$ we obtain no decay for $\varepsilon \rightarrow 0$, which reflects that $\xi$ does not properly account for the scale separation present in the problem.

Finally, we may also choose $\xi$ with level sets $\{0,2\}$ and $\{1,3\}$. In this case we find that the restricted generator is the zero matrix $\tilde{Q} = 0$ on both level sets. In particular, the restricted generator is not irreducible. Using the definition of the Fisher information \eqref{def:FI} we find that $\RF_{\tilde{Q}}(\nu_1 \vert \nu_2) \equiv 0$. Since the relative entropy $\RelEnt(\nu_1 | \nu_2)$ is strictly positive if $\nu_1 \neq \nu_2$ the log-Sobolev inequality
\begin{align*}
    \RelEnt(\nu \vert \rho(\cdot | y)) \leq \dfrac{1}{\alpha_{\LSI}}\RF_{\tilde{Q}}(\nu \vert \rho(\cdot | y)) = 0
\end{align*}
cannot hold for any $\alpha_{\LSI} > 0$.
Hence we obtain no error estimate from our theory.

\begin{rem}
(i) The last example seems to suggest that irreducibility of the restricted generator $L^y$ to the level sets $\Lambda_y \times \Lambda_y$ is a necessary assumption for our theory to work. We point out that all relevant objects, such as coarse-grained and effective dynamics, are properly defined even if $L^y$ is not irreducible. The only requirement is the irreducibility of the \emph{full} generator $L$ to obtain the existence of a full stationary measure. However, as the above example shows, the log-Sobolev inequality might fail without the  irreducibility of $L^y$. Additionally, we remark that, although not strictly necessary, assuming the irreducibility of the $L^y$ in the $\varepsilon$-dependent setting of Section \ref{sec:eps-setup} leads to a more refined analysis.

(ii) Note that the list of coarse-graining maps discussed above is not exhaustive. However, a similar analysis is possible for other possible $\xi$ and we expect that the presented cases are representative of the entire class. In particular, it shows that there are `good' and `bad' choices for the coarse-graining map $\xi$ and that their quality is reflected in our error estimates.

(iii) We restricted the example to a state space with four states for illustrative purposes. However, it is straightforward to generalise the calculations to a birth-death process with alternating birth/death-rates on a general finite state space.
\end{rem}

\section{Numerical experiments}\label{sec:numericalExp}
In this section we numerically investigate the optimality of the theoretical convergence rate established in the previous sections (see Theorem~\ref{thm:eps-est-positive-init}, Proposition~\ref{prop:convergenceRateReversible}) on a simple illustrative example of a finite-state space \emph{birth-death} process. As opposed to the theoretical results, throughout this section we will work with $\e$-independent initial data for the full, coarse-grained and effective dynamics. Furthermore, we will choose the same initial data for both the coarse-grained and effective dynamics since in practice the effective dynamics is supposed to be an approximation for the coarse-grained dynamics. A simple consequence is that $\RelEnt(\cg_0 \vert \eta_0) = 0$ and therefore the additional assumptions~\eqref{ref:ass-improved-rate} on the initial data in Theorem~\ref{thm:eps-est-positive-init} for the convergence rate~\eqref{eq:quantEstEpsImproved} is automatically satisfied. 

We now describe our test example. Choose $\Y = \{0,1\}$ and $\Z = \{0,\dots,n-1\}$ for some $n \in \N$ and let $L^\e$ be an irreducible generator of the form~\eqref{eq:CG-res-gen} with $Q_0 = Q_1$ given in~\eqref{eq:genCircle} and
\begin{equation*}
\sG_{0,1}(z_1,z_2) = \sG_{1,0}(z_1,z_2) = \begin{dcases}
    1, & \text{ if } (z_1,z_2)=(n-1,0)  \text{ or } (z_1,z_2)=(0,n-1), \\
    0, & \text{ otherwise}.
\end{dcases}
\end{equation*}
In the context of Figure~\ref{fig:energyLandscape}, this particular choice of $\sG$ states that a particle can move to a different macro-state only if it is located at the end of the current macro-state. The stationary measure corresponding to $L^\e$ is is the uniform distribution on $\X$, i.e.\ $\rho^\e(x) = \rho(x) = \frac{1}{2n}$ for all $x \in \Y \times \Z$, and the stationary measure on the level sets $\rho_y\in\Lambda_y$ satisfies $\rho_y(z) = \frac1n$ for all $z \in \Lambda_y=\Z$. Furthermore, it is easily checked that the generator $L^\e$ is reversible and consequently the effective dynamics is independent of $\e$ (see Corollary~\ref{cor:eff=aver}). Recall from the discussion at the end of the previous Section \ref{sec:eps-rev} that the conditional stationary measures $\stat_y$ satisfies the log-Sobolev inequality~\eqref{eq:LSI-esp-reversible}. Therefore assuming that the initial data is positive, Proposition \ref{prop:convergenceRateReversible} yields (we choose $\alpha_{\LSI}$ to be a part of the constant here)
\begin{align*}
\sup_{t \in [0,T]} \RelEnt(\cg_t^\e \vert \eta_t) \leq C \e.
\end{align*}

\noindent\textbf{Numerical implementation.}
We now provide details regarding the numerical implementation. Throughout the experiments we fix an end-time $T\in(0,\infty)$, a finite set $\mathcal{E}$ of possible values for the scale-separation parameter $\e$, the number of micro-states in a macro-state $n=|\Z|$ and an initial probability measure $\mu_0 \in \P(\Y \times \Z)$ for the full dynamics (which need not be positive as will be discussed below). 

To calculate the coarse-grained dynamics~\eqref{eq:cg-eps} numerically we first compute the solution to the full system~\eqref{eq:eps-KolEq} and then extract the trajectory of the projected variable using the map $\xi : \Y \times \Z \rightarrow \Y$ with $\xi((y,z)) = y$. To compute the effective dynamics~\eqref{eq:eff-eps} we first compute the effective generator $N$~\eqref{def:N-rev} and then numerically solve the system~\eqref{eq:rev-eff} of ordinary differential equations. With these solutions $(\cg^\e)_{\e \in \mathcal{E}}$ and $\eta$ we then can compute $\max_{t \in [0,T]} \RelEnt(\cg_t^\e \vert \eta_t)$ for all $\e \in \mathcal{E}$. 

In the following remarks we comment on two issues pertaining to the choice of initial data and the measure used to quantify the error between the coarse-grained and effective dynamics. 
\begin{rem}[Choice of initial data]
We briefly discuss the choice of the initial measure $\mu_0$ for the full dynamics.
To remove the influence of the initial data from the results, it seems reasonable to generate a large number of initial data randomly from the uniform distribution on $(0,1)$ (which is normalised to obtain a probability measure), calculate the corresponding trajectories and then average the resulting convergence rates. However, this leads to issues since the solutions to the coarse-grained and effective dynamics are fairly close (see explanation below) and the resulting distance in relative entropy is quite small. Consequently, in most instances, the system has already converged (up to numerical error) and it's not possible to capture the convergence profile. Furthermore, this issue becomes worse if the system-size increases, while we expect that the opposite should be the case. 
    
To give an explanation for this phenomenon, we recall that both the coarse-grained dynamics $\mu^\e$ and the effective dynamics $\eta^\e$ converge to the same stationary measure $\xi_\# \rho$ (see Proposition \ref{prop:eff-inv}).
Furthermore, $\xi_\# \rho(y) = \frac12$ for all $y \in \Y$.
In our numerical setting, the initial data for both the coarse-grained and effective dynamics is $\xi_\# \mu_0$.
However if $n \in \N$ is sufficiently large (and thus the system-size is large) one expects that 
\begin{equation*}
    \xi_\# \mu_0(y) = \sum_{z \in \Z} \mu_0((y,z)) \approx \dfrac{1}{2}
\end{equation*}
with high probability due to the law of large numbers.
Hence the initial data is already close to the stationary state of both dynamics and therefore the corresponding solutions stay close as well. This also explains why the issue is more visible if the size of the system (i.e. $|\Z|=n$) is increased .

To solve this issue, we choose initial data which has most of its mass at only a small number of states (small in comparison to the system size).
\end{rem}

\begin{rem}[Measuring pointwise or uniform decay in time]\label{rem:point-unif}
We point out that measuring the \emph{uniform} decay of the relative-entropy distance in time, i.e.~$\max_{t \in [0,T]} \RelEnt(\cg_t^\e \vert \eta_t^\e)$, is crucial to obtain a result which is comparable to the theoretical result.
In particular, considering the pointwise decay in time, i.e.~$\RelEnt(\cg_T^\e \vert \eta_T^\e)$ for a fixed $T$, we typically find a quadratic decay in $\e$ as opposed to the theoretically predicted linear decay.
The reason behind this discrepancy is that the uniform decay is significantly slower than the pointwise decay as  the equilibration on the level sets speeds up as $\e \rightarrow 0$. In particular, this implies that the time at which the maximum of $t \mapsto \RelEnt(\cg_t^\e \vert \eta_t^\e)$ is attained decreases as $\e \rightarrow 0$ (see Figure \ref{fig:plotRelEntRev}), and consequently the pointwise error decays considerably faster. 
\end{rem}

\begin{figure}[t!]
\begin{minipage}{0.49\textwidth}
    \centering
    \include{relEntDifferentEps}
    \vspace{-0.8cm}
    \caption{Plot of $t \mapsto \RelEnt(\cg_t^\e \vert \eta_t)$ in scenario \ref{scenario1Rev} for $\e = 10^{0}, 10^{-1}, 10^{-2}$. As $\e$ decreases the maximum of $\RelEnt(\cg_t^\e \vert \eta_t)$ as well the time at which this maximum is attained decreases.}
    \label{fig:plotRelEntRev}
\end{minipage}
\hfill
\begin{minipage}{0.49\textwidth}
    \centering
    \vspace{-.25cm}
    \include{RateplotScenario}
    \vspace{-0.7cm}
    \caption{Log-log plot of $\max_{t \in [0,T]} \RelEnt(\cg_t^\e \vert \eta_t)$ for \ref{scenario1Rev}--\ref{scenario3Rev}. The linear function $\e \mapsto \e$ is plotted for comparison.}
    \label{fig:plotConvergenceRev}
\end{minipage}
\end{figure}

\noindent\textbf{Numerical results.}
We choose $\mathcal{E} = \{10^0,10^{-1},10^{-2},10^{-3},10^{-4}\}$, $T = 20$ and $n = 10$.
We distinguish three different scenarios based on different choices for the birth and death rates in the generator $Q_y$ (recall~\eqref{eq:genCircle}) and the initial data $\mu_0$ for the full dynamics:
\begin{enumerate}[label=({S}\arabic*)]
\item \label{scenario1Rev} Set $r_+ = r_- = 1$ in $Q_y$ and $\mu_0 = \left(\frac{13}{10} + \frac{n}{5}\right)^{-1} \left(\delta_{(0,0)} + \frac{3}{10}\delta_{(1,0)} + \frac{1}{10}\right)$. 

\item \label{scenario2Rev} Set $r_+ = 1$ and $r_- = \frac{1}{10}$ and $\mu_0 = \left(\frac{13}{10} + \frac{n}{5}\right)^{-1} \left(\delta_{(0,0)} + \frac{3}{10}\delta_{(1,0)} + \frac{1}{10}\right)$. 

\item \label{scenario3Rev} Set $r_+ = r_- = 1$ in $Q_y$ and $\mu_0 = \frac{13}{10}\left(\delta_{(0,0)} + \frac{3}{10}\delta_{(1,0)}\right)$.
\end{enumerate}
Here $\delta_{(y,z)} \in \P(\Y \times \Z)$ is defined by $\delta_{(y,z)}(\tilde{y},\tilde{z}) = 1$ if $(y,z) = (\tilde{y},\tilde{z})$ and zero otherwise.
Both~\ref{scenario1Rev}--\ref{scenario2Rev} have positive initial data (although with different proportion of mass in the macro-states), with the key difference being in the birth and death rates. Specifically, within a macro-state~\ref{scenario1Rev} has an unbiased movement (i.e.\ symmetric $Q_y$), while in~\ref{scenario2Rev} the movement is biased to the right. In~\ref{scenario3Rev}, we choose an initial data which is concentrated only on two points in the entire state-space and therefore violates the the positivity assumption on the initial data in our theoretical results (see Proposition~\ref{prop:convergenceRateReversible}).

Figure~\ref{fig:plotRelEntRev} plots the pointwise-error profile in scenario~\ref{scenario1Rev} for varying $\e$ and clearly exhibits the behaviour discussed in Remark~\ref{rem:point-unif}, i.e.\ smaller values for $\e$ lead to increased mixing within the macro-states thereby leading to faster convergence of the coarse-grained to the ($\e$-independent) effective dynamics. Note that the error profiles start at zero since the same initial-data is chosen for the coarse-grained and effective dynamics. Similar behaviour is exhibited for the other scenarios as well. 

Figure~\ref{fig:plotConvergenceRev} plots $\max_{t \in [0,T]} \RelEnt(\cg_t^\e \vert \eta_t^\e)$ as a function of $\e$. In all the three scenarios we observe that this error decays to zero as $\e \rightarrow 0$, with a rate roughly equal to one. This is in line with our theoretical results in Theorem~\ref{thm:eps-est-positive-init} (in particular~\eqref{eq:quantEstEpsImproved}), from which we expect a linear decay of error as $\e\to 0$. Finally observe that the convergence in scenario~\ref{scenario3Rev}, which does not satisfy the assumption of positive initial data in Theorem~\ref{thm:eps-est-positive-init}, decays at the same rate as \ref{scenario1Rev}--\ref{scenario2Rev}. This provides further credence to Remark~\ref{rem:difficulty-eps}, wherein we state that the requirement of positive initial data is a purely technical assumption (in that we require it for the proof of the estimate) and we expect that the theoretical result should hold without this requirement.

\section{Conclusion and discussion}\label{sec:discussion}
In this article we provide a systematic first study of coarse-graining for linear continuous-time Markov chains on a finite state space. Inspired by related ideas for diffusions, we propose an effective dynamics which approximates the coarse-grained dynamics. Using entropy techniques and functional inequalities we provide a quantitative estimate on the coarse-graining error. We analyse the effective dynamics in the setup of multiscale averaging problems and provide modified error estimates. 

We now comment on some related issues.

\textbf{Interpreting the log-Sobolev assumption. }The assumption that the conditional stationary measure $\rho(\cdot|y)\in\P(\Lambda_y)$ satisfies the log-Sobolev inequality is the central ingredient used to prove the quantitative estimates in this article. This assumption becomes especially explicit in the reversible setting (see Section~\ref{sec:eps-rev}), where the conditional stationary measure is the stationary solution corresponding to the dynamics within a macro-state, i.e. $Q_y^T\rho(\cdot|y)=0$, and consequently any dynamics evolving according to the generator $Q_y$ converges exponentially fast to $\rho(\cdot|y)$ (recall Remark~\ref{rem:LSI}). Note that this exponential convergence does not imply that the conditional measure corresponding to the reference dynamics $\mu_t(\cdot|y)\in\P(\Lambda_y)$ converges exponentially fast to $\rho(\cdot|y)$ since $\partial_t\mu_t(\cdot|y) \neq Q^T_y\mu_t(\cdot|y)$ even in the reversible setting. However, from the very construction of effective dynamics we expect that the dynamics on the level-sets given by $\mu_t(\cdot|y)$ should converge considerably faster to $\stat(\cdot|y)$ as compared to the convergence of the full dynamics $\mu_t$ to $\stat$ (which is also exponential in our finite state-space setting). Therefore, even though this log-Sobolev assumption is crucial to proving the error estimates, it is completely unclear as to how it connects to the underlying dynamics on the level-sets.

\textbf{Comparison to diffusion processes. } As discussed in the introduction, effective dynamics have been studied for diffusion processes. In this setting with affine $\xi$ and identity diffusion, where the coarse-grained and effective dynamics $\cg_t,\eff_t$ solve Fokker-Planck equations, the quantitative estimates are of the type (see~\cite[Theorem 2.15]{DLPSS18} for details) 
\begin{equation*}
\RelEnt(\cg_t|\eff_t) \leq \RelEnt(\cg_0|\eff_0) + \frac{C}{\alpha^2_{\LSI}} \bigl[ \RelEnt(\mu_0|\rho)-\RelEnt(\mu_t|\rho) \bigr].
\end{equation*}
Note that these estimates hold for $t>0$ and have better scaling in terms of the log-Sobolev constant. This is due to the fact that in the diffusion setting, the log-Sobolev inequality implies the Talagrand inequality which can be used to considerably improve the quantitative bounds. This also leads to a linear scaling in $\e$ in the presence of scale-separation~\cite[Section 3.2]{HartmannNeureitherSharma20} (similar to linear scaling for well-prepared intial datum in Theorem~\ref{thm:eps-est-positive-init}).  However the Talagrand inequality fails in the setting of Markov jump processes (see~\cite[Lemma 8.1.7]{hilder17} for a detailed discussion) and therefore the techniques used in this work differ from the diffusion setting. Another crucial difference is that, while the effective dynamics converges to the averaged dynamics in fair generality for jump processes (recall Theorem~\ref{prop:eff-to-aver-eps}), so far similar results for diffusions only hold in the restrictive setting of linear diffusions where the explicit solution for the effective dynamics is known~\cite[Section 3.3]{HartmannNeureitherSharma20}.

\textbf{Issues with relative entropy. }
Recall that one of the main difficulties in proving an error estimate is to bound $g_t$ (see \eqref{def:C_t}) sufficiently well.
In the general case without explicit scale separation we provide such a bound in Lemma \ref{lem:bound-g}.
However, in the $\varepsilon$-dependent case we can only obtain a bound after assuming that the initial data is positive (see \ref{ass-quantres1}).
It is useful to point out that similar issues have also appeared in the construction of the FIR inequality on discrete state spaces in \cite[Proposition 2.11]{HilderPeletierSharmaTse20}.
In fact, it turns out that the classical FIR inequality using the standard Fisher information \eqref{def:FI} fails in the discrete setting if the assumption of positive initial data is dropped.
Although, as outlined in Remark \ref{rem:difficulty-eps}, we expect that error estimates can also be obtained without the restriction to positive initial data and that this assumption is only necessary for technical reasons.
However, this comparison indicates that there might be an underlying common problem with using relative entropy and the classical Fisher information in the discrete setting since this leads to a logarithmic structure, which cannot be removed as in the continuous setting due to the lack of a chain rule.
This logarithmic structure naturally creates issues when the measures tend to zero somewhere in the domain.

\textbf{General setting.} We now discuss the various possible generalisations of the setting considered in this article. 
In Section~\ref{sec:scale-sep} we limited ourselves, for simplicity, to the case of two macro-states, i.e. $|\Y|=2$, each of which had the same number of micro-states $n=|\Z|$. Our multiscale results easily generalise to the case with more than two macro-states,  each with a different (but finite) number of macro-states, i.e. $\Y$ is an arbitrary finite set and $\X=\cup_{y\in\Y} \{y\}\times \Z_y$.  

Throughout this article we restrict ourselves to finite state-space $\X$. While the construction of the coarse-grained and effective dynamics straightforwardly generalises to the setting of countable state-space, the quantitative estimates do not since the various constants involved explicitly depend on the dimensions of $\X,\Y,\Z$. We expect quantitative estimates with appropriate modifications (which involve suitable  norms) to hold for countable state space, since similar results also exist for diffusion processes on unbounded state-space $\R^d$. This is left to future work.

This article is restricted to the setting of jump processes with linear jump rates. A crucial open question deals with coarse-graining of jump processes with nonlinear jump rates, for instance chemical-reaction networks which are important to model biochemical systems. The construction of effective dynamics as proposed in this work does not straightforwardly generalise to this setting since there is no simple way to build a natural closure for the projected process due to the state-dependent jump rates. 

\textbf{Discrete-time Markov chains. }
While in this article we have focussed on continuous-time Markov chains, discrete-time Markov chains are often employed in practice.
We conjecture that our results naturally generalise to discrete-time Markov chains.
Given a transition matrix $P$ for a discrete-time Markov chain, the evolution of the corresponding probability distribution is given by $\mu_n = P \mu_{n-1}$ for $n \geq 1$.
Therefore, $\mu_n = P^n \mu_0$, for a given initial distribution $\mu_0$.
Hence, if we replace the semigroup $e^{tL^T}$ in our theory with $P^n$ we expect to obtain similar results.
In particular, following the proof of Lemma \ref{lem:CG}, the coarse-grained variable satisfies $\hat{\mu}_n = \hat{P}(n-1) \hat{\mu}_{n-1}$ with
\begin{align*}
\hat{P}(n)(y_1,y_2) = \sumsum_{x_1 \in \Lambda_{y_1}, x_2 \in \Lambda_{y_2}} P(x_1,x_2) \mu_n(x_2 \vert y_2),
\end{align*}
for $y_1,y_2 \in \Y$.
Similarly, we define the effective transition matrix by
\begin{align*}
P_{\text{eff}}(y_1,y_2) = \sumsum_{x_1 \in \Lambda_{y_1}, x_2 \in \Lambda_{y_2}} P(x_1,x_2) \rho(x_2\vert y_2),
\end{align*}
where $y_1,y_2 \in \Y$ and $\rho$ is the stationary measure.
We point out that to guarantee the existence and uniqueness of a stationary measure in the discrete-time setting it is necessary to additionally assume that the Markov chain is aperiodic.

\paragraph{Acknowledgements.}The authors thank the anonymous referee for valuable suggestions and comments. The authors would like to thank Matthias Erbar and Andr{\'e} Schlichting for discussions on functional inequalities. The authors also thank Maksim Maydanskiy for providing arguments used in the proof of Proposition~\ref{lem:forKol-exp}. 
BH was funded by the Deutsche Forschungsgemeinschaft (DFG, German Research Foundation) under Germany's Excellence Strategy -- EXC 2075 -- 390740016 and acknowledges the support by the Stuttgart Center for Simulation Science (SimTech).
The work of US is supported by the Alexander von Humboldt foundation. 

\begin{appendix}

\section{Properties of irreducible continuous-time Markov chains}\label{app:forKol-low}

In this section we state two general results for irreducible Markov jump processes.  Throughout this section  $M\in\R^{|\Z|\times |\Z|}$ is an irreducible generator on a finite space $\Z$ and $\zeta_0\in\P(\Z)$. Let $\zeta\in C^1([0,T];\P(\Z))$ be the solution to the forward Kolmogorov equation
\begin{equation}\label{app-eq:forKol}
\begin{aligned}
\partial_t\zeta&=M^T\zeta,\\
\zeta|_{t=0}&=\zeta_0.
\end{aligned}
\end{equation}
Furthermore, let $\stat\in\P_+(\X)$ be the stationary measure for~\eqref{app-eq:forKol}, i.e.\ $M^T\stat=0$

The following result provides a lower bound for the solution of the forward Kolmogorov equation~\eqref{app-eq:forKol} for short and intermediate times. 
For long times, i.e.~for $t \rightarrow \infty$, we obtain a lower bound from the convergence to the stationary measure in Proposition \ref{lem:forKol-exp} and since the stationary measure has full support.

\begin{prop}\label{lem:forKol-low}
The solution $\zeta\in C^1([0,\infty);\P(\Z))$ to the forward Kolmogorov equation~\eqref{app-eq:forKol} satisfies the following bounds:
\begin{enumerate}
\item For any $\delta\in(0,1)$ and $t\in [0,\delta]$, there exists $c=c(\delta)>0$ and $N\in\mathbb N$ independent of $z\in\Z$ such that 
\begin{equation}\label{app-eq:delta}
\forall z\in\Z: \ \zeta_t(z)\geq c(\delta)t^{N}.
\end{equation}    
\item For any $0<\delta<\tau <\infty$ and $t\in [\delta,\tau]$, there exists $c=c(\delta,\tau)>0$ independent of $z\in\Z$ such that  
\begin{equation}\label{app-eq:delta-tau}
\forall z\in\Z: \ \zeta_t(z)\geq c(\delta,\tau).
\end{equation}
If $\zeta_0(z) > 0$ for all $z \in \Z$, the estimate \eqref{app-eq:delta-tau} also holds for $\delta = 0$.
\end{enumerate}
\end{prop}
\begin{proof}
Since $M$ is a irreducible generator on a finite space $\Z$, using $\nu\coloneqq\sup_{z\in\Z}|M(z,z)|$ we can define a $P\in\R^{|\Z|\times|\Z|}$ with non-negative entries via $M=P-\nu I$, where $I$ is the identity matrix. Furthermore we have $\zeta_t = e^{tM^T} \zeta_0$, with $e^{tM^T}= e^{t(P^T-\nu I)}=e^{-\nu t}e^{tP^T}=e^{-\nu t}\sum_{n\geq 0} \frac{t^n (P^T)^n}{n!} $. Since $M$ is irreducible for any $z,z'\in\Z$ with $z\neq z'$, there exists a  finite sequence $z_0,z_1\ldots,z_N\in\Z$ containing no doubled points with $z_0=z'$, $z_N=z$ and $M(z_n,z_{n+1})>0$. Since none of the points in the sequence repeat it follows that $P(z_i,z_{i+1})=M(z_i,z_{i+1})>0$ and therefore we find
\begin{equation}\label{eq:P^N}
\begin{aligned}
(P^T)^N(z,z')&=\sum_{\alpha_1,\ldots,\alpha_{N-1}\in\Z} P^T(z',\alpha_1) P^T(\alpha_1,\alpha_2) \ldots P^T(\alpha_{N-1},z) \\
&\geq P^T(z',z_{N-1}) P^T(z_{N-1},z_{N-2}) \ldots P^T(z_1,z) \\
&= P(z,z_{1}) P(z_{1},z_{2}) \ldots P(z_{N-1},z')>0.
\end{aligned}
\end{equation}
In the case $z = z'$ we choose another $\tilde{z} \in \Z$ with $\tilde{z} \neq z$. Again since $M$ is irreducible there exists a finite sequence $\hat{z}_0,\hat{z}_1, \dots, \hat{z}_{\hat{N}} \in \Z$ containing no doubled points with $\hat{z}_0 = z$ and $\hat{z}_N = \tilde{z}$ and $M(\hat{z}_n,\hat{z}_{n+1}) > 0$. Similarly there exists a finite sequence $\check{z}_0,\dots,\check{z}_{\hat{N}} \in \Z$ containing no doubled points with $\check{z}_0 = \tilde{z}$ and $\check{z}_{\check{N}} = z$ and $M(\check{z}_n,\check{z}_{n+1}) > 0$. Concatenating the two sequences yields a sequence $z_0, \dots, z_N$ with no successive double point such that $z_0 = z_N = z$ and $M(z_n,z_{n+1}) > 0$. Thus we can proceed as in the case $z \neq z'$ and obtain $(P^T)^N(z,z) > 0$.
Consequently, since $P$ has non-negative components, we find 
\begin{equation*}
e^{tP^T}(z,z') = \sum_{n\geq 0} \frac{t^n (P^T(z,z'))^n}{n!} \geq  \frac{t^N (P^T(z,z'))^N}{N!}>0.
\end{equation*}
Since $\zeta_0\in\P(\Z)$, there exists $y\in\Z$ with $\zeta_0(y) \geq 1/|\Z|$, and for any $z\in\Z$ we find
\begin{equation}\label{app-eq:general-lower}
\zeta_t(z) = \sum_{z'\in\Z} e^{tM^T}(z,z')\zeta_0(z') \geq e^{tM^T}(z,y)\zeta_0(y) = e^{-\nu t} e^{tP^T}(z,y)\dfrac{1}{|\Z|} = \Bigl( \frac{1}{|\Z|}N! (P^T(z,y))^N e^{-\nu t}\Bigr) t^N.
\end{equation}
Note that in this bound $N$ depends on the choice of $z$, i.e.\ $N=N(z)$.

For $\delta\in (0,1]$ and $t\in [0,\delta]$,~\eqref{app-eq:general-lower} leads to the uniform bound
\begin{equation*}
\zeta_t(z) \geq c(\delta) t^{(\min_z N(z))}, \  \ c(\delta)\coloneqq \frac{1}{(\max_z N(z))!} \bigl( \min_z P^T(z,y) \bigr)^{\min_{z}N(z)} e^{-\nu\delta} \zeta_0(y),
\end{equation*}
which is the claimed estimate~\eqref{app-eq:delta}.  For $\delta\in (0,1]$, $\delta<\tau<\infty$ and $t\in [\delta,\tau]$, we have the bound 
\begin{equation*}
\zeta_t(z) \geq c(\delta,\tau) , \  \ c(\delta,\tau)\coloneqq \frac{1}{(\max_z N(z))!} \bigl( \min_z P^T(z,y) \bigr)^{\min_{z}N(z)} e^{-\nu\tau} \zeta_0(y) \delta^{(\min_z N(z))}.
\end{equation*}
and with $1<\delta<\tau<\infty$, $t\in [\delta,\tau]$ we have the bound 
\begin{equation*}
\zeta_t(z) \geq c(\delta,\tau) , \  \ c(\delta,\tau)\coloneqq \frac{1}{(\max_z N(z))!} \bigl( \min_z P^T(z,y) \bigr)^{\min_{z}N(z)} e^{-\nu\tau} \zeta_0(y).
\end{equation*}
This is the claimed estimate~\eqref{app-eq:delta-tau}.

Finally, assume that $\zeta_0(z) > 0$ for all $z \in \Z$. Rewriting $\partial_t \zeta = M^T \zeta$ in time-integrated form, for any $t\in [0,T]$ and $z\in\Z$ we find
\begin{align}
\zeta_t(z) &= \zeta_0(z) + \int_0^t M^T \zeta_s \,ds \geq \zeta_0(z) + \int_0^t M(z,z) \zeta_s(z)ds = \zeta_0(z) - \sum_{z'\in\Z,z'\neq z } \int_0^t M(z,z') \zeta_s(z)ds \nonumber\\
& \geq \zeta_0(z) - t  \sum_{z'\in\Z,z'\neq z } M(z,z') \geq \zeta_0(z) - \tilde ct, \label{eq:app-full-support}
\end{align}
where the first inequality follows since $M(z,z')\geq 0$ for $z\neq z'$, the second equality follows by definition of a generator, the second inequality follows since $0\leq \zeta_s(z) \leq 1$, and the final inequality follows by using $\tilde c \coloneqq \max_{z\in\Z} \sum_{z'\neq z} M(z,z')$ which is a positive constant as $\sum_{z'\neq z} M(z,z')>0$ for any $z\in\Z$ due to the irreducibility of $M$.  Consequently, there exists a small enough $\tilde\delta>0$ such that for any $t\in [0,\tilde\delta]$ and $z\in\Z$ we have $\zeta_0(z)> \tilde c t$. Using~\eqref{eq:app-full-support} it follows that $\zeta_t(z)>c$ for any $t\in [0,\tilde\delta]$ and $z\in\Z$. Combining this with \eqref{app-eq:delta-tau} completes the proof of the proposition. 
\end{proof}

The following result discusses the exponential convergence to the stationary measure in total-variation distance. Part of this argument is provided by Maksim Maydanskiy~\cite{StackExchange}. 
\begin{prop}\label{lem:forKol-exp}
Let $\zeta\in C^1([0,T];\P(\Z))$ be the solution to the forward Kolmogorov equation~\eqref{app-eq:forKol} and $\stat\in\P_+(\Z)$ the corresponding stationary measure. Then there exist constants $C(\zeta_0),D>0$ independent of time, such that  
\begin{equation*}
\bigl\| \zeta_t - \stat  \bigr\|_{\TV} \leq C(\zeta_0)\,e^{-D t},
\end{equation*}
where $\|\cdot\|_{\TV}$ is the total-variation norm.   
\end{prop}
\begin{proof}
Since $M$ is a generator, using Gerschgorin's circle theorem yields that the real part of its spectrum is non-positive, i.e.\ 
\begin{equation*}
\operatorname{Re}(\sigma(M)) \leq 0.
\end{equation*}
Furthermore, applying the Perron-Frobenius theorem as in the proof of Theorem \ref{prop:eff-to-aver-eps} yields that $\lambda = 0$ is a simple eigenvalue and all other eigenvalues have strictly negative real part.
Therefore, $M^T$ also has a simple eigenvalue zero and the remaining eigenvalues have negative real part.
Additionally, since $\rho$ is the unique stationary measure the eigenspace corresponding to the zero eigenvalue is spanned by $\rho$.

Using $P_0$ as the projection onto the zero eigenspace of $L^T$ we obtain
\begin{equation*}
e^{tM^T} = [(I-P_0) + P_0] e^{tM^T} [(I-P_0) + P_0] 
= (I-P_0) e^{tM^T} (I-P_0) + P_0 e^{tM^T} P_0,
\end{equation*}
where we have used $P_0 (I-P_0) = (I-P_0)P_0 = 0$ and that the projections commute with the semigroup $e^{tM^T}$.
Furthermore, we note that $P_0 e^{tM^T} P_0 = P_0$ since $M^TP_0 = 0$.
Hence, using $\zeta_t = e^{tM^T}\zeta_0$ for initial data $\zeta_0 \in P(\Z)$ we find
\begin{equation}\label{app:eq-pre-result-a2}
\left\| \zeta_t - \stat \right\|_{\TV} \leq \bigl\| (I-P_0) e^{tM^T} (I-P_0) \zeta_0 \bigr\|_{\TV} + \left\| P_0 \zeta_0 - \stat\right\|_{\TV}.
\end{equation}
Since $(I-P_0)$ is a spectral projection onto the eigenspaces corresponding to the eigenvalues with strictly negative real part there exists a $D > 0$ such that
\begin{equation*}
\bigl\| (I-P_0) e^{tM^T} (I-P_0) \zeta_0 \bigr\|_{\TV} \leq C(\zeta_0) e^{-Dt}.
\end{equation*}

Following \cite{StackExchange}, we now argue that $P_0 \zeta_0 = \rho$ for any $\zeta_0 \in \P(\Z)$, which concludes the proof by substituting the inequality above into~\eqref{app:eq-pre-result-a2}.
Since $M^T$ is a finite-dimensional matrix, we can write $\zeta_0$ as a linear combination of the (generalized) eigenvectors of $M^T$, i.e.
\begin{equation}\label{eq:spec-decom}
\zeta_0 = c\rho + \sum_{i = 1}^{\vert \Z \vert - 1} c_i \rho_i,
\end{equation}
where coefficients $c, c_i \in \R$ and $\rho_i$  ($i\in\{1,\ldots,|\Z|-1\}$) are the generalized eigenvectors corresponding to the eigenvalues of $M^T$ with strictly negative real part.
Since $M$ is a generator, the left eigenvector of $M^T$ corresponding to the eigenvalue zero is given by $\mathds{1} = (1,\dots,1)$.
In particular, this yields $\mathds{1} \rho_i = 0$ since 
\begin{align*}
0 = \mathds{1} 0 = \mathds{1} (M^T - \lambda_i I)^n \rho_i = (-\lambda_i)^n \mathds{1} \rho_i,
\end{align*}
where we used that $\rho_i$ is the generalized eigenvector corresponding  to the eigenvalue $\lambda_i \neq 0$ and $\mathds{1}M^T=0$. Since $\zeta_0, \rho \in \P(\Z)$, using~\eqref{eq:spec-decom} we obtain
\begin{equation*}
1 = \mathds{1} \zeta_0 = \mathds{1} \Bigl(c \rho + \sum_{i = 1}^{\vert \Z \vert - 1} c_i \rho_i\Bigr) = c \mathds{1} \rho = c.
\end{equation*}
This proves $P_0 \zeta_0 = c\rho = \rho$.
\end{proof}

We now discuss the constants in Propositions \ref{lem:forKol-low} in more detail if $M=N$, where $N$ is the generator of the effective dynamics, see \eqref{eq:eff}. It turns out that we can obtain constants, which do not depend on $N$, but only on the full generator $L$. In particular, the constants are thus independent of the choice of the coarse-graining map $\xi : \X \mapsto \Y$.

\begin{prop}\label{prop:lowerBounds_eff_dyn}
    Let $N$ be the generator of the effective dynamics defined in \eqref{def:eff-gen} and let $\eta \in C^1([0,\infty);\P(\Y))$ be the solution of the corresponding forward Kolmogorov equation \eqref{eq:eff} with initial data $\eta_0 \in \P(\Y)$. Then the following estimates hold:
    \begin{enumerate}
        \item For any $\delta \in (0,1)$ exists $c = c(\delta,L,|\X|) > 0$ independent of $N$ such that
        \begin{align}
            \forall y \in \Y \text{ and } t \in [0,\delta] : \ \eta_t(y) \geq c t^{|\X|}.
            \label{app-eq:lowerboundSmallTime_N}
        \end{align}
        \item For any $0 < \delta < \tau < \infty$ exists $c = c(\delta,\tau,L,|\X|) > 0$ independent of $N$ such that
        \begin{align}
            \forall y \in \Y \text{ and } t \in [\delta,\tau] : \ \eta_t(y) \geq c.
            \label{app-eq:lowerboundIntermediateTime_N}
        \end{align}
        If $\eta_0(y) > 0$ for all $y \in \Y$, the estimate \eqref{app-eq:lowerboundIntermediateTime_N} also holds for $\delta = 0$.
    \end{enumerate}
\end{prop}
\begin{proof}
    We first prove the estimates \eqref{app-eq:lowerboundSmallTime_N} and \eqref{app-eq:lowerboundIntermediateTime_N} by adapting the proof of Proposition \ref{lem:forKol-low}.
    For that we note that
    \begin{align*}
        \nu_N \coloneqq \sup_{y \in \Y} |N(y,y)| = \sup_{y \in \Y} \Bigl|\sumsum_{x_1,x_2 \in \Lambda_y} L(x_1,x_2) \rho(x_1 \vert y)\Bigr| \leq \sum_{x \in \X} |L(x,x)| =: \bar{\nu},
    \end{align*}
    where we used that $\rho(x_1 \vert y) \leq 1$. 
    Moreover, let $y_1 \neq y_2$ such that $N(y_1,y_2) > 0$. Then it holds that
    \begin{align*}
        N(y_1,y_2) = \sumsum_{x_1 \in \Lambda_{y_1}, x_2 \in \Lambda_{y_2}} L(x_1,x_2) \rho(x_1 \vert y_1) \geq \inf_{(x_1,x_2) \in p(L)} L(x_1,x_2) \rho(x_1) > 0.
    \end{align*}
    where $p(L) \coloneqq \{(x_1,x_2) \,:\, L(x_1,x_2) > 0\}$ is the positivity set of $L$. Here, we used that $x_1 \neq x_2$ since $y_1 \neq y_2$ which gives $L(x_1,x_2) \geq 0$ and $\rho(x_1 \vert y_1) \geq \rho(x_1)$ which follows by definition of the conditional measure, see \eqref{def:marg-cond}.
    Finally, we also used that $p(L) \cap (\Lambda_{y_1} \times \Lambda_{y_2}) \neq \emptyset$ since $N(y_1,y_2) > 0$.
    
    Next, define $P = N + \bar{\nu} I$. Since $\nu_N \leq \bar{\nu}$ the matrix $P \in \R^{|\Y| \times |\Y|}$ has only non-negative entries.
    With this we can write
    \begin{align*}
        e^{tN^T} = e^{t(P^T - \bar{\nu}I)} = e^{-\bar{\nu} t} e^{tP^T} = e^{-t\bar{\nu}} \sum_{n \geq 0} \dfrac{t^n (P^T)^n}{n!}.
    \end{align*}
    The generator $N$ is irreducible by Lemma \ref{lem:eff-irred} and hence, for any $y, y^\prime \in \Y$ with $y \neq y^\prime$ exists a finite sequence $y_0,y_1,\dots,y_m \in \Y$ containing no doubled points with $y_0 = y$, $y_m = y^\prime$ and $N(y_n,y_{n+1}) > 0$ for all $n = 0,\dots m-1$.
    Then, we find using $N(y_n,y_{n+1}) = P(y_n,y_{n+1})$ since $y_n \neq y_{n+1}$ that
    \begin{align*}
        (P^T)^m(y,y^\prime) \geq P(y,y_1)P(y_1,y_2)\cdots P(y_{m-1},y^\prime) \geq \left(\inf_{(x_1,x_2) \in p(L)} L(x_1,x_2)\rho(x_1)\right)^{m} \geq \underline{c}
    \end{align*}
    with
    \begin{align*}
        \underline{c} \coloneqq \begin{dcases}
            \left(\inf_{(x_1,x_2) \in p(L)} L(x_1,x_2)\rho(x_1)\right)^{|\X|}, & \text{if } \inf_{(x_1,x_2) \in p(L)} L(x_1,x_2)\rho(x_1) < 1, \\
            \inf_{(x_1,x_2) \in p(L)} L(x_1,x_2)\rho(x_1), & \text{if } \inf_{(x_1,x_2) \in p(L)} L(x_1,x_2)\rho(x_1) \geq 1,
        \end{dcases}.
    \end{align*}
    Here, we use that $1 \leq m \leq |\Y| \leq |\X|$ for all $y_1,y_2 \in \Y$ since the sequence contains no double points.
    For the case $z = z'$ we note that one obtains a sequence with similar properties of length at most $2|\Y|$ by following the construction in the proof of Proposition \ref{lem:forKol-low}.
    Consequently, since $P$ has only non-negative components, we find
    \begin{align*}
        e^{tP^T}(y,y^\prime) \geq \dfrac{t^m (P^T)^m(y,y^\prime)}{m!} \geq \dfrac{t^m \underline{c}}{|\X|!}
    \end{align*}
    Since $\eta_0 \in \P(\Y)$ there exists at least one $y^\prime \in \Y$ such that $\eta_0(y^\prime) \geq 1/|\Y| \geq 1/|\X|$ and for any $y \in \Y$ holds that
    \begin{align*}
        \eta_t(y) \geq \dfrac{\underline{c}}{(|\X|!)|\X|}e^{-\bar{\nu}t} t^m.
    \end{align*}
    Note that the constant $\underline{c}$ only depends on $L$ and $|\X|$.
    Following the proof of Proposition \ref{lem:forKol-low} with this improved estimate gives the desired result.
\end{proof}

\end{appendix}

%\newpage	
{\small
\bibliographystyle{alphainitials}
\bibliography{BibDeskLibrary.bib}
}

\end{document}

%% file: relEntDifferentEps.tex
\definecolor{mycolor1}{rgb}{0.00000,0.44700,0.74100}%
\definecolor{mycolor2}{rgb}{0.85000,0.32500,0.09800}%
\definecolor{mycolor3}{rgb}{0.92900,0.69400,0.12500}%
\begin{tikzpicture}

\begin{axis}[%
width=0.856\textwidth,
height=0.675\textwidth,
at={(0\textwidth,0\textwidth)},
scale only axis,
xmin=0,
xmax=5,
xlabel style={font=\color{white!15!black}},
xlabel={$t$},
ymin=-0.0005,
ymax=0.0035,
axis background/.style={fill=white},
legend style={legend cell align=left, align=left, draw=white!15!black}
]
\addplot [color=mycolor1]
  table[row sep=crcr]{%
0	-2.65780663470871e-16\\
0.01	6.05644832447167e-06\\
0.02	2.30940397209925e-05\\
0.03	5.02490576806429e-05\\
0.04	8.62549403401304e-05\\
0.05	0.000130161821377787\\
0.06	0.000180670962742208\\
0.07	0.000237522063823313\\
0.08	0.00029892502993119\\
0.09	0.000365513112676047\\
0.1	0.000433246447362173\\
0.11	0.000504448382082991\\
0.12	0.000581019584014412\\
0.13	0.000661067680691713\\
0.14	0.000736611827056644\\
0.15	0.000816200707883186\\
0.16	0.000899828777928876\\
0.17	0.000980124219523287\\
0.18	0.00105908043063319\\
0.19	0.00114105764255193\\
0.2	0.00122529342485283\\
0.21	0.00130105977418397\\
0.22	0.00137906731661097\\
0.23	0.00145931380359829\\
0.24	0.00153543534181744\\
0.25	0.00160722155052187\\
0.26	0.00168062174811911\\
0.27	0.00175563459203106\\
0.28	0.00182568927593722\\
0.29	0.00189078479369347\\
0.3	0.00195700020256091\\
0.31	0.00202433475532552\\
0.32	0.00208653856142062\\
0.33	0.00214378839176837\\
0.34	0.0022017976700615\\
0.35	0.00226056601093538\\
0.36	0.00231425682743476\\
0.37	0.00236318190635071\\
0.38	0.00241260768583307\\
0.39	0.00246253399369656\\
0.4	0.00250713805427231\\
0.41	0.00254686481944802\\
0.42	0.00258689667379939\\
0.43	0.00262723357076681\\
0.44	0.00266787546737659\\
0.45	0.00269985996343525\\
0.46	0.00273010002783782\\
0.47	0.00276401569033761\\
0.48	0.00280468206435092\\
0.49	0.00284447276546455\\
0.5	0.00287593478447497\\
0.51	0.00290756615036283\\
0.52	0.00293936686470245\\
0.53	0.00297133693035786\\
0.54	0.00299987802955076\\
0.55	0.00302349424047035\\
0.56	0.00304720140758832\\
0.57	0.00307099955296274\\
0.58	0.00309488869925889\\
0.59	0.00311300977775437\\
0.6	0.00312882210947789\\
0.61	0.0031446745849698\\
0.62	0.00316056722911325\\
0.63	0.00317650006702219\\
0.64	0.00319101162425674\\
0.65	0.0031992094301937\\
0.66	0.00320741926608407\\
0.67	0.00321564114829755\\
0.68	0.00322387509325889\\
0.69	0.0032321211174495\\
0.7	0.00323591006025993\\
0.71	0.00323741764505289\\
0.72	0.0032389281646554\\
0.73	0.00324044162143123\\
0.74	0.00324195801775741\\
0.75	0.00324269554887448\\
0.76	0.00323846320258889\\
0.77	0.00323423697279491\\
0.78	0.00323001684622901\\
0.79	0.00322580280966633\\
0.8	0.00322159484991975\\
0.81	0.0032141397671086\\
0.82	0.00320462068893233\\
0.83	0.00319511965866096\\
0.84	0.00318563664641117\\
0.85	0.00317617162239978\\
0.86	0.00316672455694408\\
0.87	0.00315729542046031\\
0.88	0.0031432641262371\\
0.89	0.00312891901496989\\
0.9	0.00311461104040354\\
0.91	0.00310034015603704\\
0.92	0.00308610631555251\\
0.93	0.00307190947281637\\
0.94	0.00306285060478518\\
0.95	0.00305329232364226\\
0.96	0.00304375166858514\\
0.97	0.00303422861899286\\
0.98	0.0030247231543062\\
0.99	0.00301523525403004\\
1	0.00300576489773048\\
1.01	0.00299435425023833\\
1.02	0.00298148546816958\\
1.03	0.00296864727680591\\
1.04	0.0029558396475976\\
1.05	0.0029430625520947\\
1.06	0.00293031596194622\\
1.07	0.00291759984890086\\
1.08	0.00290219358659333\\
1.09	0.00288648900607186\\
1.1	0.00287083001917495\\
1.11	0.00285521659020913\\
1.12	0.00283964868361936\\
1.13	0.00282412626398836\\
1.14	0.002808649296037\\
1.15	0.00279321774462261\\
1.16	0.00277512799448475\\
1.17	0.00275700123543058\\
1.18	0.00273893688800178\\
1.19	0.00272093491015583\\
1.2	0.00270299526002684\\
1.21	0.00268511789592606\\
1.22	0.00266730277634093\\
1.23	0.00264945401393871\\
1.24	0.00262944719569246\\
1.25	0.00260951920598721\\
1.26	0.00258966999764015\\
1.27	0.00256989952368095\\
1.28	0.00255020773734967\\
1.29	0.00253059459209916\\
1.3	0.00251106004159141\\
1.31	0.00249138443768804\\
1.32	0.00247015188305311\\
1.33	0.00244901310668374\\
1.34	0.00242796805786354\\
1.35	0.00240701668611735\\
1.36	0.0023861589412105\\
1.37	0.00236539477314861\\
1.38	0.0023447241321771\\
1.39	0.00232382947203342\\
1.4	0.00230204517203669\\
1.41	0.00228623041001805\\
1.42	0.0022704719637212\\
1.43	0.0022547698075209\\
1.44	0.00223912391589467\\
1.45	0.00222353426342142\\
1.46	0.00220800082478086\\
1.47	0.00219252357475468\\
1.48	0.00217710248822545\\
1.49	0.00216097821986801\\
1.5	0.00214416896348108\\
1.51	0.00212742709091602\\
1.52	0.0021107525745452\\
1.53	0.00209414538685854\\
1.54	0.00207760550046111\\
1.55	0.00206113288807371\\
1.56	0.00204472752253398\\
1.57	0.00202838937679501\\
1.58	0.00201211842392465\\
1.59	0.00199491085254435\\
1.6	0.00197754748671874\\
1.61	0.00196026167923929\\
1.62	0.00194305340120558\\
1.63	0.00192592262384554\\
1.64	0.0019088693185162\\
1.65	0.00189189345670158\\
1.66	0.00187499501001507\\
1.67	0.00185817395019722\\
1.68	0.00184131285529848\\
1.69	0.00182367963867345\\
1.7	0.00180613283289462\\
1.71	0.00178867240827002\\
1.72	0.00177129833524582\\
1.73	0.00175401058440739\\
1.74	0.00173680912647769\\
1.75	0.00171969393231857\\
1.76	0.00170266497292972\\
1.77	0.00168572221944939\\
1.78	0.0016685045984362\\
1.79	0.00165089989787713\\
1.8	0.00163339002443241\\
1.81	0.0016159749480827\\
1.82	0.00159865463895627\\
1.83	0.00158142906732903\\
1.84	0.00156429820362404\\
1.85	0.00154726201841201\\
1.86	0.00153032048241036\\
1.87	0.00151518175241374\\
1.88	0.00150264185666008\\
1.89	0.00149015494523263\\
1.9	0.00147749547871739\\
1.91	0.00146433800536539\\
1.92	0.00145124025704197\\
1.93	0.00143820221845493\\
1.94	0.00142522387437611\\
1.95	0.00141230520964186\\
1.96	0.00139944620915267\\
1.97	0.00138664685787327\\
1.98	0.00137390714083254\\
1.99	0.00136122704312437\\
2	0.00134860654990535\\
2.01	0.00133604564639729\\
2.02	0.00132345015474083\\
2.03	0.00131037599553092\\
2.04	0.00129736754276311\\
2.05	0.00128442478101391\\
2.06	0.00127154769492851\\
2.07	0.0012587362692209\\
2.08	0.00124599048867377\\
2.09	0.00123331033813849\\
2.1	0.00122069580253471\\
2.11	0.00120814686685095\\
2.12	0.00119566351614375\\
2.13	0.00118324573553855\\
2.14	0.00117088995473484\\
2.15	0.001158120684471\\
2.16	0.00114542215674215\\
2.17	0.00113279435624529\\
2.18	0.00112023726774947\\
2.19	0.0011077508760959\\
2.2	0.0010953351661979\\
2.21	0.00108299012304088\\
2.22	0.00107071573168256\\
2.23	0.00105851197725184\\
2.24	0.0010463788449501\\
2.25	0.0010343163200507\\
2.26	0.00102232438789827\\
2.27	0.00101004867322117\\
2.28	0.000997793541305202\\
2.29	0.000985613860511905\\
2.3	0.000973509615868031\\
2.31	0.000961480792475745\\
2.32	0.000949527375512578\\
2.33	0.000937649350230458\\
2.34	0.000925846701957306\\
2.35	0.000914119416095156\\
2.36	0.000902467478121963\\
2.37	0.000892023352628821\\
2.38	0.000883311556190727\\
2.39	0.000874642888725134\\
2.4	0.000866017343141835\\
2.41	0.000857434912381713\\
2.42	0.000848733393705673\\
2.43	0.000839858924267661\\
2.44	0.00083103146407983\\
2.45	0.00082225100581745\\
2.46	0.000813517542187739\\
2.47	0.000804831065930382\\
2.48	0.000796191569816591\\
2.49	0.000787599046648841\\
2.5	0.0007790534892618\\
2.51	0.000770554890521642\\
2.52	0.000762103243326533\\
2.53	0.000753698540606137\\
2.54	0.000745340775321984\\
2.55	0.000737029940466802\\
2.56	0.000728766029065638\\
2.57	0.000720549034174654\\
2.58	0.000712080757939763\\
2.59	0.000703658272637568\\
2.6	0.000695286215377771\\
2.61	0.000686964579028927\\
2.62	0.000678693356491851\\
2.63	0.000670472540701836\\
2.64	0.000662302124626223\\
2.65	0.000654182101266498\\
2.66	0.000646112463657218\\
2.67	0.00063809320486595\\
2.68	0.000630124317993662\\
2.69	0.000622205796174467\\
2.7	0.00061433763257562\\
2.71	0.000606519820397756\\
2.72	0.000598752352874644\\
2.73	0.000590914705675346\\
2.74	0.000583028561688181\\
2.75	0.000575195671133309\\
2.76	0.000567416027138765\\
2.77	0.000559689622865962\\
2.78	0.000552016451511008\\
2.79	0.000544396506304622\\
2.8	0.000536829780512296\\
2.81	0.000529316267433206\\
2.82	0.000521855960401461\\
2.83	0.00051444885278577\\
2.84	0.000507094937988568\\
2.85	0.000499794209447348\\
2.86	0.000492546660633517\\
2.87	0.000486041651350556\\
2.88	0.000480578206271958\\
2.89	0.000474986896990813\\
2.9	0.000469428457534168\\
2.91	0.000463902884591853\\
2.92	0.000458410174867519\\
2.93	0.000452950325078869\\
2.94	0.000447523331958782\\
2.95	0.000442129192254162\\
2.96	0.000436767902726279\\
2.97	0.00043143946015094\\
2.98	0.00042614386131829\\
2.99	0.000420881103033234\\
3	0.000415651182114432\\
3.01	0.000410454095395443\\
3.02	0.000405289839724156\\
3.03	0.000400158411962798\\
3.04	0.000395059808987587\\
3.05	0.000389994027689749\\
3.06	0.000384961064974695\\
3.07	0.00037996091776194\\
3.08	0.00037486698527385\\
3.09	0.000369772438964343\\
3.1	0.000364712874125431\\
3.11	0.000359688287608071\\
3.12	0.000354698676277774\\
3.13	0.000349744037015294\\
3.14	0.000344824366716569\\
3.15	0.000339939662292461\\
3.16	0.000335089920668838\\
3.17	0.000330275138786387\\
3.18	0.000325495313601485\\
3.19	0.000320750442084702\\
3.2	0.000316040521221981\\
3.21	0.000311365548014461\\
3.22	0.000306725519477899\\
3.23	0.000302120432643322\\
3.24	0.000297550284556792\\
3.25	0.000293015072278887\\
3.26	0.000288514792886135\\
3.27	0.000283992891896657\\
3.28	0.000279441066563598\\
3.29	0.000274926116584159\\
3.3	0.000270448038991289\\
3.31	0.000266006830834638\\
3.32	0.000261602489178218\\
3.33	0.000257235011102346\\
3.34	0.000252904393703002\\
3.35	0.000248610634091507\\
3.36	0.000244353729394296\\
3.37	0.000240525261971047\\
3.38	0.000237303834248795\\
3.39	0.000234104183385353\\
3.4	0.000230926307957256\\
3.41	0.000227770206547595\\
3.42	0.000224635877746045\\
3.43	0.000221523320148173\\
3.44	0.000218432532356339\\
3.45	0.000215363512979089\\
3.46	0.00021230450893202\\
3.47	0.000209195272832537\\
3.48	0.000206109026164768\\
3.49	0.000203045767512454\\
3.5	0.000200005495466253\\
3.51	0.000196988208623166\\
3.52	0.000193993905586829\\
3.53	0.000191022584968002\\
3.54	0.000188074245383302\\
3.55	0.000185148885455979\\
3.56	0.000182246503816702\\
3.57	0.000179367099101747\\
3.58	0.000176510669954614\\
3.59	0.000173677215025527\\
3.6	0.000170866732970357\\
3.61	0.000168079222452648\\
3.62	0.000165314682142106\\
3.63	0.000162573110715174\\
3.64	0.000159854506854327\\
3.65	0.000157158869249554\\
3.66	0.000154442027786061\\
3.67	0.000151739171107257\\
3.68	0.000149060216915279\\
3.69	0.00014640516389112\\
3.7	0.000143774010723019\\
3.71	0.000141166756105466\\
3.72	0.000138583398739654\\
3.73	0.000136023937334109\\
3.74	0.00013348837060353\\
3.75	0.000130976697269672\\
3.76	0.000128488916061181\\
3.77	0.000126025025713482\\
3.78	0.000123585024968446\\
3.79	0.000121168912574959\\
3.8	0.000118776687288607\\
3.81	0.000116408347871684\\
3.82	0.000114056220391852\\
3.83	0.000111695163965158\\
3.84	0.000109358835531971\\
3.85	0.000107047233852256\\
3.86	0.000104760357694184\\
3.87	0.00010270407878671\\
3.88	0.000100970873092693\\
3.89	9.92524372884877e-05\\
3.9	9.75487707578467e-05\\
3.91	9.58598728875611e-05\\
3.92	9.41857430678062e-05\\
3.93	9.25263806910897e-05\\
3.94	9.08817851531015e-05\\
3.95	8.92519558523238e-05\\
3.96	8.76368921900626e-05\\
3.97	8.60365935707978e-05\\
3.98	8.44510594014986e-05\\
3.99	8.2862967795444e-05\\
4	8.12762963457543e-05\\
4.01	7.97049807623402e-05\\
4.02	7.81490204480925e-05\\
4.03	7.66084148088042e-05\\
4.04	7.50831632527474e-05\\
4.05	7.35732651922019e-05\\
4.06	7.20787200408938e-05\\
4.07	7.05995272170983e-05\\
4.08	6.91356861404766e-05\\
4.09	6.76871962351369e-05\\
4.1	6.62540569274575e-05\\
4.11	6.4836267646441e-05\\
4.12	6.34338278246837e-05\\
4.13	6.20467368976849e-05\\
4.14	6.0674994303375e-05\\
4.15	5.93185994831359e-05\\
4.16	5.79566699833705e-05\\
4.17	5.66089169066839e-05\\
4.18	5.52770332343947e-05\\
4.19	5.39610184063025e-05\\
4.2	5.26608718659826e-05\\
4.21	5.1376593059255e-05\\
4.22	5.01081814353894e-05\\
4.23	4.88556364464758e-05\\
4.24	4.76189575482049e-05\\
4.25	4.63981441983013e-05\\
4.26	4.51931958585627e-05\\
4.27	4.40041119932951e-05\\
4.28	4.2830892069772e-05\\
4.29	4.16735355584594e-05\\
4.3	4.05320419327761e-05\\
4.31	3.94064106694451e-05\\
4.32	3.8291923121595e-05\\
4.33	3.71823252923609e-05\\
4.34	3.60890517434511e-05\\
4.35	3.50121019530997e-05\\
4.36	3.39514754021646e-05\\
4.37	3.30046805218148e-05\\
4.38	3.22141018800392e-05\\
4.39	3.14331130986337e-05\\
4.4	3.06617139243606e-05\\
4.41	2.98999041051661e-05\\
4.42	2.91476833903992e-05\\
4.43	2.84050515309787e-05\\
4.44	2.7672008278625e-05\\
4.45	2.69485533863942e-05\\
4.46	2.62346866089741e-05\\
4.47	2.55304077016868e-05\\
4.48	2.48357164219725e-05\\
4.49	2.4142869381222e-05\\
4.5	2.34559732208563e-05\\
4.51	2.27789951641462e-05\\
4.52	2.21119349670144e-05\\
4.53	2.14547923866339e-05\\
4.54	2.08075671815941e-05\\
4.55	2.01702591119436e-05\\
4.56	1.95428679386758e-05\\
4.57	1.89253934240524e-05\\
4.58	1.83178353321107e-05\\
4.59	1.77201934274218e-05\\
4.6	1.7132467476502e-05\\
4.61	1.65546572470825e-05\\
4.62	1.59867625077123e-05\\
4.63	1.54287830287127e-05\\
4.64	1.48807185816804e-05\\
4.65	1.4342296430111e-05\\
4.66	1.38061201935165e-05\\
4.67	1.32801616718283e-05\\
4.68	1.27644206366724e-05\\
4.69	1.22588968604516e-05\\
4.7	1.17635901176485e-05\\
4.71	1.12785001838258e-05\\
4.72	1.08036268357294e-05\\
4.73	1.03389698511238e-05\\
4.74	9.88452900968235e-06\\
4.75	9.44030409238223e-06\\
4.76	9.00629488093151e-06\\
4.77	8.58250115856303e-06\\
4.78	8.16892271048031e-06\\
4.79	7.76555932212883e-06\\
4.8	7.37241078106656e-06\\
4.81	6.98947687595452e-06\\
4.82	6.61478112557075e-06\\
4.83	6.24756370784605e-06\\
4.84	5.89083464105087e-06\\
4.85	5.5445937143829e-06\\
4.86	5.20884071941491e-06\\
4.87	4.91461776219574e-06\\
4.88	4.67380896435553e-06\\
4.89	4.43905012363979e-06\\
4.9	4.21034114047106e-06\\
4.91	3.98768191638989e-06\\
4.92	3.77107235321659e-06\\
4.93	3.56051235304096e-06\\
4.94	3.35600181890238e-06\\
4.95	3.15754065411127e-06\\
4.96	2.96512876260036e-06\\
4.97	2.7787660489562e-06\\
4.98	2.59845241811049e-06\\
4.99	2.42217822934033e-06\\
};
\addlegendentry{$\varepsilon = 10^0$}

\addplot [color=mycolor2]
  table[row sep=crcr]{%
0	-2.65780663470871e-16\\
0.01	5.44710639469471e-06\\
0.02	1.90122953126113e-05\\
0.03	3.77072376166014e-05\\
0.04	5.94166189139727e-05\\
0.05	8.28096896899177e-05\\
0.06	0.000106986067108513\\
0.07	0.000131113723389881\\
0.08	0.000155040147067539\\
0.09	0.000178203733420179\\
0.1	0.000200672549587437\\
0.11	0.000222237957653318\\
0.12	0.000242649466513197\\
0.13	0.000262006652161509\\
0.14	0.000280283540193641\\
0.15	0.000297434124547795\\
0.16	0.000313613993064609\\
0.17	0.00032872983051809\\
0.18	0.000342797673702757\\
0.19	0.000355751984287785\\
0.2	0.000367560262733106\\
0.21	0.000378580117699057\\
0.22	0.000388796137373007\\
0.23	0.000398013304679664\\
0.24	0.000406225723639162\\
0.25	0.000413660739700042\\
0.26	0.000420465454751644\\
0.27	0.000426244703998149\\
0.28	0.000431670677684303\\
0.29	0.000435913922818901\\
0.3	0.000439784429700617\\
0.31	0.000442897351225702\\
0.32	0.000445279264314502\\
0.33	0.000447213044024239\\
0.34	0.0004486236565802\\
0.35	0.000449336912283816\\
0.36	0.00044978582997003\\
0.37	0.000449392751144171\\
0.38	0.000448952251019522\\
0.39	0.000447573546728943\\
0.4	0.000446197883958723\\
0.41	0.0004439894481488\\
0.42	0.000441650021855908\\
0.43	0.000439159412781259\\
0.44	0.000435869427979301\\
0.45	0.000432593156493096\\
0.46	0.000428928254884661\\
0.47	0.000426201718747082\\
0.48	0.000426028025677682\\
0.49	0.000425269836045256\\
0.5	0.000424342497706868\\
0.51	0.000423369325630847\\
0.52	0.000421849995291402\\
0.53	0.000420328936420224\\
0.54	0.000418339833696815\\
0.55	0.000416356197821217\\
0.56	0.000413965251136902\\
0.57	0.000411556648270802\\
0.58	0.000408805612200153\\
0.59	0.000406013978647856\\
0.6	0.000402937063902686\\
0.61	0.000399801148717722\\
0.62	0.000396428506012956\\
0.63	0.00039298168770531\\
0.64	0.00038934209136021\\
0.65	0.000385616603290421\\
0.66	0.000381737883956632\\
0.67	0.000377764695783735\\
0.68	0.000373669431633425\\
0.69	0.000369470268455092\\
0.7	0.000365206663720974\\
0.71	0.000360801123047703\\
0.72	0.000356384756740935\\
0.73	0.000351798776525468\\
0.74	0.000347243569274137\\
0.75	0.000342506412278083\\
0.76	0.000337798350625541\\
0.77	0.000332966001484157\\
0.78	0.000328129280989011\\
0.79	0.000323210752500752\\
0.8	0.000318266638273435\\
0.81	0.000313277104539867\\
0.82	0.000308244127787953\\
0.83	0.000303197828787386\\
0.84	0.000298093614225621\\
0.85	0.00029300517129225\\
0.86	0.000287851055603252\\
0.87	0.000282716971407901\\
0.88	0.000277531804767871\\
0.89	0.000272369698459533\\
0.9	0.000267166148632866\\
0.91	0.000261988653938332\\
0.92	0.000256778799167627\\
0.93	0.000251598451528705\\
0.94	0.000248136918076057\\
0.95	0.000245406645736279\\
0.96	0.000242630533214631\\
0.97	0.000239832152290863\\
0.98	0.000237000921882538\\
0.99	0.000234141999532561\\
1	0.000231262158005497\\
1.01	0.000228348887401569\\
1.02	0.000225425070617763\\
1.03	0.000222468456591157\\
1.04	0.000219505088035589\\
1.05	0.00021651063321548\\
1.06	0.000213513538420435\\
1.07	0.00021048708795526\\
1.08	0.000207461858526761\\
1.09	0.000204406957723099\\
1.1	0.000201356182157398\\
1.11	0.000198287093293188\\
1.12	0.000195213362646211\\
1.13	0.000192133762315307\\
1.14	0.000189042388848201\\
1.15	0.000185956607111957\\
1.16	0.000182852700550015\\
1.17	0.000179764800707962\\
1.18	0.000176652449076289\\
1.19	0.000173567738542452\\
1.2	0.000170452531738953\\
1.21	0.00016736435133396\\
1.22	0.000164260572250734\\
1.23	0.000161173528332821\\
1.24	0.000158084666764941\\
1.25	0.000155003220688491\\
1.26	0.000151933280157685\\
1.27	0.000148863261918625\\
1.28	0.00014581021146976\\
1.29	0.000142757143265409\\
1.3	0.000139725580439587\\
1.31	0.000136693255125383\\
1.32	0.000133686702643691\\
1.33	0.000130678829116092\\
1.34	0.000127700922748223\\
1.35	0.000124722677038991\\
1.36	0.000121772249832968\\
1.37	0.000118828573664403\\
1.38	0.000115910601765158\\
1.39	0.000113004717039061\\
1.4	0.000110193262735911\\
1.41	0.000108670412384173\\
1.42	0.000107142228868654\\
1.43	0.000105608076095935\\
1.44	0.000104066492612914\\
1.45	0.000102522918553651\\
1.46	0.000100969971147413\\
1.47	9.94188324041937e-05\\
1.48	9.78565324652857e-05\\
1.49	9.62996723193037e-05\\
1.5	9.47289435227041e-05\\
1.51	9.31669660544487e-05\\
1.52	9.15934549437893e-05\\
1.53	9.00270281926109e-05\\
1.54	8.84523713852248e-05\\
1.55	8.6883436628808e-05\\
1.56	8.53094435906187e-05\\
1.57	8.37399310634175e-05\\
1.58	8.2167817560889e-05\\
1.59	8.0598662115983e-05\\
1.6	7.90337467445968e-05\\
1.61	7.7465931343495e-05\\
1.62	7.59094125994623e-05\\
1.63	7.43449470331121e-05\\
1.64	7.2796928880469e-05\\
1.65	7.12393329986277e-05\\
1.66	6.96972502528936e-05\\
1.67	6.81529650024331e-05\\
1.68	6.66191614960704e-05\\
1.69	6.50884840709114e-05\\
1.7	6.35652709680655e-05\\
1.71	6.2049980682724e-05\\
1.72	6.05392959660755e-05\\
1.73	5.90410754307703e-05\\
1.74	5.75448536697882e-05\\
1.75	5.60655029782685e-05\\
1.76	5.45866782266551e-05\\
1.77	5.31246711931123e-05\\
1.78	5.16660995172214e-05\\
1.79	5.02242070102932e-05\\
1.8	4.87876750353694e-05\\
1.81	4.73677427762472e-05\\
1.82	4.59550374387998e-05\\
1.83	4.45590297483657e-05\\
1.84	4.31723578913409e-05\\
1.85	4.18010059122379e-05\\
1.86	4.04425365512565e-05\\
1.87	3.93726502889132e-05\\
1.88	3.87137061854531e-05\\
1.89	3.80499056487783e-05\\
1.9	3.73871344764576e-05\\
1.91	3.67176393901436e-05\\
1.92	3.60515284943711e-05\\
1.93	3.53785489652682e-05\\
1.94	3.47096205824777e-05\\
1.95	3.40340774148103e-05\\
1.96	3.33633168583926e-05\\
1.97	3.26862204533132e-05\\
1.98	3.201462224337e-05\\
1.99	3.1336325307804e-05\\
2	3.06641843229329e-05\\
2.01	2.9988103013967e-05\\
2.02	2.93157660465186e-05\\
2.03	2.8642777778629e-05\\
2.04	2.79713471158355e-05\\
2.05	2.73023959250188e-05\\
2.06	2.66329822840609e-05\\
2.07	2.59689869516106e-05\\
2.08	2.53025519322708e-05\\
2.09	2.46448252569894e-05\\
2.1	2.39825910603873e-05\\
2.11	2.33291838730438e-05\\
2.12	2.26750601860287e-05\\
2.13	2.20271306557887e-05\\
2.14	2.13820641385329e-05\\
2.15	2.07407778137145e-05\\
2.16	2.01058676080413e-05\\
2.17	1.94728752342258e-05\\
2.18	1.88473886732839e-05\\
2.19	1.82243117769865e-05\\
2.2	1.76096722965462e-05\\
2.21	1.69976254714445e-05\\
2.22	1.6394911211302e-05\\
2.23	1.57950156878311e-05\\
2.24	1.52053574956805e-05\\
2.25	1.46191528166313e-05\\
2.26	1.40424913573811e-05\\
2.27	1.34714817391418e-05\\
2.28	1.29093667531402e-05\\
2.29	1.23546139699886e-05\\
2.3	1.18082157435162e-05\\
2.31	1.12707872333257e-05\\
2.32	1.07412553306051e-05\\
2.33	1.02221857344294e-05\\
2.34	9.7108003903148e-06\\
2.35	9.21113465268252e-06\\
2.36	8.71906270896511e-06\\
2.37	8.34895017135222e-06\\
2.38	8.14055524397776e-06\\
2.39	7.93399053524357e-06\\
2.4	7.72473171335632e-06\\
2.41	7.51804171240955e-06\\
2.42	7.30813071452398e-06\\
2.43	7.10120219024743e-06\\
2.44	6.89154430484054e-06\\
2.45	6.68470821874501e-06\\
2.46	6.47607551745866e-06\\
2.47	6.26991095317481e-06\\
2.48	6.06267739056252e-06\\
2.49	5.8575571125802e-06\\
2.5	5.65308926567818e-06\\
2.51	5.4494121916489e-06\\
2.52	5.24816131078679e-06\\
2.53	5.0465687011416e-06\\
2.54	4.84893321282213e-06\\
2.55	4.65019106905525e-06\\
2.56	4.45550539577102e-06\\
2.57	4.26153729754649e-06\\
2.58	4.07056782546059e-06\\
2.59	3.88150174584177e-06\\
2.6	3.69501333783405e-06\\
2.61	3.51145636588658e-06\\
2.62	3.33010773210531e-06\\
2.63	3.15264280367896e-06\\
2.64	2.97710038288639e-06\\
2.65	2.80633284672018e-06\\
2.66	2.63747119647593e-06\\
2.67	2.47337557897679e-06\\
2.68	2.3120486382876e-06\\
2.69	2.15548971537892e-06\\
2.7	2.00233038605778e-06\\
2.71	1.85398024930904e-06\\
2.72	1.70962783044065e-06\\
2.73	1.57018889994419e-06\\
2.74	1.43534429623078e-06\\
2.75	1.30536007004372e-06\\
2.76	1.18071747528528e-06\\
2.77	1.06093140058274e-06\\
2.78	9.47124403830393e-07\\
2.79	8.38248854271141e-07\\
2.8	7.35745627965375e-07\\
2.81	6.38602981493908e-07\\
2.82	5.48155578950502e-07\\
2.83	4.63474593016767e-07\\
2.84	3.85719379815581e-07\\
2.85	3.14191342943508e-07\\
2.86	2.49832034061024e-07\\
2.87	2.06101006588989e-07\\
2.88	1.85868599301953e-07\\
2.89	1.66022993972894e-07\\
2.9	1.47298093936129e-07\\
2.91	1.2896152182132e-07\\
2.92	1.11821075367188e-07\\
2.93	9.5423140077886e-08\\
2.94	8.01725814764391e-08\\
2.95	6.5973725042539e-08\\
2.96	5.2927165777346e-08\\
2.97	4.11938219522438e-08\\
2.98	3.06709429939375e-08\\
2.99	2.16886436818926e-08\\
3	1.401884940285e-08\\
3.01	8.01609710986866e-09\\
3.02	3.58490511359394e-09\\
3.03	8.98502711916085e-10\\
3.04	1.92975131495138e-13\\
3.05	9.83524162404616e-10\\
3.06	3.90807881607384e-09\\
3.07	8.9112357994401e-09\\
3.08	1.5996554191672e-08\\
3.09	2.53655074929115e-08\\
3.1	3.69493815750074e-08\\
3.11	5.10342961355557e-08\\
3.12	6.74635205883106e-08\\
3.13	8.66210428066981e-08\\
3.14	1.08244604063261e-07\\
3.15	1.328039807198e-07\\
3.16	1.60073531793256e-07\\
3.17	1.9035664029413e-07\\
3.18	2.23634417664223e-07\\
3.19	2.59989712895255e-07\\
3.2	2.99667481333597e-07\\
3.21	3.42448707558269e-07\\
3.22	3.88935387463378e-07\\
3.23	4.38518867741482e-07\\
3.24	4.92184808708425e-07\\
3.25	5.4910702100325e-07\\
3.26	6.10204605061473e-07\\
3.27	6.74873182544486e-07\\
3.28	7.43781104505393e-07\\
3.29	8.16648820290192e-07\\
3.3	8.93854856461518e-07\\
3.31	9.75034487560875e-07\\
3.32	1.06104763380999e-06\\
3.33	1.15098470925088e-06\\
3.34	1.24623168030981e-06\\
3.35	1.34532023267635e-06\\
3.36	1.45023247234623e-06\\
3.37	1.52771284641197e-06\\
3.38	1.5621206446247e-06\\
3.39	1.59713252911343e-06\\
3.4	1.6339426023084e-06\\
3.41	1.67215788772458e-06\\
3.42	1.71149237853329e-06\\
3.43	1.75305028841542e-06\\
3.44	1.79510595452203e-06\\
3.45	1.84007456833048e-06\\
3.46	1.8856790504183e-06\\
3.47	1.93345542427227e-06\\
3.48	1.9826836973626e-06\\
3.49	2.03371973501315e-06\\
3.5	2.08665195209549e-06\\
3.51	2.14106725344608e-06\\
3.52	2.19786831163397e-06\\
3.53	2.25581522833975e-06\\
3.54	2.31665628259775e-06\\
3.55	2.37830585059196e-06\\
3.56	2.44317863751253e-06\\
3.57	2.50920540087896e-06\\
3.58	2.57813899061897e-06\\
3.59	2.64856775959049e-06\\
3.6	2.72172426078266e-06\\
3.61	2.79674380019679e-06\\
3.62	2.87430281997129e-06\\
3.63	2.95407916459057e-06\\
3.64	3.03616644547057e-06\\
3.65	3.12103417047268e-06\\
3.66	3.20777731639124e-06\\
3.67	3.29791514759809e-06\\
3.68	3.3898194968711e-06\\
3.69	3.48512321944974e-06\\
3.7	3.58264832238842e-06\\
3.71	3.68325892104908e-06\\
3.72	3.78626487271966e-06\\
3.73	3.89253029513151e-06\\
3.74	4.00132393890478e-06\\
3.75	4.11346110843099e-06\\
3.76	4.22826282812038e-06\\
3.77	4.34649301347153e-06\\
3.78	4.46761274937183e-06\\
3.79	4.59238339557016e-06\\
3.8	4.71926975365868e-06\\
3.81	4.85103899208359e-06\\
3.82	4.98462056193428e-06\\
3.83	5.12305114894738e-06\\
3.84	5.26397959864345e-06\\
3.85	5.40888696393777e-06\\
3.86	5.55741884602449e-06\\
3.87	5.6605775684837e-06\\
3.88	5.69528148618875e-06\\
3.89	5.73009123942098e-06\\
3.9	5.76875173150513e-06\\
3.91	5.80759755394625e-06\\
3.92	5.84919566513822e-06\\
3.93	5.89212313163376e-06\\
3.94	5.93670824187097e-06\\
3.95	5.98376870874682e-06\\
3.96	6.03136990391857e-06\\
3.97	6.08237552079797e-06\\
3.98	6.1339282403693e-06\\
3.99	6.18869758886522e-06\\
4	6.24404483181911e-06\\
4.01	6.30264443849077e-06\\
4.02	6.36185448746785e-06\\
4.03	6.42435557329031e-06\\
4.04	6.48748194799357e-06\\
4.05	6.55376778878141e-06\\
4.06	6.62137685581094e-06\\
4.07	6.69133384175119e-06\\
4.08	6.76333778144422e-06\\
4.09	6.8370464207804e-06\\
4.1	6.91353463370915e-06\\
4.11	6.99114446286683e-06\\
4.12	7.07221529673389e-06\\
4.13	7.15429368988523e-06\\
4.14	7.23948210813231e-06\\
4.15	7.32619479646692e-06\\
4.16	7.41559682974706e-06\\
4.17	7.5070472152659e-06\\
4.18	7.60076253036174e-06\\
4.19	7.6972262931939e-06\\
4.2	7.79563193186545e-06\\
4.21	7.89639669021092e-06\\
4.22	7.99987560881404e-06\\
4.23	8.10515267151736e-06\\
4.24	8.21382241133215e-06\\
4.25	8.32372329213037e-06\\
4.26	8.43770519650842e-06\\
4.27	8.55246119410092e-06\\
4.28	8.67193803255877e-06\\
4.29	8.79236621560264e-06\\
4.3	8.91621809713961e-06\\
4.31	9.0426309996422e-06\\
4.32	9.17097894913071e-06\\
4.33	9.30351380458959e-06\\
4.34	9.43699830641999e-06\\
4.35	9.57527349647108e-06\\
4.36	9.71508319243253e-06\\
4.37	9.80483001478691e-06\\
4.38	9.81795310294083e-06\\
4.39	9.83248706039504e-06\\
4.4	9.849629012397e-06\\
4.41	9.86752339080139e-06\\
4.42	9.88867653098141e-06\\
4.43	9.90993197413654e-06\\
4.44	9.93509303386748e-06\\
4.45	9.96028609402374e-06\\
4.46	9.98862509464472e-06\\
4.47	1.00176108760834e-05\\
4.48	1.00498906133615e-05\\
4.49	1.00828564423656e-05\\
4.5	1.01185615912994e-05\\
4.51	1.01555190892381e-05\\
4.52	1.01946661161157e-05\\
4.53	1.0235592709519e-05\\
4.54	1.0278128918957e-05\\
4.55	1.03232009530958e-05\\
4.56	1.03691568902439e-05\\
4.57	1.04182835694274e-05\\
4.58	1.04683866579123e-05\\
4.59	1.05208938293673e-05\\
4.6	1.05752889843925e-05\\
4.61	1.06314315008278e-05\\
4.62	1.06895920646641e-05\\
4.63	1.07496268938151e-05\\
4.64	1.0811695296343e-05\\
4.65	1.0875659675245e-05\\
4.66	1.0941674774004e-05\\
4.67	1.10096091117326e-05\\
4.68	1.10797655811022e-05\\
4.69	1.11520084946446e-05\\
4.7	1.12252731892079e-05\\
4.71	1.1302164231218e-05\\
4.72	1.13793174212821e-05\\
4.73	1.14603160104886e-05\\
4.74	1.15421951271209e-05\\
4.75	1.16265671997935e-05\\
4.76	1.17132256171604e-05\\
4.77	1.18010253416017e-05\\
4.78	1.18925134020877e-05\\
4.79	1.19843543412291e-05\\
4.8	1.20801960611918e-05\\
4.81	1.21769527767616e-05\\
4.82	1.22763908081327e-05\\
4.83	1.23781271154392e-05\\
4.84	1.24812248039146e-05\\
4.85	1.25880074046763e-05\\
4.86	1.26952443802106e-05\\
4.87	1.27563054403714e-05\\
4.88	1.27435432630508e-05\\
4.89	1.27340791544125e-05\\
4.9	1.27248639149517e-05\\
4.91	1.27190249924183e-05\\
4.92	1.27133272693216e-05\\
4.93	1.27110848011829e-05\\
4.94	1.2708941421016e-05\\
4.95	1.2709987771736e-05\\
4.96	1.27119224766653e-05\\
4.97	1.27159894863691e-05\\
4.98	1.27217660822925e-05\\
4.99	1.27288401354556e-05\\
};
\addlegendentry{$\varepsilon = 10^{-1}$}

\addplot [color=mycolor3]
  table[row sep=crcr]{%
0	-2.65780663470871e-16\\
0.01	2.47935879978557e-06\\
0.02	5.49783369106879e-06\\
0.03	7.93177814496586e-06\\
0.04	9.77992180294877e-06\\
0.05	1.11446170585946e-05\\
0.06	1.21259458785248e-05\\
0.07	1.2825030471363e-05\\
0.08	1.33159289970142e-05\\
0.09	1.36533057665173e-05\\
0.1	1.38799215909004e-05\\
0.11	1.40245719587365e-05\\
0.12	1.41075189016441e-05\\
0.13	1.41436745705946e-05\\
0.14	1.41423874829368e-05\\
0.15	1.41108377480683e-05\\
0.16	1.40535280889896e-05\\
0.17	1.39737685659527e-05\\
0.18	1.38742197067191e-05\\
0.19	1.37563360532748e-05\\
0.2	1.36213129400403e-05\\
0.21	1.34703237109519e-05\\
0.22	1.33039372469e-05\\
0.23	1.3123048194019e-05\\
0.24	1.29279689175268e-05\\
0.25	1.27194172260077e-05\\
0.26	1.24979292863469e-05\\
0.27	1.22638750836282e-05\\
0.28	1.20178749105503e-05\\
0.29	1.17604438125357e-05\\
0.3	1.14920065764018e-05\\
0.31	1.12132579034212e-05\\
0.32	1.09247067746321e-05\\
0.33	1.06269390140216e-05\\
0.34	1.03205025989298e-05\\
0.35	1.00060011680183e-05\\
0.36	9.68420133858582e-06\\
0.37	9.35571907873098e-06\\
0.38	9.02114532193993e-06\\
0.39	8.68125919066239e-06\\
0.4	8.33668799234263e-06\\
0.41	7.98826567032429e-06\\
0.42	7.63670259215968e-06\\
0.43	7.28271795502914e-06\\
0.44	6.9271909100754e-06\\
0.45	6.57081606259143e-06\\
0.46	6.21450057015156e-06\\
0.47	6.02119584234054e-06\\
0.48	6.12297204306334e-06\\
0.49	6.2162391583835e-06\\
0.5	6.30082060384301e-06\\
0.51	6.37650584660093e-06\\
0.52	6.44322352090238e-06\\
0.53	6.50076797053785e-06\\
0.54	6.5491911486438e-06\\
0.55	6.58838541415423e-06\\
0.56	6.61828258519343e-06\\
0.57	6.63884528547291e-06\\
0.58	6.65005688041832e-06\\
0.59	6.65202900651421e-06\\
0.6	6.64468094903488e-06\\
0.61	6.62814536533571e-06\\
0.62	6.60245781642461e-06\\
0.63	6.56775768991197e-06\\
0.64	6.52404964504366e-06\\
0.65	6.47156509608567e-06\\
0.66	6.41043215625272e-06\\
0.67	6.34080627627217e-06\\
0.68	6.26281749676423e-06\\
0.69	6.17670272064653e-06\\
0.7	6.08271478637888e-06\\
0.71	5.98107912495242e-06\\
0.72	5.87198125600783e-06\\
0.73	5.75572256040027e-06\\
0.74	5.63256834096824e-06\\
0.75	5.50286804801552e-06\\
0.76	5.36686726855941e-06\\
0.77	5.22493678192034e-06\\
0.78	5.07735413403268e-06\\
0.79	4.92455566988349e-06\\
0.8	4.76688424314443e-06\\
0.81	4.6046642669783e-06\\
0.82	4.43839137597205e-06\\
0.83	4.2684087537177e-06\\
0.84	4.09520391620984e-06\\
0.85	3.91912111096719e-06\\
0.86	3.74068282735706e-06\\
0.87	3.56039431587867e-06\\
0.88	3.37864459023867e-06\\
0.89	3.19595064989051e-06\\
0.9	3.01285255917337e-06\\
0.91	2.82987026117582e-06\\
0.92	2.64745932800677e-06\\
0.93	2.46622964501331e-06\\
0.94	2.45702856189437e-06\\
0.95	2.5171230904745e-06\\
0.96	2.5729434952194e-06\\
0.97	2.62434951073948e-06\\
0.98	2.67118927943684e-06\\
0.99	2.71336961626327e-06\\
1	2.75080155556469e-06\\
1.01	2.78337410254724e-06\\
1.02	2.81103495869796e-06\\
1.03	2.83376331707864e-06\\
1.04	2.85147173459291e-06\\
1.05	2.86412295072251e-06\\
1.06	2.87173531339437e-06\\
1.07	2.87430428768829e-06\\
1.08	2.87188421504623e-06\\
1.09	2.86441021843799e-06\\
1.1	2.8519900679533e-06\\
1.11	2.83465970603454e-06\\
1.12	2.81248694898139e-06\\
1.13	2.78552275124523e-06\\
1.14	2.7539054607902e-06\\
1.15	2.71772222887258e-06\\
1.16	2.67705390139959e-06\\
1.17	2.63206008939095e-06\\
1.18	2.58285129856256e-06\\
1.19	2.52961059928536e-06\\
1.2	2.47250046055939e-06\\
1.21	2.41165386846687e-06\\
1.22	2.34726836639528e-06\\
1.23	2.27951728090604e-06\\
1.24	2.20866206203005e-06\\
1.25	2.13488153846873e-06\\
1.26	2.05838030726082e-06\\
1.27	1.97940898485969e-06\\
1.28	1.89820644480641e-06\\
1.29	1.81506434308095e-06\\
1.3	1.73019032020136e-06\\
1.31	1.64389431657052e-06\\
1.32	1.55645644417564e-06\\
1.33	1.46814485602997e-06\\
1.34	1.37928826690907e-06\\
1.35	1.29017757952653e-06\\
1.36	1.20114281581618e-06\\
1.37	1.11252462533216e-06\\
1.38	1.02464170062652e-06\\
1.39	9.37819118309322e-07\\
1.4	8.5866853593501e-07\\
1.41	8.92653304949368e-07\\
1.42	9.24795756559408e-07\\
1.43	9.54980065666811e-07\\
1.44	9.83077015809827e-07\\
1.45	1.00900849460796e-06\\
1.46	1.03269340658908e-06\\
1.47	1.05401422897131e-06\\
1.48	1.07294642032239e-06\\
1.49	1.08940593658818e-06\\
1.5	1.10334463829015e-06\\
1.51	1.11472046115989e-06\\
1.52	1.12351413538987e-06\\
1.53	1.12972235127287e-06\\
1.54	1.13328682030379e-06\\
1.55	1.13424760207433e-06\\
1.56	1.1325717100087e-06\\
1.57	1.12831415810283e-06\\
1.58	1.1214984306136e-06\\
1.59	1.11213832238539e-06\\
1.6	1.10028385636189e-06\\
1.61	1.08597561438004e-06\\
1.62	1.0692899702424e-06\\
1.63	1.05028256098692e-06\\
1.64	1.02904621669111e-06\\
1.65	1.00565022575876e-06\\
1.66	9.80186875606604e-07\\
1.67	9.52754748669515e-07\\
1.68	9.23457687824819e-07\\
1.69	8.92428650921023e-07\\
1.7	8.5978573885774e-07\\
1.71	8.25642560768679e-07\\
1.72	7.90149626786951e-07\\
1.73	7.53431675534007e-07\\
1.74	7.15680366861898e-07\\
1.75	6.77011294991561e-07\\
1.76	6.37614570537094e-07\\
1.77	5.97653239753056e-07\\
1.78	5.57309047660389e-07\\
1.79	5.16765346733573e-07\\
1.8	4.76213834861019e-07\\
1.81	4.35857495204148e-07\\
1.82	3.95897725683706e-07\\
1.83	3.56538724208548e-07\\
1.84	3.18005142276544e-07\\
1.85	2.80512926130081e-07\\
1.86	2.44298940157831e-07\\
1.87	2.30147551527518e-07\\
1.88	2.46628889722712e-07\\
1.89	2.62571612319278e-07\\
1.9	2.77887195239553e-07\\
1.91	2.92480686856899e-07\\
1.92	3.06260363558126e-07\\
1.93	3.19163445085355e-07\\
1.94	3.31109164123115e-07\\
1.95	3.42047715093667e-07\\
1.96	3.51916068460126e-07\\
1.97	3.60668544630818e-07\\
1.98	3.68271234693987e-07\\
1.99	3.74678506084034e-07\\
2	3.79870556706696e-07\\
2.01	3.83819006997326e-07\\
2.02	3.86521936333135e-07\\
2.03	3.87961649542354e-07\\
2.04	3.88139960698884e-07\\
2.05	3.87051665048509e-07\\
2.06	3.84717537244331e-07\\
2.07	3.81156213602257e-07\\
2.08	3.76385604000347e-07\\
2.09	3.7042543245356e-07\\
2.1	3.63318992840033e-07\\
2.11	3.55107552062856e-07\\
2.12	3.45830686077715e-07\\
2.13	3.35538922092841e-07\\
2.14	3.24295591809119e-07\\
2.15	3.12151761223901e-07\\
2.16	2.99184736330662e-07\\
2.17	2.85461716794668e-07\\
2.18	2.71055279444418e-07\\
2.19	2.56058648310674e-07\\
2.2	2.40553201224288e-07\\
2.21	2.24634039466332e-07\\
2.22	2.08392392801759e-07\\
2.23	1.91934029223728e-07\\
2.24	1.75372529584862e-07\\
2.25	1.58810052636074e-07\\
2.26	1.4236674634213e-07\\
2.27	1.261626404684e-07\\
2.28	1.10326393238308e-07\\
2.29	9.498388801036e-08\\
2.3	8.02676919735539e-08\\
2.31	6.63224246201608e-08\\
2.32	5.32839642352582e-08\\
2.33	4.1304609498617e-08\\
2.34	3.0532080752191e-08\\
2.35	2.1122787268165e-08\\
2.36	1.32371140434755e-08\\
2.37	1.05711561318249e-08\\
2.38	1.38749487545896e-08\\
2.39	1.73928337653643e-08\\
2.4	2.10527554140563e-08\\
2.41	2.4783448203649e-08\\
2.42	2.85221569342614e-08\\
2.43	3.2209368343479e-08\\
2.44	3.57913985442543e-08\\
2.45	3.92213689613197e-08\\
2.46	4.24512274015913e-08\\
2.47	4.54416451247178e-08\\
2.48	4.81597971218749e-08\\
2.49	5.05725398777595e-08\\
2.5	5.26533277976562e-08\\
2.51	5.43830016514813e-08\\
2.52	5.57450088522972e-08\\
2.53	5.67227988745053e-08\\
2.54	5.73076437262521e-08\\
2.55	5.74968533877541e-08\\
2.56	5.72902183627454e-08\\
2.57	5.66936682398875e-08\\
2.58	5.57085407281225e-08\\
2.59	5.43515763417528e-08\\
2.6	5.26354659321593e-08\\
2.61	5.05834094988264e-08\\
2.62	4.82160069025471e-08\\
2.63	4.55617484359117e-08\\
2.64	4.26505361793793e-08\\
2.65	3.95159228559142e-08\\
2.66	3.619960156726e-08\\
2.67	3.27398681574244e-08\\
2.68	2.91843425807061e-08\\
2.69	2.55817265653973e-08\\
2.7	2.19839767798762e-08\\
2.71	1.84452958789858e-08\\
2.72	1.5025971362897e-08\\
2.73	1.17884376562514e-08\\
2.74	8.79923542329e-09\\
2.75	6.12515387353821e-09\\
2.76	3.83821747885896e-09\\
2.77	2.01469925555005e-09\\
2.78	7.32336124827494e-10\\
2.79	7.19280853513072e-11\\
2.8	1.17739901470359e-10\\
2.81	9.56831693751611e-10\\
2.82	2.67946094221279e-09\\
2.83	5.37784451028873e-09\\
2.84	9.14859432643838e-09\\
2.85	1.40891109661411e-08\\
2.86	2.03006667104041e-08\\
2.87	2.29107974893296e-08\\
2.88	1.89895396571904e-08\\
2.89	1.56156796870887e-08\\
2.9	1.27342493063472e-08\\
2.91	1.02924536840674e-08\\
2.92	8.24087309568393e-09\\
2.93	6.53238750886745e-09\\
2.94	5.12331076137274e-09\\
2.95	3.9741649950767e-09\\
2.96	3.04858474781052e-09\\
2.97	2.31257299913152e-09\\
2.98	1.73619281674931e-09\\
2.99	1.2922709326845e-09\\
3	9.56283745559664e-10\\
3.01	7.07103152394658e-10\\
3.02	5.26605617415959e-10\\
3.03	3.99827299691624e-10\\
3.04	3.14218390146269e-10\\
3.05	2.60676344095165e-10\\
3.06	2.32268784641524e-10\\
3.07	2.25314163737972e-10\\
3.08	2.38853600989175e-10\\
3.09	2.74722798985233e-10\\
3.1	3.3731384986284e-10\\
3.11	4.33396329124441e-10\\
3.12	5.73094473759568e-10\\
3.13	7.68208575609641e-10\\
3.14	1.03410894853728e-09\\
3.15	1.38763327312157e-09\\
3.16	1.84914966591191e-09\\
3.17	2.44077471892013e-09\\
3.18	3.1867816503405e-09\\
3.19	4.11457575478315e-09\\
3.2	5.25393068768587e-09\\
3.21	6.63611840498679e-09\\
3.22	8.29611783523948e-09\\
3.23	1.02684090763183e-08\\
3.24	1.25933543189112e-08\\
3.25	1.53109908780478e-08\\
3.26	1.84647924837965e-08\\
3.27	2.20986565378485e-08\\
3.28	2.6259897128493e-08\\
3.29	3.09994737093464e-08\\
3.3	3.63661480552865e-08\\
3.31	4.24160019773699e-08\\
3.32	4.92015714241426e-08\\
3.33	5.67825911700233e-08\\
3.34	6.52172431756177e-08\\
3.35	7.45656560733319e-08\\
3.36	8.48907061251159e-08\\
3.37	8.85856129987164e-08\\
3.38	8.16718280017604e-08\\
3.39	7.53646165125351e-08\\
3.4	6.96252100652807e-08\\
3.41	6.44145530122569e-08\\
3.42	5.96945225130486e-08\\
3.43	5.54311931266944e-08\\
3.44	5.15902652746634e-08\\
3.45	4.81437962833627e-08\\
3.46	4.5063675069019e-08\\
3.47	4.23223708349526e-08\\
3.48	3.98979487058605e-08\\
3.49	3.77671005629614e-08\\
3.5	3.59126382458821e-08\\
3.51	3.43151411864146e-08\\
3.52	3.29598460933021e-08\\
3.53	3.1832782060678e-08\\
3.54	3.0924495535205e-08\\
3.55	3.0224022004382e-08\\
3.56	2.97233554979461e-08\\
3.57	2.94170633891776e-08\\
3.58	2.93027538312239e-08\\
3.59	2.9377748808832e-08\\
3.6	2.96422404685664e-08\\
3.61	3.00967002219246e-08\\
3.62	3.07464585293856e-08\\
3.63	3.15960906900462e-08\\
3.64	3.26534621472021e-08\\
3.65	3.39258044159633e-08\\
3.66	3.54244026094189e-08\\
3.67	3.71625100135703e-08\\
3.68	3.91518298350178e-08\\
3.69	4.14101044813751e-08\\
3.7	4.39516599962488e-08\\
3.71	4.67986462831865e-08\\
3.72	4.99687391290206e-08\\
3.73	5.34827623979771e-08\\
3.74	5.73663437433486e-08\\
3.75	6.16446376573395e-08\\
3.76	6.63412637469255e-08\\
3.77	7.14840675953134e-08\\
3.78	7.71052410289646e-08\\
3.79	8.32323455040348e-08\\
3.8	8.99001927990922e-08\\
3.81	9.71395687182186e-08\\
3.82	1.0498567327443e-07\\
3.83	1.13474821938567e-07\\
3.84	1.22644579996336e-07\\
3.85	1.32533950152937e-07\\
3.86	1.43182400215024e-07\\
3.87	1.46744249316332e-07\\
3.88	1.39258757172827e-07\\
3.89	1.32324031209652e-07\\
3.9	1.25915797127493e-07\\
3.91	1.20004240179001e-07\\
3.92	1.1456403418777e-07\\
3.93	1.09569270259521e-07\\
3.94	1.04995847416449e-07\\
3.95	1.00826386217027e-07\\
3.96	9.70384829055381e-08\\
3.97	9.36153851104646e-08\\
3.98	9.05369966863832e-08\\
3.99	8.77925972322729e-08\\
4	8.53627767744241e-08\\
4.01	8.32375364668203e-08\\
4.02	8.14064945671206e-08\\
4.03	7.98581387450648e-08\\
4.04	7.85848635966892e-08\\
4.05	7.75775328104404e-08\\
4.06	7.6830128807495e-08\\
4.07	7.63410353960534e-08\\
4.08	7.61038264095236e-08\\
4.09	7.61198226527534e-08\\
4.1	7.63829693114051e-08\\
4.11	7.68967930883343e-08\\
4.12	7.76626015019791e-08\\
4.13	7.86840880016404e-08\\
4.14	7.99634678472487e-08\\
4.15	8.15043586425794e-08\\
4.16	8.33175636168204e-08\\
4.17	8.54060821881225e-08\\
4.18	8.7780141122941e-08\\
4.19	9.04506847131353e-08\\
4.2	9.34262408126114e-08\\
4.21	9.67197941036547e-08\\
4.22	1.00343297320125e-07\\
4.23	1.0430983591117e-07\\
4.24	1.08635866248347e-07\\
4.25	1.13339205457631e-07\\
4.26	1.18432231811609e-07\\
4.27	1.23934929550159e-07\\
4.28	1.29866774136107e-07\\
4.29	1.36246792496514e-07\\
4.3	1.43097495823943e-07\\
4.31	1.50438132943205e-07\\
4.32	1.58291391444822e-07\\
4.33	1.66684893194687e-07\\
4.34	1.75636757302638e-07\\
4.35	1.85178307489527e-07\\
4.36	1.95331648936701e-07\\
4.37	1.98488689937503e-07\\
4.38	1.90871151155888e-07\\
4.39	1.83747246785724e-07\\
4.4	1.7709894037859e-07\\
4.41	1.70904182297577e-07\\
4.42	1.65145769556306e-07\\
4.43	1.59802417949641e-07\\
4.44	1.5485737858142e-07\\
4.45	1.50296964970169e-07\\
4.46	1.46105072573873e-07\\
4.47	1.42268558529269e-07\\
4.48	1.38772375409566e-07\\
4.49	1.35606459639192e-07\\
4.5	1.32757311934535e-07\\
4.51	1.30218623120299e-07\\
4.52	1.27977874853514e-07\\
4.53	1.26027646987813e-07\\
4.54	1.24363018209569e-07\\
4.55	1.22973581725191e-07\\
4.56	1.21855865589345e-07\\
4.57	1.21004717307336e-07\\
4.58	1.204177583051e-07\\
4.59	1.2009167371592e-07\\
4.6	1.20020823499004e-07\\
4.61	1.20209410147038e-07\\
4.62	1.20654146706413e-07\\
4.63	1.21355368965707e-07\\
4.64	1.22313902749791e-07\\
4.65	1.23533932095028e-07\\
4.66	1.25016299210696e-07\\
4.67	1.26764633658614e-07\\
4.68	1.28785024471761e-07\\
4.69	1.31081044470725e-07\\
4.7	1.33658020367504e-07\\
4.71	1.36525111957254e-07\\
4.72	1.39687079638267e-07\\
4.73	1.43151550027385e-07\\
4.74	1.46929803987136e-07\\
4.75	1.51030098880032e-07\\
4.76	1.55460536880876e-07\\
4.77	1.60233491038053e-07\\
4.78	1.65362896605882e-07\\
4.79	1.70857305918659e-07\\
4.8	1.76731218089557e-07\\
4.81	1.82995374035954e-07\\
4.82	1.89668231192542e-07\\
4.83	1.96761473387542e-07\\
4.84	2.04289877931947e-07\\
4.85	2.1227197356924e-07\\
4.86	2.20722917874212e-07\\
4.87	2.22987359821153e-07\\
4.88	2.15808451080627e-07\\
4.89	2.09051915534789e-07\\
4.9	2.02700202736052e-07\\
4.91	1.96740148613022e-07\\
4.92	1.91155147568321e-07\\
4.93	1.8593132051028e-07\\
4.94	1.81057501555902e-07\\
4.95	1.76522160680857e-07\\
4.96	1.72313205075696e-07\\
4.97	1.68417491245719e-07\\
4.98	1.64826611780469e-07\\
4.99	1.61532397072596e-07\\
};
\addlegendentry{$\varepsilon = 10^{-2}$}

\end{axis}
\end{tikzpicture}%

%% file: RateplotScenario.tex
\definecolor{mycolor1}{rgb}{0.00000,0.44700,0.74100}%
\definecolor{mycolor2}{rgb}{0.85000,0.32500,0.09800}%
\definecolor{mycolor3}{rgb}{0.92900,0.69400,0.12500}%
\definecolor{mycolor4}{rgb}{0.49400,0.18400,0.55600}%
\begin{tikzpicture}

\begin{axis}[%
width=0.856\textwidth,
height=0.675\textwidth,
at={(0\textwidth,0\textwidth)},
scale only axis,
xmode=log,
xmin=0.0001,
xmax=1,
xminorticks=true,
xlabel style={font=\color{white!15!black}},
xlabel={$\varepsilon$},
ymode=log,
ymin=1e-07,
ymax=1,
yminorticks=true,
axis background/.style={fill=white},
legend style={legend cell align=left, align=left, draw=white!15!black},
legend pos = north west
]
\addplot [color=mycolor1, mark=x, mark options={solid, mycolor1}]
  table[row sep=crcr]{%
1	0.00324269554887448\\
0.1	0.00044978582997003\\
0.01	1.41436745705946e-05\\
0.001	8.13623820180849e-07\\
0.0001	3.92483265793044e-07\\
};
\addlegendentry{S1}

\addplot [color=mycolor2, mark=x, mark options={solid, mycolor2}]
  table[row sep=crcr]{%
1	0.0182681468624896\\
0.1	0.00136108223347771\\
0.01	2.17060979968132e-05\\
0.001	2.34555749126122e-06\\
0.0001	2.41969979541761e-06\\
};
\addlegendentry{S2}

\addplot [color=mycolor3, mark=x, mark options={solid, mycolor3}]
  table[row sep=crcr]{%
1	0.0303329124831197\\
0.1	0.00438651421403131\\
0.01	0.000121460516823787\\
0.001	5.85750255854074e-06\\
0.0001	3.38788413936707e-06\\
};
\addlegendentry{S3}

\addplot [color=mycolor4]
  table[row sep=crcr]{%
1	1\\
0.1	0.1\\
0.01	0.01\\
0.001	0.001\\
0.0001	0.0001\\
};
\addlegendentry{linear comp.}

\end{axis}
\end{tikzpicture}%